\documentclass{amsart}
\usepackage[margin=1in]{geometry}
\usepackage{amsmath}
\usepackage{amssymb}
\usepackage{amsthm}
\usepackage{amscd}
\usepackage{enumerate}
\newcounter{pcounter}

\newcommand{\PP}{\Bbb P}

\newcommand{\ZZ}{\Bbb Z}
\newcommand{\RR}{\Bbb R}
\newcommand{\TT}{\Bbb T}
\newcommand{\NN}{\Bbb N}
\newcommand{\QQ}{\Bbb Q}
\newcommand{\CC}{\Bbb C}
\newcommand{\ip}[1]{\langle #1 \rangle}

\newcommand{\widetidle}{\widetilde}

\newcommand{\varespilon}{\varepsilon}

\newcommand{\varpesilon}{\varepsilon}

\newcommand{\actson}{\curvearrowright}

\newtheorem{question}{Question}
\newtheorem{conjecture}[question]{Conjecture}
\newtheorem{?}{Question}
\newtheorem{theorem}{Theorem}
\newtheorem{definition}[theorem]{Definition}
\newtheorem{proposition}[theorem]{Proposition}
\newtheorem{cor}[theorem]{Corollary}
\newtheorem{lemma}[theorem]{Lemma}

\newcommand{\FF}{\Bbb F}

\DeclareMathOperator{\ev}{ev}
\DeclareMathOperator{\spec}{spec}

\DeclareMathOperator{\vol}{vol}
\DeclareMathOperator{\vr}{vr}

\DeclareMathOperator{\Sym}{Sym}

\DeclareMathOperator{\Map}{Map}
\DeclareMathOperator{\im}{im}
\DeclareMathOperator{\Fix}{Fix}
\DeclareMathOperator{\id}{Id}

\DeclareMathOperator{\tr}{tr}
\DeclareMathOperator{\supp}{supp}

\DeclareMathOperator{\Proj}{Proj}

\DeclareMathOperator{\Rea}{Re}

\DeclareMathOperator{\Det}{det}

\DeclareMathOperator{\Hom}{Hom}

\DeclareMathOperator{\Tr}{Tr}

\DeclareMathOperator{\mdim}{mdim}

\DeclareMathOperator{\AP}{AP}
\DeclareMathOperator{\Prob}{Prob}
\DeclareMathOperator{\Ball}{Ball}

\numberwithin{theorem}{section}

\begin{document}

\title{Fuglede-Kadison Determinants and Sofic Entropy}      
\author{Ben Hayes}\thanks{The author is grateful for support from NSF Grants DMS-1161411 and DMS-0900776. This work is partially supported by the ERC starting grant 257110 ``RaWG."}
\address{UCLA Math Sciences Building\\
         Los Angeles,CA 90095-1555}
\email{brh6@ucla.edu}
\date{\today}

\begin{abstract} We relate Fuglede-Kadison determinants to entropy of finitely-presented algebraic actions in essentially complete generality. We show that if $f\in M_{m,n}(\ZZ(\Gamma))$ is injective as a left multiplication operator on $\ell^{2}(\Gamma)^{\oplus n},$ then  the topological entropy of the action of $\Gamma$ on the dual of  $\ZZ(\Gamma)^{\oplus n}/\ZZ(\Gamma)^{\oplus m}f$ is at most the logarithm of the positive Fuglede-Kadison determinant of $f,$ with equality if $m=n.$   We also prove that when $m=n$  the measure-theoretic entropy of the action of $\Gamma$ on the dual of $\ZZ(\Gamma)^{\oplus n}/\ZZ(\Gamma)^{\oplus n}f$ is the logarithm of the Fuglede-Kadison determinant of $f.$ This work completely settles the connection between entropy of principal algebraic actions and Fuglede-Kadison determinants in the generality in which dynamical entropy is defined. Our main Theorem  partially generalizes results of Li-Thom from amenable groups to sofic groups. Moreover, we show that the obvious full generalization of the Li-Thom theorem for amenable groups is false for general sofic groups. Lastly, we undertake a study of when the Yuzvinski\v\i\  addition formula fails for a non-amenable sofic group $\Gamma$, showing it always fails if $\Gamma$ contains a nonabelian free group, and relating it to the possible values of $L^{2}$-torsion in general.

\end{abstract}

\maketitle
\tableofcontents

\section{Introduction}

	 The goal of this paper is to relate Fuglede-Kadison determinants and entropy of finitely-presented algebraic actions in the largest possible generality.  Let $\Gamma$ be a countable discrete group, and let $A$ be a $\ZZ(\Gamma)$-module. We have a natural action of $\Gamma$  on $\widehat{A},$ the Pontryagin dual of $A$, (i.e. the group of all continuous homomorphism from $A$ into $\TT=\RR/\ZZ$) given by
\[(g\chi)(a)=\chi(g^{-1}a),\mbox{ $g\in\Gamma,\chi\in\widehat{A},a\in A.$}\]
This  action on $\widehat{A}$ by automorphisms is called an algebraic action. We are typically interested in forgetting the algebraic structure of $\widehat{A}.$ That is, we wish to think of $\Gamma\actson \widehat{A}$ as either an action of $\Gamma$ on a compact metrizable space by homeomorphisms, or an action of $\Gamma$ on a probability measure space (using the Haar measure on $\widehat{A}$, which we denote by $m_{\widehat{A}}$) by measure-preserving transformations.

It is trivial by Pontryagin duality that any invariant of $\Gamma\actson \widehat{A}$  as either a probability measure-preserving action or an action by homeomorphisms comes from a  $\ZZ(\Gamma)$-module invariant. However, it is unclear what $\ZZ(\Gamma)$-module  invariants arise in this manner, i.e. which $\ZZ(\Gamma)$ module-invariants only depend upon the action $\Gamma\actson \widehat{A}$ when we view this as either a probability measure-preserving action or an action by homeomorphisms. It turns out that most invariants which do only depend upon the topological or measure-theoretic structure of $\Gamma\actson \widehat{A}$ are defined via functional analysis.  For instance, ergodicity, mixing, expansiveness can all be expressed in terms of functional analytic objects associated to $A$ (see \cite{Schmidt} Lemma 1.2, Theorem 1.6, \cite{ChungLi} Theorem 3.1).
	
	The case of finitely presented $\ZZ(\Gamma)$-modules is particularly enlightening. Given $f\in M_{m,n}(\CC(\Gamma)), let f_{pq}=\sum_{g\in\Gamma}\widehat{f_{pq}}(g)g$ for$1\leq p\leq m,1\leq q\leq n$. Define $r(f):\ell^{2}(\Gamma)^{\oplus m}\to \ell^{2}(\Gamma)^{\oplus n},$  $\lambda(f)\colon \ell^{2}(\Gamma)^{\oplus n}\to \ell^{2}(\Gamma)^{m}$ by
\[(r(f)\xi)(l)(h)=\sum_{1\leq s\leq m}\sum_{g\in\Gamma}\xi(s)(hg^{-1})\widehat{f_{sl}}(g),\mbox{ for $1\leq l\leq n$,}\]
\[(\lambda(f)\xi)(l)(h)=\sum_{1\leq s\leq n}\sum_{g\in\Gamma}\widehat{f_{ls}}(g)\xi(s)(g^{-1}h),\mbox{  for $1\leq l\leq m$}.\]
Every finitely-presented $\ZZ(\Gamma)$-module is of the form $\ZZ(\Gamma)^{\oplus n}/r(f)(\ZZ(\Gamma)^{\oplus m}).$  For $f\in M_{m,n}(\ZZ(\Gamma)),$ we will use $X_{f}$ for the Pontryagin dual of $\ZZ(\Gamma)^{\oplus n}/r(f)(\ZZ(\Gamma)^{\oplus m}).$  In this case, the duality between functional analytic properties of $\ZZ(\Gamma)$-modules and the dynamics of algebraic actions translates to a duality between dynamics of $\Gamma\actson X_{f}$ and operator theoretic properties or $\lambda(f).$ Regarded as a representation of $\Gamma$ we call  $\lambda$ the left regular representation.
	
		When $m=n=1,$ the action $\Gamma\actson X_{f}$ is called a \emph{principal} algebraic action. The entropy of the action $G\actson X_{f}$ has been well-studied, particularly in the principal case. Entropy is an important numerical invariant of actions defined for $\Gamma=\ZZ$ by Kolmogorov, Sina\v\i, Adler-Konheim-MacAndrew (see \cite{Kol58},\cite{Sin59},\cite{AKM}) and for amenable groups by Kieffer and Ornstein-Weiss (see \cite{Kieff},\cite{OrnWeiss}). For a probability measure-preserving action $\Gamma\actson (X,\mu)$ with $\Gamma$ amenable, we use $h_{\mu}(X,\Gamma)$ for the entropy.
		
		For $\Gamma=\ZZ^{d}$ we may identify, in a natural way, $\ZZ(\Gamma)$ as the Laurent polynomial ring $\ZZ[u_{1}^{\pm 1},\dots,u_{d}^{\pm 1}].$ Under this identification, it is known that the entropy of $\Gamma\actson X_{f}$ is
	\begin{equation}\label{E:Mahlermeasure}
	\int_{\TT^{d}}\log|f(e^{2\pi i\theta})|\,d\theta.
	\end{equation}
This was shown for $\Gamma=\ZZ$ by Yuzvinks\v\i\ in \cite{Yuz} and $\Gamma=\ZZ^{d}$ by Lind-Schmidt-Ward in \cite{LindSchmidt2} (both of these results in fact include a complete calculation of entropy in the case of a finitely-presented $\ZZ(\Gamma)$ module for $\Gamma=\ZZ,\ZZ^{d}$). This integral is known as the logarithmic Mahler measure of $f$ and is important in number theory.

	It was Deninger in \cite{Den} who first realized that (\ref{E:Mahlermeasure}) has a natural generalization to noncommutative $\Gamma$ via Fuglede-Kadison determinants. The Fuglede-Kadison determinant of $f\in M_{n}(\CC(\Gamma))$, denoted $\Det_{L(\Gamma)}^{+}(f),$  is a natural generalization of the usual determinant in finite-dimensional linear algebra. Here we are defining a determinant of $\lambda(f)$ and a priori it is not clear how one would do this, as $\lambda(f)$ operators on an infinite-dimensional space. The crucial analytic object that makes this possible is the group von Neumann algebra.  The group von Neumann algebra, which we denote by $L(\Gamma)$, is a functional analytic object associated to $\Gamma$ which ``encodes'' the structure of the left regular representation. Deninger pointed out the possibility that
	\[h_{m_{X_{f}}}(X_{f},\Gamma)=\Det_{L(\Gamma)}^{+}(f)\]
and established that this is the case when  $f$ is positive (i.e. $\lambda(f)$ is a positive operator), invertible in $\ell^{1}(\Gamma),$ and $\Gamma$ is of polynomial growth. Then Deninger-Schmidt in \cite{DenSchmidt} show this equality in the principal case was true when $\Gamma$ is amenable, residually finite and when $f$ is invertible in $\ell^{1}(\Gamma).$ Li in \cite{Li2} proved this equality when $\Gamma$ is amenable and $\lambda(f)$ is invertible (equivalently $f$ is invertible in $L(\Gamma)$). It was only recently that the connection between entropy and determinants was completely settled in the amenable case by Li-Thom in \cite{LiThom}.  Li-Thom equated the entropy of $\Gamma\actson X_{f}$ to $\Det_{L(\Gamma)}^{+}(f)$ when $\Gamma$ is amenable and $\lambda(f)$ is injective. We remark that the results of Li-Thom are complete for the case of amenable $\Gamma$, as  Chung-Li showed (see \cite{ChungLi} Theorem 4.11) that if $\lambda(f)$ is not injective, and $\Gamma$ is amenable, then the entropy of $\Gamma\actson X_{f}$ is infinite and thus cannot be equal to $\det_{L(\Gamma)}^{+}(f).$

%
	
	  Seminal work of Bowen  in \cite{Bow} extended the definition of measure entropy for actions of amenable groups to the much larger class of sofic groups assuming the action has a finite generating partition. Kerr-Li in \cite{KLi} removed this generation assumption and also defined topological entropy for actions of sofic groups.  Roughly, soficity is the assumption of a sequence of ``asymptotic homomorphisms'' $\Gamma\to S_{d_{i}},$ where $S_{n}$ is the symmetric group on $n$ letters, which when regarded as ``almost actions'' on $\{1,\dots,d_{i}\}$ are ``almost free.'' Such a sequence is called a \emph{sofic approximation}. The class of sofic groups contains all amenable groups, all residually finite groups, all linear groups and is closed under direct unions and free products with amalgamation over amenable subgroups (see \cite{ESZ2},\cite{DKP},\cite{LPaun},\cite{PoppArg}). Residually sofic groups are also sofic. Thus sofic groups are a significantly larger class of groups than amenable groups. We remark that there is no known example of a nonsofic group. We use $h_{\Sigma,\mu}(X,\Gamma)$ for the entropy of a measure-preserving $\Gamma\actson (X,\mu)$ on a standard probability space $(X,\mu)$ with respect to a sofic approximation $\Sigma$ of $\Gamma$ (see Section \ref{S:measureentropy}). We use $h_{\Sigma}(X,\Gamma)$ for the entropy of an action $\Gamma\actson X$ by homeomorphisms of a compact metric space $X$ with respect to a sofic approximation $\Sigma$ of $\Gamma$ (see Section \ref{S:mainreduction}).

	Bowen in \cite{BowenEntropy} proved the equality between entropy and Fuglede-Kadison determinants when $f$ is invertible in $\ell^{1}(\Gamma)$ and $\Gamma$ is residually finite. Kerr-Li in \cite{KLi} then proved this equality when $\Gamma$ is residually finite and $f$ is invertible in the full $C^{*}$-algebra of $\Gamma.$ Bowen-Li in \cite{BowenLi} also proved this equality when $\Gamma$ is residually finite and $f$ is a Laplacian operator.  Notice that all of these results are only valid when the group is residually finite and all of them require invertibility assumptions on $\lambda(f),$  or very specific structure of $f.$ Thus they leave open the relationship between Fuglede-Kadison determinants and entropy for the general case of principal algebraic actions of sofic groups.

		 In this paper, we completely settle the connection between Fuglede-Kadison determinants and entropy for principal algebraic actions under the utmost minimal hypotheses. This connection is a consequence of the following result (the principal case being $m=n=1$), which is the main Theorem of the paper.

\begin{theorem}\label{T:FKD} Let $\Gamma$ be a countable discrete sofic group with sofic approximation $\Sigma.$  Fix a $f\in M_{m,n}(\ZZ(\Gamma)).$

(i): The topological entropy of $\Gamma\actson X_{f}$ (with respect to $\Sigma$) is finite if and only if $\lambda(f)$ is injective as an operator on $\ell^{2}(\Gamma)^{\oplus n}.$

(ii): If  $m=n,$ and $\lambda(f)$ is injective as an operator on $\ell^{2}(\Gamma)^{\oplus n},$ then
\[h_{\Sigma}(X_{f},\Gamma)=h_{\Sigma,m_{X_{f}}}(X_{f},\Gamma)= \log\Det^{+}_{L(\Gamma)}(f).\]

(iii): If $m\ne n$ and $\lambda(f)$ is injective, then
\[h_{\Sigma,m_{X_{f}}}(X_{f},\Gamma)\leq h_{\Sigma}(X_{f},\Gamma)\leq \log \Det^{+}_{L(\Gamma)}(f).\]
 \end{theorem}

	 Since  $\Det_{L(\Gamma)}^{+}(f)$ is manifestly less than $\infty,$ part  $(i)$ implies that if $f\in M_{m,n}(\ZZ(\Gamma))$ and $\lambda(f)$ is not injective, then the topological entropy of $\Gamma\actson X_{f}$ for $f\in\ZZ(\Gamma)$ cannot be equal to $\log\Det_{L(\Gamma)}^{+}(f).$ Combining this observation with part $(i)$ of the Theorem  settles the connection between entropy of principal algebraic actions and Fuglede-Kadison determinants for actions of sofic groups (i.e. the class of groups for which dynamical entropy is defined). Part (i) is a trivial consequence of the main results of our results in \cite{Me4} (see Theorem \ref{T:whenisitinfinite} of this paper). Thus we focus on (ii),(iii) for most of the paper.

	Let us mention why the above theorem is essentially optimal even in the nonprincipal case. First,  as previously mentioned, $\Det_{L(\Gamma)}^{+}(f)$ is always less than $\infty$, so by (i) there cannot be any connection to Fuglede-Kadison determinants and topological entropy if $\lambda(f)$ is not injective as a left multiplication operator on $\ell^{2}(\Gamma)^{\oplus n}.$  Secondly, $h_{\Sigma}(X_{f},\Gamma)\ne\log\Det_{L(\Gamma)}^{+}(f)$  in general when $m\ne n$ even if $\lambda(f)$ is injective. For example, let $n\in\ZZ\setminus\{0\},$ and $\alpha\in\ZZ(\Gamma)$ and set
\[A=\begin{bmatrix}
\alpha\\
n
\end{bmatrix}\in M_{2,1}(\ZZ(\Gamma)).\]
Then $X_{A}$ is a subgroup of $(\ZZ/n\ZZ)^{\Gamma}$ and hence the topological entropy of $\Gamma\actson X_{A}$ is at most $\log n$ (the same is true for the measure-theoretic entropy). However, a simple calculation shows that
\[\log\Det_{L(\Gamma)}^{+}(A)=\frac{1}{2}\log\Det_{L(\Gamma)}(\alpha^{*}\alpha+n^{2})\]
which is strictly bigger than $\log(n),$ if $\alpha\ne 0.$  Thus, the inequality in (iii) is not always an equality.

	Lastly, sofic groups are the largest class of groups $\Gamma$ for which $\Det_{L(\Gamma)}^{+}(f)\geq 1$ for all $f\in M_{m,n}(\ZZ(\Gamma)).$  The statement that $\Det_{L(\Gamma)}^{+}(f)\geq 1$ for all $f\in M_{m,n}(\ZZ(\Gamma))$ is called the determinant conjecture. The fact that $\det_{L(\Gamma)}^{+}(f)\geq 1$ is key in the proof of Theorem \ref{T:FKD} as well as the Li-Thom Theorem. Let us suppose, for the sake of argument, that one develops a good definition of entropy for non-sofic groups. Such a definition would likely have similar nonnegative properties as sofic entropy. In particular if there is a connection between entropy and Fuglede-Kadison determinants analogous to Theorem \ref{T:FKD} for nonsofic groups, then the determinant conjecture should be true. As the proof of Theorem \ref{T:FKD} as well as \cite{LiThom} rely on the fact that $\Det^{+}_{L(\Gamma)}(f)\geq 1$ for all $f\in M_{m,n}(\ZZ(\Gamma))$, it is likely that any hypothetical version of entropy for nonsofic groups would rely on knowing the determinant conjecture and would not be a likely route to prove the determinant conjecture. Since it is unclear how to prove this conjecture for \emph{any} group which is not known to be sofic, it seems unlikely to generalize Theorem \ref{T:FKD} to any potential definition of entropy for a nonsofic group.

	In the amenable case the results of \cite{LiThom} are more general as  Li-Thom relate the $L^{2}$-torsion of $A,$ when it is defined, to the entropy of $\Gamma\actson \widehat{A}.$ We use $\rho^{(2)}(A,\Gamma)$ for the $L^{2}$-torsion of a $\ZZ(\Gamma)$-module $A$. The $L^{2}$-torsion is one of several invariants which fall under the name $L^{2}$-invariants, all of which are functional analytic invariants of $\ZZ(\Gamma)$-modules defined via the group von Neumann algebra. See \cite{Luck} for a good introduction to $L^{2}$-invariants. Part (ii) of Theorem \ref{T:FKD} is the ``base case'' for sofic groups of the results of Li-Thom. This begs the questions of whether our results can be further generalized to show that the entropy of an algebraic action is the $L^{2}$-torsion of the dual module.  The next proposition (a simple application of known results on $L^{2}$-torsion and basic facts about sofic entropy) shows that a generalization of Theorem \ref{T:FKD} in this direction is not possible for a totally general sofic group.

\begin{proposition}\label{P:torsionintro} Let $\Gamma$ be a cocompact lattice in $SO(n,1)$ with $n$ odd. Then:

(i): $\Gamma$ is a sofic group,

(ii): the $L^{2}$-torsion of the trivial $\ZZ(\Gamma)$-module $\ZZ$ is defined,

(iii): for every sofic approximation $\Sigma$ of $\Gamma$ one has
\[h_{\Sigma}(\TT,\Gamma)\ne \rho^{(2)}(\ZZ,\Gamma),\]
\[h_{\Sigma,m_{\TT}}(\TT,\Gamma)\ne \rho^{(2)}(\ZZ,\Gamma).\]

\end{proposition}

In the case $n$ is congruent to $1$ modulo $4,$ we can in fact say  that the measure-theoretic entropy of $\Gamma\actson \TT$ with respect to any \emph{random} sofic approximation  of $\Gamma$ is not $\rho^{2}(\ZZ,\Gamma)$ (for a precise statement see Proposition \ref{P:L2torsioncounterexample}). It may still be possible to connect torsion to entropy with respect to a random sofic approximation when $n$ is congruent to $3$ modulo $4,$ but these remarks show that such a connection is completely impossible for $n$ congruent to $1$ modulo $4.$ Thus there is no possible generalization of the Li-Thom theorem (using sofic entropy) for cocompact lattices in $SO(n,1)$ for $n$ congruent to $1$ modulo $.4$

	After establishing Theorem \ref{T:FKD} in the amenable case, the remaining piece Li-Thom use to complete the connection  between entropy and $L^{2}$-torsion for amenable groups is the  Yuzvinski\v\i\ addition formula, which says that entropy is additive under exact sequences of algebraic actions. We show that the Yuzvinski\v\i\ addition formula is false for many nonamenable groups in this paper. Our main result in this direction is the following.

\begin{theorem}\label{T:YuvFail} Let $\Gamma$ be a countable discrete sofic group with sofic approximation $\Sigma.$ Suppose that either $\Gamma$ contains a nonabelian free group, or that $\Gamma$ contains a subgroup with defined and nonzero $L^{2}$-torsion. Then Yuzvinski\v\i's addition formula fails for $\Gamma.$ That is, there is an exact sequence
\[\begin{CD}
0 @>>> A @>>> B @>>> C @>>> 0,
\end{CD}\]
of countable $\ZZ(\Gamma)$-modules so that
\[h_{\Sigma}(\widehat{B},\Gamma)\ne h_{\Sigma}(\widehat{A},\Gamma)+h_{\Sigma}(\widehat{C},\Gamma).\]
Further, we can choose $A,B,C$ to be finitely presented.
\end{theorem}
The main ingredients of the proof of the above Theorem are the arguments of Li-Thom as well as a standard counterexample of the theory due to Ornstein-Weiss. The only current version of nonamenable entropy for which there is some hope of having a  Yuzvinski\v\i\ addition formula is the $f$-invariant entropy defined for actions of free groups by Bowen in \cite{Bowenfinvariant}. For example, a  Yuzvinski\v\i\ addition formula is known for $f$-invariant entropy for actions on totally disconnected abelian groups by \cite{BowenGun}. The $f$-invariant can be regarded as sofic entropy with respect to a \emph{random} sofic approximation and has many properties that general sofic entropy does not: it satisfies a Rokhlin formula (see \cite{BowenGun}), a subgroup formula (see \cite{SewardSubgroup}) and an ergodic decomposition formula (see \cite{SewardFree}). Motivated by the properties that $f$-invariant entropy enjoys, and the possibility of a Yuzvinsk\v\i\ addition formula, we prove a version of Theorem \ref{T:FKD} for  Bowen's $f$-invariant entropy. We use $\FF_{r}$ for the free group on $r$ letters.

\begin{theorem} Let $h\in M_{m,n}(\ZZ(\FF_{r}))$ and suppose $\lambda(h)$ is injective. Then
\[f_{m_{X_{h}}}(X_{h},\FF_{r})\leq \log \Det^{+}_{L(\FF_{r})}(h),\]
with equality if $m=n.$
\end{theorem}
We remark that it follows automatically from the preceding theorem and the techniques of Li-Thom that if one proves a totally general Yuzvinski\v\i\ addition formula for $f$-invariant entropy for actions of free groups, then automatically one equates the $f$-invariant entropy of an algebraic action with the $L^{2}$-torsion of the dual module.

 	Combining with our previous work in \cite{Me4}, as well as the techniques given in \cite{LiLiang}, we have an application related to metric mean dimension of actions. For the definition of metric mean dimension  of an action $\Gamma\actson X$ on a compact metrizable space $X$,  denoted $\mdim_{\Sigma,M}(X,\Gamma)$, see \cite{Li}. It is clear from the definition that if $h_{\Sigma}(\widehat{A},\Gamma)<\infty,$ then $\mdim_{\Sigma,M}(X,\Gamma)=0.$ It is an open problem as to whether or not every action with zero metric mean dimension can be ``built" out of actions with finite entropy (to be precise, it is an open problem if every it is an open problem if every action with zero metric mean dimension is an inverse limit of actions with finite entropy). We use our results to show that for algebraic actions coming from finitely presented $\ZZ(\Gamma)$-modules zero metric mean dimension is equivalent to finite topological entropy.

\begin{theorem}\label{T:whenitisfintie} Let $\Gamma$ be a countable discrete sofic group with sofic approximation $\Sigma.$  Let $A$ be a finitely presented $\ZZ(\Gamma)$-module. Then $h_{\Sigma}(\widehat{A},\Gamma)<\infty$ if and only if $\mdim_{\Sigma,M}(\widehat{A},\Gamma)=0.$
\end{theorem}

In many cases, we can replace the assumption in the preceding theorem  that $A$ is finitely presented with $A$ being finitely generated. This is related to whether or not $\Gamma$ satisfies the Atiyah conjecture (see Section \ref{S:apply}).

Let us briefly summarize the differences between our proof of Theorem \ref{T:FKD} and previous proofs of special cases of Theorem \ref{T:FKD}. To simplify the discussion, we stick to the principal case, so fix a $f\in\ZZ(\Gamma).$  We only summarize the proof of
\[h_{\Sigma}(X_{f},\Gamma)\geq \Det^{+}_{L(\Gamma)}(f),\]
as the upper bound is simpler. Given a sofic approximation $\sigma_{i}\colon\Gamma\to S_{d_{i}}$ extend $\sigma_{i}$ to a map
$\sigma_{i}\colon \CC(\Gamma)\to M_{d_{i}}(\CC)$ by
\[\sigma_{i}(\alpha)=\sum_{g\in\Gamma}\widehat{\alpha}(g)\sigma_{i}(g),\mbox{ if $\alpha=\sum_{g\in\Gamma}\widehat{\alpha}(g)g$}.\]
In order to make sense of the above sum, we view $S_{d_{i}}\subseteq M_{d_{i}}(\CC)$ as permutation matrices. Note that $\sigma_{i}(\ZZ(\Gamma))\subseteq M_{d_{i}}(\ZZ).$ Thus we can view $\sigma_{i}(f)$ as a homomorphism $\TT^{d_{i}}\to \TT^{d_{i}}.$

	Previous proofs consisted of a two step process. First, one bounds the entropy from below by the exponential growth rate of the size of the kernel of $\sigma_{i}(f)$ as a homomorphism $\TT^{d_{i}}\to \TT^{d_{i}}.$ This either requires knowing that $\sigma_{i}(f)\in GL_{d_{i}}(\RR)$  or having ``good control'' over the kernel of $\sigma_{i}(f)$ as an operator $\CC^{d_{i}}\to \CC^{d_{i}}.$ This ``good control'' of the kernel in previous results could only be achieved when $f$ is either a very specific operator (such as the Laplace operators consider by Bowen-Li) or when $\Gamma$ is amenable. In this step, both the invertibility hypothesis on $f$ and the fact that $\Gamma$ was residually finite played an important role in the nonamenable case.  For instance, the fact that $\Gamma$ is residually finite allowed one to take $\sigma_{i}$ to be an honest homomorphism.
	
	The second step is to prove that
\begin{equation}\label{E:introunlikely}
\Det^{+}(\sigma_{i}(f))^{1/d_{i}}\to \Det_{L(\Gamma)}^{+}(f),
\end{equation}
we call this the \emph{determinant approximation}. Here $\det^{+}(A)$ for $A\in M_{n}(\CC)$ is defined to be the product of all nonzero singular values of $A$ (with repetition). All previous results on Fuglede-Kadison determinants and entropy use (\ref{E:introunlikely}), however we suspect that this approximation is false in general. We will discuss at the end of the introduction why we believe that  previous proofs of special cases of (\ref{E:introunlikely}) rely on various heavily simplifying assumptions, and do not indicate or even suggest that the general result is true.

	Because of the possibility that the determinant approximation is false, we must avoid determinant approximations and so we need genuinely new techniques to prove Theorem \ref{T:FKD}. Instead of bounding the entropy from below by the exponential growth rate of the size of the kernel, our approach is to bound the entropy from below by the exponential growth rate of the size of the ``approximate kernel'' of $\sigma_{i}(f)$ as a homomorphism $\TT^{d_{i}}\to \TT^{d_{i}}.$  This is achieved by a simple compactness argument, and in this case the lower bound becomes an equality. No previous proof of the relationship between determinants and entropy used this method. The approach using the ``approximate kernel'' has two main advantages. First, to compute the entropy, we are allowed to replace the``approximate kernel'' of $\sigma_{i}(f)$ with the ``approximate kernel'' of any operator ``close'' to $\sigma_{i}(f).$ Using that $\lambda(f)$ is injective we show that we can choose such a perturbation of $\sigma_{i}(f)$ to be in $GL_{d_{i}}(\RR),$ which simplifies many of the technicalities involved. Most importantly, counting the size of the ``approximate kernel'' instead of the actual kernel has the desirable effect of increasing previous lower bounds on entropy. It turns out that these lower bounds are improved enough to completely avoid approximations such as (\ref{E:introunlikely}). This is the first proof of equality between Fuglede-Kadison determinants and entropy which does not use the determinant approximation. We remark that perturbations of $\sigma_{i}(f)$ were used in \cite{Li2} (for slightly different purposes) in the case where $\Gamma$ is amenable and $\lambda(f)$ is invertible. However, \cite{Li2} does not use our ``approximate kernel'' approach and so still has to use (\ref{E:introunlikely}).

	As we mentioned before, we believe that previous results establishing special cases of the determinant approximation are too specific to indicate its validity in general. Most of these results require an invertibility hypothesis on $f.$ These  invertibility hypotheses  imply a uniform lower bound on the smallest singular value of $\sigma_{i}(f)$ and make (\ref{E:introunlikely}) a simple consequence of weak$^{*}$ convergence of spectral measures. In the presence of singular values close to zero, weak$^{*}$-convergence is not strong enough to conclude (\ref{E:introunlikely}). As all of these results implicitly assume an absence of small singular values they do not indicate how to approach the determinant approximation when $\lambda(f)$ is injective and not invertible, since $\sigma_{i}(f)$ will always have singular values close to zero in this case.

	There are  essentially only three special cases where (\ref{E:introunlikely}) has been established without implicitly assuming any lower bounds on singular values, these cases are the following (listed in chronological order):

(a): For a residually finite group, it is natural to consider the sofic approximation given by the action on a chain of normal subgroups. We call this the \emph{residually finite sofic approximation}. This is one of the nicest and most natural sofic approximations, as the maps are honest homomorphisms.  For this sofic approximation, the \emph{only} case where (\ref{E:introunlikely}) is known for every $f\in\ZZ(\Gamma)$ is when $\Gamma$ is virtually cyclic, (see \cite{Luck} Lemma 13.53).

(b): If $f$ is a Laplacian operator (\ref{E:introunlikely}) is a consequence of a result of Lyons in \cite{Lyons}, as noted in Section 3 of \cite{BowenLi}.

 (c): If $\Gamma$ is amenable and the sofic approximation is by F\o lner sequences $(\ref{E:introunlikely})$ is proved in \cite{LiThom} using a variant of the Ornstein-Weiss Lemma (for $\Gamma=\ZZ$ this is a classical result of Szeg\H{o} in \cite{Szego}).

    For $(a)$ one reduces to $\Gamma=\ZZ$ and uses nontrivial number-theoretic facts to show that there are ``not many small singular values''. These results are very dependent on the group being the integers. So we feel that the group  is far too restricted in this case and the tools required are far too strong to indicate any belief in the determinant approximation.  Case $(b)$  is too specific to the structure of $f,$ since for a general $f\in \ZZ(\Gamma)$ there will be no connection between determinants and graph theory. Case $(c)$ is too specific to the structure of the group, as there is no analogue of the Ornstein-Weiss Lemma beyond amenable groups.

	On the other hand, we have a good reasons to disbelieve the general determinant approximation. One reason for our disbelief is the fact that  the determinant approximation is extremely difficult even when the sofic approximation is very nice. The residually finite sofic approximation may be the most natural sofic approximation and the determinant approximation is unknown in this case even when $\Gamma=\ZZ^{2}.$ Our second reason is that there are ``near counterexamples" to the determinant approximation. For instance,  it is known that (\ref{E:introunlikely}) fails for the residually finite sofic approximations if $f\in\CC(\ZZ)$ (see the remarks after Lemma 13.53 in \cite{Luck}). This failure indicates that these number theoretic techniques are necessary in the integer  case. As it is absolutely unknown how to generalize these techniques, we do not anticipate being able to generalize to the case of a general residually finite group. Another ``near counterexample'' is discussed  in  \cite{Grab} (see Remark 10) where it is remarked that one can find counterexamples to the determinant approximation for Laplacian operators if one replaces sofic approximations with graph convergence in the Benjamimi-Schramm sense. In short, it is completely unclear how to prove a determinant approximation for $\Gamma$ sofic and $\lambda(f)$ injective. Such approximations are very difficult without implicitly assuming uniform lower bounds on singular values, nor are they likely to be true in general.  Given these difficulties, we strongly believe that avoiding these approximations is a useful technique for studying entropy of algebraic actions of general sofic groups.

\textbf{Acknowledgments.} I would like to thank Dimitri Shlyakhtenko for his continued support and helpful advice on the problem. I would like to thank Lewis Bowen, Phu Chung, Hanfeng Li, Dave Penneys, Brandon Seward and Andreas Thom for comments on  previous versions of this paper, as well as interesting discussions related to this problem. I would like to thank Yoann Dabrowski for interesting discussions which led to the abstract concentration arguments in Lemma \ref{L:AbstractConc}, which conceptually simplified the argument. I would like to thank the anonymous referee for their numerous comments, which vastly improved the paper. Part of this work was inspired by discussions at the Arbeitsgemeinschaft on Sofic Entropy at Oberwolfach in October 2013.

\section{Preliminaries}

\subsection{Notation and Terminology}

	We will use $e$ for the identity  element of a group, unless the group is assumed abelian, in which case we will use $0.$ Abelian group operations will be written additively, unless otherwise stated. In particular, we use $\TT=\RR/\ZZ$ with group operations written additively, and do not view it is as the unit circle in the complex plane. If $\mathcal{H},\mathcal{K}$ are Hilbert spaces, we use $B(\mathcal{H},\mathcal{K})$ for the space of bounded linear operators from  $\mathcal{H}\to\mathcal{K}.$ We often use $B(\mathcal{H})$ instead of $B(\mathcal{H},\mathcal{H}).$

	We will use standard functional calculus notation for normal operators. Functional calculus will most often be used in the finite dimensional case, for which we note the following: if $\phi\colon \CC\to \CC$ is Borel, $\mathcal{H}$ is a finite-dimensional Hilbert space, and $T\in B(\mathcal{H})$ is normal, then
\[\phi(T)=\sum_{\lambda\in \spec(T)}\phi(\lambda)\Proj_{\ker(T-\lambda I)}.\]
Here $\Proj_{\mathcal{K}}$ denotes the orthogonal projection onto the subspace $\mathcal{K},$ and $\spec(T)$ denotes the spectrum of $T.$ For any operator on a Hilbert space (normal or not), we use $|T|=(T^{*}T)^{1/2}.$ For $x\in \RR^{n},$ we will typically use
\[\|x\|_{2}^{2}=\frac{1}{n}\sum_{j=1}^{n}|x_{j}|^{2},\]
we shall usually not need to consider $\|x\|_{\ell^{2}(n)}$ (in fact we will only need $\|\cdot\|_{\ell^{2}(n)}$ in Section \ref{S:measureentropy}).
We use $\tr_{n}\colon M_{n}(\CC)\to \CC$ for $\frac{1}{n}\Tr$ where
\[\Tr(A)=\sum_{j=1}^{n}A_{jj}.\]
 This will usually be more natural to use than $\Tr.$ We will often drop the subscript $n$ if it is clear from context. We let  $\Tr\otimes \tr_{n}\colon M_{m}(M_{n}(\CC))\to \CC$ be given by
\[\Tr\otimes \tr_{n}(A)=\sum_{j=1}^{m}\tr_{n}(A_{jj}).\]

	A \emph{pseudometric} on a set $X$ is a function $d\colon X\times X\to [0,\infty)$ satisfying symmetry and the triangle inequality, but we might have that $d(x,y)=0$ and $x\ne y.$ A set $X$ with a pseudometric $d$ will be called a \emph{pseudometric space}. If $(X,d)$ is a pseudometric space and $A,B\subseteq X,$ and $\varepsilon\geq 0,$ we say that $A$ is $\varepsilon$-contained in $B,$ and write $A\subseteq_{\varespilon}B$ if for all $a\in A,$ there is a $b\in B$ so that $d(a,b)\leq \varepsilon.$ We say $A\subseteq X$ is $\varespilon$-dense if $X\subseteq_{\varepsilon}A.$  We use $S_{\varepsilon}(X,d)$ for the smallest cardinality of an $\varepsilon$-dense subset of $X.$ We say that $A\subseteq X$ is $\varepsilon$-separated if for all $a\ne b$ in $A$ we have $d(a,b)>\varepsilon.$ We use $N_{\varepsilon}(X,d)$ for the largest cardinality of a $\varespilon$-separated subset of $X.$ We always have the following inequalities:
\begin{equation}\label{E:sepspan}
S_{\varepsilon}(X,d)\leq N_{\varepsilon}(X,d)\leq S_{\varespilon/2}(X,d).
\end{equation}
 If $\delta,\varespilon\geq 0,$ and $A\subseteq_{\delta}B$ we have
\begin{equation}\label{E:sepspanfuzz}
N_{2(\varepsilon+\delta)}(A,d)\leq S_{\varepsilon}(B,d).
\end{equation}
Lastly, we use $u_{n}$ for the uniform probability measure on $\{1,\cdots,n\},$ and if $A$ is  a finite set, we use $u_{A}$ for the uniform probability measure on the finite set $A.$

\subsection{Preliminaries on Sofic Groups and Spectral Measures}\label{S:spectralmeasure}

	We start by defining the basic notions of the group von Neumann algebra and trace. For our purposes, we will need to induce all of our operations to the matricial level, and this will be done in quite a natural way. Let $\Gamma$ be a countable discrete group.  We define the left regular representation $\lambda\colon \Gamma\to \mathcal{U}(\ell^{2}(\Gamma))$ by
\[(\lambda(g)f)(x)=f(g^{-1}x).\]
 We extend this to a map $\lambda\colon \CC(\Gamma)\to B(\ell^{2}(\Gamma))$ in the usual way. We extend to a map
\[\lambda\colon M_{m,n}(\CC(\Gamma))\to B(\ell^{2}(\Gamma)^{\oplus n},\ell^{2}(\Gamma)^{\oplus m})\]
by
\[(\lambda(f)\xi)(j)=\sum_{k=1}^{n}\lambda(f_{jk})\xi(k).\]

	Then for $f\in M_{m,n}(\CC(\Gamma)),g\in M_{n,k}(\CC(\Gamma))$ we have $\lambda(fg)=\lambda(f)\lambda(g).$ For $f=\sum_{g\in \Gamma}a_{g}g\in \CC(\Gamma)$ we let
\[f^{*}=\sum_{g\in \Gamma}\overline{a_{g^{-1}}}g.\]
If $f\in M_{m,n}(\CC(\Gamma)),$ we define $f^{*}\in M_{n,m}(\CC(\Gamma))$ by $(f^{*})_{jk}=f_{kj}^{*},$ then $\lambda(f^{*})=\lambda(f)^{*}.$

	We let $L(\Gamma)$ be the closure of $\lambda(\CC(\Gamma))$  in the weak operator topology, this is called the \emph{group von Neumann algebra} of $\Gamma.$ We can view $M_{m,n}(L(\Gamma))$ as operators in $B(\ell^{2}(\Gamma)^{\oplus n},\ell^{2}(\Gamma)^{\oplus m}).$ Under this identification, $\lambda(M_{m,n}(\CC(\Gamma)))\subseteq M_{m,n}(L(\Gamma)).$ For $x\in L(\Gamma),$ we define $\tau(x)=\ip{x\delta_{e},\delta_{e}}.$ Additionally, we define
\[\Tr\otimes \tau\colon M_{n}(L(\Gamma))\to \CC\]
by
\[\Tr\otimes\tau(x)=\sum_{j=1}^{n}\tau(x_{jj}).\]
To save time, we will  identify $M_{m,n}(\CC(\Gamma))\subseteq M_{m,n}(L(\Gamma))$ via $\lambda,$ thus any construction which applies to $M_{m,n}(L(\Gamma))$ will apply to $M_{m,n}(\CC(\Gamma)).$ If $x\in M_{m,n}(L(\Gamma))$ we let $\widehat{x}\colon \{1,\cdots,m\}\times \{1,\cdots,n\}\times \Gamma\to \CC$ be given by
\[\widehat{x}(j,k,g)=\tau(x_{jk}g^{-1}).\]
 If $E_{jk}\otimes g\in M_{m,n}(\CC(\Gamma))$ is given by $g$ in the $jk$ position and zero elsewhere, we have for any $f\in M_{m,n}(\CC(\Gamma))$
\[f=\sum_{\substack{1\leq j\leq m,\\ 1\leq k\leq n}}\sum_{g\in \Gamma}\widehat{f}(j,k,g)E_{jk}\otimes g.\]

	For $x\in M_{m,n}(L(\Gamma))$ we use $\|x\|_{\infty}$ for the operator norm of $x.$ In particular, since we will view $\CC(\Gamma)\subseteq L(\Gamma),$ we use $\|f\|_{\infty}$ for the operator norm of $f\in \CC(\Gamma)$ as an operator on $\ell^{2}(\Gamma),$ and a similar remark applies to elements in $M_{m,n}(\CC(\Gamma)).$ Additionally, for $x\in M_{m,n}(L(\Gamma)),$ we will use
\[\|x\|_{2}^{2}=\Tr\otimes \tau(x^{*}x),\]
and similarly for $f\in M_{m,n}(\CC(\Gamma)).$ We thus caution the reader that $\|f\|_{\infty}$ does not refer to the supremum of $|\widehat{f}(j,k,g)|,$ for this we will use $\|\widehat{f}\|_{\infty}.$ Note that this agrees in the case $\Gamma=\ZZ^{d}$ with the usual practice of viewing elements of $\ZZ(\ZZ^{d})$ as elements of $C(\TT^{d}).$ For $f\in M_{m,n}(L(\Gamma)),$
\[\|f\|_{2}=\|\widehat{f}\|_{2}.\]
Similarly, we will use $|f|\in M_{n}(L(\Gamma))$ for the operator square root of $f^{*}f$ if $f\in M_{m,n}(\CC(\Gamma)),$ and not for the element of $\CC(\Gamma)$ whose coefficients are the pointwise absolute value of the coefficients of $f.$ We leave it as an exercise to verify the following properties  (using $1$ for the identity element of $M_{n}(L(\Gamma))).$
\begin{list}{ \arabic{pcounter}:~}{\usecounter{pcounter}}
\item $\Tr\otimes \tau(1)=n,$\\
\item $\Tr\otimes \tau(x^{*}x)\geq 0,$  with equality if and only if $x=0$,\\
\item $\Tr\otimes \tau(xy)=\Tr\otimes \tau(yx),$  for all $x,y\in M_{n}(L(\Gamma)),$
\item $\Tr\otimes \tau$ is weak operator topology continuous.
\end{list}

\begin{definition}\emph{Let $\Gamma$ be a countable discrete group, and $x\in M_{n}(L(\Gamma))$ be a normal element. We let $\mu_{x}$ be the Borel measure on the spectrum of $x$ defined by $\mu_{x}(E)=\Tr\otimes \tau(\chi_{E}(x)),$ it is called the} spectral measure of $x$ with respect to $\tau.$\emph{ Additionally, if $A\in M_{n}(M_{m}(\CC))$ is normal we define the} spectral measure of $A$ with respect to $\tr_{m}$ \emph{by $\mu_{A}(E)=\Tr\otimes\tr_{m}(\chi_{E}(A)).$}\end{definition}

For a normal element $x\in M_{n}(L(\Gamma)),$ we remark that $\mu_{x}$ is supported in
\[\{z\in \CC:|z|\leq \|x\|_{\infty}\}.\]
For readers less familiar with functional calculus, we note that we may characterize the spectral measure of $x$ in the following equivalent way:
\[\int t^{n}\,d\mu_{x}(t)=\Tr\otimes \tau(x^{n}).\]
For the definition of topological entropy, we need to restrict ourselves to the class of \emph{sofic} groups.

\begin{definition}\emph{Let $\Gamma$ be a countable discrete group. A} sofic approximation of $\Gamma$ \emph{is a sequence $\Sigma=(\sigma_{i}\colon \Gamma\to S_{d_{i}})$ of functions (not assumed to be homomorphisms) such that (using $u_{d_{i}}$ for the uniform measure on $\{1,\dots,d_{i}\}$)}
\begin{list}{ \arabic{pcounter}:~}{\usecounter{pcounter}}
\item $d_{i}\to \infty$,
\item $u_{d_{i}}(\{j:(\sigma_{i}(g)\sigma_{i}(h))(j)=\sigma_{i}(gh)(j)\})\to 1,$\emph{ for all $g,h\in \Gamma$,}
\item $u_{d_{i}}(\{j:\sigma_{i}(g)(j)\ne \sigma_{i}(h)(j)\})\to 1,$ \emph{for all $g\ne h\in \Gamma$.}
\end{list}
\emph{We say that $\Gamma$ is} sofic\emph{ if it has a sofic approximation.}\end{definition}

	We could remove the condition $d_{i}\to \infty,$ and still have the same definition of a sofic group. However, in order for the definition of topological entropy to be an invariant we need $d_{i}\to \infty.$ The condition $d_{i}\to \infty$ is also implied if $\Gamma$ is infinite, which will be the main case we are interested in anyway. It is known that the class of sofic groups contain all amenable groups, all residually sofic groups, all locally sofic groups, all linear groups and is closed under free products with amalgamation over amenable subgroups. For more see \cite{ESZ2},\cite{LPaun},\cite{DKP},\cite{PoppArg}.

	Let $\Sigma=(\sigma_{i}\colon \Gamma\to S_{d_{i}})$ be a sofic approximation. We extend $\sigma_{i}$ to a map
\[\sigma_{i}\colon \CC(\Gamma)\to M_{d_{i}}(\CC)\]
by
\[\sigma_{i}(f)=\sum_{g\in \Gamma}\widehat{f}(g)\sigma_{i}(g),\]
and
\[\sigma_{i}\colon M_{m,n}(\CC(\Gamma))\to M_{m,n}(M_{d_{i}}(\CC))\]
by
\[\sigma_{i}(f)_{jk}=\sigma_{i}(f_{jk}).\]

We shall present a Lemma from \cite{Me4}. For terminology, if $A\in M_{m,n}(\CC),$ we use $\|A\|_{2}^{2}=\tr_{n}(A^{*}A),$ we shall use $\|A\|_{\infty}$ for the operator norm of $A.$

\begin{lemma}[Lemma 2.6 in \cite{Me4}]\label{L:weak*convergence}Let $\Gamma$ be a countable discrete sofic group wit sofic approximation $\Sigma=(\sigma_{i}\colon \Gamma\to S_{d_{i}}).$ Let $f\in M_{m,n}(\CC(\Gamma)),$ and let $A_{i}\in M_{m,n}(M_{d_{i}}(\CC))$ with $\sup_{i}\|A_{i}\|_{\infty}<\infty,$ and $\|\sigma_{i}(f)-A_{i}\|_{2}\to 0.$ Then,
\[\mu_{|A_{i}|}\to \mu_{|f|}\]
in the weak$^{*}$-topology.
\end{lemma}

We leave it as an exercise to verify that if $x\in M_{n}(L(\Gamma))$ is normal, and $\phi\colon \CC\to \CC$ is bounded and Borel, then
\[\Tr\otimes \tau(\phi(x))=\int \phi(t)\,d\mu_{x}(t).\]
This motivates the following definition.

\begin{definition}\emph{Let $\Gamma$ be a countable discrete group, and $x\in M_{n}(L(\Gamma)).$ We define the} Fuglede-Kadison determinant of $x$ \emph{by}
\[\Det_{L(\Gamma)}(x)=\exp\left(\int_{[0,\infty)} \log(t)\,d\mu_{|x|}(t)\right).\]
\emph{With the convention that $\exp(-\infty)=0.$ Note that $\mu_{|x|}$ is supported in a compact set, so this definition makes sense. If $x\in M_{m,n}(L(\Gamma)$ we define the} positive Fuglede-Kadison determinant of $x$ \emph{by}
\[\Det_{L(\Gamma)}^{+}(x)=\Det_{L(\Gamma)}(|x|+\chi_{\{0\}}(|x|))=\exp\left(\int_{(0,\infty)}\log(t)\,d\mu_{|x|}(t)\right).\]
\end{definition}
If $A\in M_{n}(\CC),$ then the positive determinant of $A,$ written $\Det^{+}(A)$ is equal to $\Det(|A|+\chi_{\{0\}}(|A|))$ i.e. the product of the nonzero eigenvalues of $|A|.$ We leave it as an exercise to verify that
\[\Det^{+}(A)=\Det^{+}_{\CC}(A)\]
where $\CC$ is regarded as a tracial von Neumann algebra with its unique tracial state. We need the following result of Elek-Lippner (see \cite{ElekLip} Theorem 3 in Section 6), which follows from the weak$^{*}$ convergence we have already shown.

\begin{cor}\label{C:integrability} Let $\Gamma$ be a countable discrete sofic group, and $f\in M_{m,n}(\ZZ(\Gamma)).$ Then
\[\Det^{+}_{L(\Gamma)}(f)\geq 1.\]

	In particular, $|\log(t)|$ is integrable with respect to $\mu_{|f|}$ on $\spec(|f|)\setminus\{0\}.$
\end{cor}

Finally, we end with one more approximation Lemma which will be relevant for our purposes. For this, we need some more functional analysis. Let $\rho\colon \Gamma\to \ell^{2}(\Gamma)$ be the right regular representation given by
\[[\rho(g)f](x)=f(xg).\]
If $\mathcal{K}\subseteq\ell^{2}(\Gamma)^{\oplus n}$ is a closed linear subspace which is invariant under $\rho^{\oplus n},$ it is known that $\Proj_{\mathcal{K}}\in M_{n}(L(\Gamma)).$ We define the von Neumann dimension of $\mathcal{K}$ by
\[\dim_{L(\Gamma)}(\mathcal{K})=\Tr\otimes \tau(\Proj_{\mathcal{K}}).\]
It is known that $\dim_{L(\Gamma)}(\mathcal{K})=0$ if and only if $\mathcal{K}=0,$ that $\mathcal{K}\cong \mathcal{H}$ as representations of $\Gamma$ implies that $\dim_{L(\Gamma)}(\mathcal{K})=\dim_{L(\Gamma)}(\mathcal{H}),$  and that for any $x\in M_{m,n}(L(\Gamma)),$
\[\dim_{L(\Gamma)}(\overline{\im x})+\dim_{L(\Gamma)}(\ker(x))=n.\]
See Theorem 1.12 in \cite{Luck} for proofs of these facts. We need the following analogue of L\"{u}ck approximation valid for a general sofic group.

\begin{lemma}[\cite{ElekSzaboDeterminant} Proposition 6.1]\label{L:injective} Let $\Gamma$ be a countable discrete sofic group with sofic approximation $\Sigma=(\sigma_{i}\colon \Gamma\to S_{d_{i}}).$ Let $f\in M_{m,n}(\ZZ(\Gamma)),$ and let $A_{i}\in M_{m,n}(M_{d_{i}}(\ZZ))$ with $\sup_{i}\|A_{i}\|_{\infty}<\infty,$ and $\|\sigma_{i}(f)-A_{i}\|_{2}\to 0.$ Then,
\[\dim_{L(\Gamma)}(\ker \lambda(f))=\lim_{i\to \infty}\frac{\dim_{\RR}(\ker(A_{i})\cap (\RR^{d_{i}})^{\oplus n})}{d_{i}}.\]
\end{lemma}

We end this section with a proposition which should will translate our hypotheses in terms of viewing $f\in M_{m,n}(\CC(\Gamma))$ as a ``left'' multiplication operator to that of a  ``right'' multiplication operator. This proposition is well-known, we only decide to include it to clarify any potential confusion the reader might have between left and right multiplication operators.  For $f\in \CC(\Gamma)$ we define
\[r(f):\ell^{2}(\Gamma)\to \ell^{2}(\Gamma)\]
by
\[(r(f)\xi)(g)=\sum_{x\in\Gamma}\xi(gx^{-1})\widehat{f}(x).\]
For $A\in M_{m,n}(\CC(\Gamma))$ define
\[r(A)\colon \ell^{2}(\Gamma)^{\oplus m}\to \ell^{2}(\Gamma)^{\oplus n}\]
by
\[(r(A)\xi)(l)=\sum_{s=1}^{m}r(A_{sl})\xi(s).\]

\begin{proposition}\label{P:Left/RightIssues} Let $\Gamma$ be a countable discrete group, and fix $f\in M_{m,n}(\ZZ(\Gamma)).$ Consider the following conditions:
\begin{enumerate}[(a)]
\item $\lambda(f)$  is injective,
\item $\lambda(f)$ has dense image,
\item $r(f)$ is injective,
\item $r(f)$ has dense image.
\end{enumerate}
Then
\begin{enumerate}[(i)]
\item  $(a)$ and $(d)$ are equivalent,
\item $(b)$ and $(c)$ are equivalent,
\item $(a)$ implies that $n\leq m,$
\item $(b)$ implies that $m\leq n$,
\item if $m=n,$ then all of $(a),(b),(c),(d)$ are equivalent.
\end{enumerate}

\end{proposition}

\begin{proof} To prove $(i),$ consider the unitary
\[R\colon \ell^{2}(\Gamma)\to \ell^{2}(\Gamma)\]
given by
\[(R\xi)(g)=\xi(g^{-1}).\]
Fix $\alpha\in \ZZ(\Gamma)$ and $\xi\in \ell^{2}(\Gamma).$ Then for any $h\in \Gamma:$
\begin{align*}
(\lambda(\alpha)R(\xi))(h)&=\sum_{x\in \Gamma}\widehat{\alpha}(x)(R(\xi))(x^{-1}h)\\
&=\sum_{x\in \Gamma}\widehat{\alpha}(x)\xi(h^{-1}x)\\
&=\sum_{x\in \Gamma}\widehat{\alpha}(x^{-1})\xi(h^{-1}x^{-1})\\
&=\sum_{x\in\Gamma}\widehat{\alpha^{*}}(x)\xi(h^{-1}x^{-1})\\
&=\sum_{x\in\Gamma}\xi(h^{-1}x^{-1})\widehat{\alpha^{*}}(x)\\
&=(r(\alpha^{*})\xi)(h^{-1})\\
&=R(r(\alpha^{*})\xi)(h),
\end{align*}
so $\lambda(\alpha)R=Rr(\alpha^{*}).$
Now fix $\zeta\in \ell^{2}(\Gamma)^{\oplus n},$ we then have for all $1\leq l\leq m:$
\[(\lambda(f)R^{\oplus n}(\zeta))(l)=\sum_{j=1}^{n}\lambda(f_{lj})R(\zeta(j))=\sum_{j=1}^{n}R(r(f_{lj}^{*})\zeta(j))=R\left(\sum_{j=1}^{n}r((f^{*})_{jl})\zeta(j)\right)
=R((r(f^{*})\zeta)(l)),\]
where we use that we already showed that $\lambda(\alpha)R=R(r(\alpha^{*}))$ for all $\alpha\in \ZZ(\Gamma).$
We can summarize the above computation by saying that
\[ \lambda(f)R^{\oplus n}=   R^{\oplus n}r(f^{*})=R^{\oplus n}r(f)^{*}.\]
Statements $(i),(ii)$ now follow from the functional analytic fact that $(\ker(T))^{\perp}=\overline{\im(T^{*})}$ for any bounded linear operator $T$ between two Banach spaces. Statements (iii),(iv) are consequences of the Rank-Nullity Theorem for von Neumann dimension (see \cite{Luck} Theorem 1.12 (2)), and $(v)$ also follows from the Rank-Nullity Theorem for von Neumann dimension.

\end{proof}

\section{The Main Reduction}\label{S:mainreduction}

	Let $f\in M_{m,n}(\ZZ(\Gamma)),$ we will use $X_{f}$ for the Pontryagin dual of $\ZZ(\Gamma)^{\oplus n}/r(f)(\ZZ(\Gamma)^{\oplus m})$ The goal of this section is to give an alternate formula for the entropy of $\Gamma\actson X_{f},$ which will be simpler for us to deal with and will reduce the problem to (a limit of) finite-dimensional analysis. The essential idea, as we stated before, is that instead of dealing with microstates
\[\{1,\cdots,d_{i}\}\to X_{f}\]
we deal with microstates
\[\{1,\cdots,d_{i}\}\to (\TT^{\Gamma})^{n},\]
which are ``small" on $r(f)(\ZZ(\Gamma)^{\oplus m})$ (viewing $ (\TT^{\Gamma})^{n}$ as the dual of $\ZZ(\Gamma)^{\oplus n}$) and note that these have to be ``close" to microstates which actually take values in $X_{f}.$

	We first recall the definition of topological entropy for a sofic group, throughout whenever $X$ is a set we identify $X^{n}$ with all functions $\{1,\cdots,n\}\to X.$ If $(X,d)$ is a pseudometric on $X$ and $1\leq p\leq\infty,$ we let $d_{p}$ be the pseudometric on $X^{n}$ defined by
\[d_{p}(\phi,\psi)^{p}=\frac{1}{n}\sum_{j=1}^{n}d(\phi(j),\psi(j))^{p},\]
with the usual modification if $p=\infty.$

\begin{definition}\emph{Let $\Gamma$ be a countable discrete sofic group with sofic approximation $\Sigma=(\sigma_{i}\colon \Gamma\to S_{d_{i}}).$  Let $X$ be a compact metrizable space with $\Gamma\actson X$ by homeomorphisms. If $\rho$ is a continuous pseudometric on $X,$ $\delta>0$ and $F\subseteq \Gamma$ finite we let $\Map(\rho,F,\delta,\sigma_{i})$ be all maps $\phi \colon \{1,\cdots,d_{i}\}\to X$ such that $\rho_{2}(\phi \circ \sigma_{i}(g),g\phi)<\delta$ for all $g\in F.$ }\end{definition}
We will typically refer to the elements of $\Map(\rho,F,\delta,\sigma_{i})$ as ``microstates". This is only a heuristic term and will not be defined rigorously.

\begin{definition}\emph{Let $\Gamma$ be a countable discrete group and $\Gamma \actson X$ by homeomorphisms. We say that a continuous pseudometric $\rho$ on $X$ is} dynamically generating \emph{if whenever $x\ne y$ in $X,$ then there is a $g\in \Gamma$ so that $\rho(gx,gy)>0.$}\end{definition}

\begin{definition} \emph{Let $\Gamma$ be a countable discrete sofic group with sofic approximation $\Sigma=(\sigma_{i}\colon \Gamma\to S_{d_{i}}).$ Let $X$ be a compact metrizable space and $\Gamma\actson X$ by homeomorphisms and fix a dynamically generating pseudometric $\rho$ on $X.$ We define the} topological entropy of $\Gamma\actson X$ \emph{by}

\[h_{\Sigma}(\rho,F,\delta,\varespilon)=\limsup_{i\to \infty}\frac{\log S_{\varepsilon}(\Map(\rho,F,\delta,\sigma_{i}),\rho_{2})}{d_{i}}\]
\[h_{\Sigma}(\rho,\varepsilon)=\inf_{\substack{F\subseteq \Gamma\mbox{ \emph{ finite, }}\\ \delta>0}}h_{\Sigma}(\rho,F,\delta,\varepsilon)\]
\[h_{\Sigma}(X,\Gamma)=\sup_{\varepsilon>0}h_{\Sigma}(\rho,\varepsilon).\]
\emph{By Theorem 4.5  in \cite{KLi}, and Proposition 2.4 in \cite{KLi2} this does not depend on the choice of dynamically generating pseudometric.}\end{definition}

	A few remarks about the definition. First in \cite{KLi2} Kerr-Li use $N_{\varepsilon}$ instead of $S_{\varepsilon},$ however by inequality (\ref{E:sepspan}) in Section 1 this does not matter. We will actually use both $N_{\varespilon}$ and $S_{\varespilon}.$   Secondly, in \cite{KLi2}, Kerr-Li typically use
\[N_{\varespilon}(\Map(\rho,F,\delta,\sigma_{i}),\rho_{\infty}).\]
We will prefer to use $\rho_{2}$ instead of $\rho_{\infty},$ firstly because we will use a large amount of Borel functional calculus, which is much nicer in a Hilbert-space, even in the finite dimensional setting. Secondly, we will get our lower estimates on topological entropy by using a perturbation argument. Essentially, we will include our space $X$ in a larger space $Y,$ and consider microstates $\{1,\cdots,d_{i}\}\to Y$ which are ``close" to $X,$ our methods will necessitate this closeness being with respect to $\rho_{2},$ not $\rho_{\infty}.$ If we use $S_{\varespilon}(\cdots,\rho_{2})$ instead of $S_{\varespilon}(\cdots,\rho_{\infty}),$ then it is  significantly easier to show that this method gives the topological entropy. The idea of using $\rho_{2}$ instead of $\rho_{\infty}$ goes back to Hanfeng Li in \cite{Li2}, it was also used in the proof of Lemma 7.12 in \cite{BowenLi}.

	Let us formulate the perturbation ideas more precisely, in the case of algebraic actions. For notation, if $x\in \RR,$ we use
\[|x+\ZZ|=\inf_{k\in \ZZ}|x+k|.\]

\begin{definition}\emph{Let $\Gamma$ be a countable discrete sofic group with sofic approximation $\Sigma=(\sigma_{i}\colon \Gamma\to S_{d_{i}}).$ Let $B\subseteq A$ be countable $\ZZ(\Gamma)$-modules. Let $\rho$ be a dynamically generating pseudometric on $\widehat{A},$ and $D\subseteq B$ be any set so that $\Gamma D$ generates $B$ as an abelian group. For $F\subseteq\Gamma$ finite, $E\subseteq D$ finite, and $\delta>0,$ we let $\Map(\rho|E,F,\delta,\sigma_{i})$ be the set of all $\phi\in \Map(\rho,F,\delta,\sigma_{i})$ so that}
\[\frac{1}{d_{i}}\sum_{j=1}^{d_{i}}|\phi(j)(a)|^{2}<\delta^{2}\]
\emph{for all $a\in E.$ We set}
\[h_{\Sigma}(\rho|E,F,\delta,\varepsilon)=\limsup_{i\to \infty}\frac{1}{d_{i}}\log S_{\varepsilon}(\Map(\rho|E,F,\delta,\sigma_{i}),\rho_{2}),\]
\[h_{\Sigma}(\rho|D,\varespilon)=\inf_{\substack{F\subseteq \Gamma\mbox{\emph{ finite, }}\\ E\subseteq D\mbox{ \emph{ finite,} }\\ \delta>0}}h_{\Sigma}(\rho|E,F,\delta,\varespilon),\]
\[h_{\Sigma}(\rho|D)=\sup_{\varespilon>0}h_{\Sigma}(\rho|D,\varespilon).\]
\end{definition}

	Let us motivate the definition a little. Intuitively, a microstate is a finitary model of our dynamical system. Given the algebraic structure of $A/B,$ a microstate for $\Gamma\actson \widehat{A/B}$ can be thought of in two different ways: one is an element of $\Map(\rho|_{\widehat{A/B}},F,\delta,\sigma_{i})$ (recall that $\widehat{A/B}$ can be viewed as a subspace of $\widehat{A}).$ But, since $\widehat{A/B}$ are all elements in $\widehat{A}$ which are zero on $B$, we may also think of a microstate for $\Gamma\actson \widehat{A/B},$ as a microstate for $\Gamma\actson\widehat{A}$ which is small on $B.$ Thus, $\Map(\rho|E,F,\delta,\sigma_{i})$ can be simply be viewed as another microstates space for the action of $\Gamma$ on $\widehat{A/B}.$ We now reformulate topological entropy in terms of this new microstates space.

\begin{lemma}\label{L:equimicro} Let $\Gamma$ be  countable discrete sofic group with sofic approximation $\Sigma.$ Let $B\subseteq A$ be countable $\ZZ(\Gamma)$-modules, let $\rho$ be a dynamically generating pseudometric on $\widehat{A},$ and let $D\subseteq B$ be such that $\Gamma D$ generates $B$ as an abelian group. Then,
\[h_{\Sigma}(\widehat{A/B},\Gamma)=h_{\Sigma}(\rho|D,\Gamma).\]
\end{lemma}

\begin{proof} We use $\rho\big|_{\widehat{A/B}}$ to compute the entropy of $\Gamma \actson \widehat{A/B}.$ A compactness argument implies that for all $F\subseteq \Gamma$ finite,  $\delta>0,$ there are finite $E\subseteq D,F'\subseteq \Gamma,$ and a $\delta'>0$ so that if $\chi\in \widehat{A},$ and
$|\chi(ga)|<\delta'$
for all $g\in F',a\in E,$ then there is a $\widetidle{\chi}\in \widehat{A/B}$ so that
\[\sup_{g\in F}\rho(g\chi,g\widetidle{\chi})<\delta.\]
The proof now follows in the same way as Proposition 4.3 in \cite{Me4}.

\end{proof}

	We will apply this to the situation when $A=\ZZ(\Gamma)^{\oplus n},B=r(f)(\ZZ(\Gamma)^{\oplus m}),$ where $f\in M_{m,n}(\ZZ(\Gamma)),$ but first we need more notation.  If $x\in \RR^{n}$ we will always use $\|x\|_{2}$ for the $\ell^{2}$-norm of $x$ with respect to the uniform probability measure. If $\Lambda\subseteq \RR^{n}$ is a lattice we set
\[\|x\|_{2,\Lambda}=\inf_{\lambda\in \Lambda}\|x-\lambda\|_{2},\]
and we use
\[\theta_{2,\Lambda}(x,y)=\|x-y\|_{2,\Lambda}.\]
Finally, if $T\in M_{m,n}(\ZZ),$ and $\delta>0,$ we set
\[\Xi_{\delta}(T)=\{\xi\in \RR^{n}:\|T\xi\|_{2,\ZZ^{m}}<\delta\}.\]

\begin{proposition}\label{P:mainreduction} Let $\Gamma$ be a countable discrete sofic group with sofic approximation $\Sigma=(\sigma_{i}\colon \Gamma\to S_{d_{i}}).$ Let $f\in M_{m,n}(\ZZ(\Gamma)),$ then
\[h_{\Sigma}(X_{f},\Gamma)=\sup_{\varespilon>0}\inf_{\delta>0}\limsup_{i\to \infty}\frac{1}{d_{i}}\log S_{\varespilon}(\Xi_{\delta}(\sigma_{i}(f)),\theta_{2,(\ZZ^{d_{i}})^{\oplus n}}).\]
\end{proposition}

\begin{proof} Set $A=\ZZ(\Gamma)^{\oplus n},B=r(f)(\ZZ(\Gamma)^{\oplus m}).$ We shall view $\widehat{A}=(\TT^{\Gamma})^{n},$ by
\[\ip{\zeta,\alpha}=\sum_{\substack{1\leq l\leq n,\\ g\in \Gamma}}\zeta(l)(g)\widehat{\alpha}(l)(g),\]
for $\zeta\in (\TT^{\Gamma})^{n},\alpha\in \ZZ(\Gamma)^{\oplus n}.$

	Let $\rho$ be the dynamically generating pseudometric on $\widehat{A}$ given by
\[\rho(\chi_{1},\chi_{2})^{2}=\frac{1}{n}\sum_{k=1}^{n}|\chi_{1}(k)(e)-\chi_{2}(k)(e)|^{2}.\]
Given $x\in M_{s,t}(L(\Gamma)),$ we let $\widetilde{x}\in M_{t,s}(L(\Gamma))$ be defined by
\[(\widetilde{x})_{ij}=x_{ji}.\]
Write
\[f=\begin{bmatrix}
\widetilde{f_{1}}\\
\widetilde{f_{2}}\\
\vdots\\
\widetilde{f_{m}}
\end{bmatrix},\]
where $f_{j}\in M_{n,1}(\ZZ(\Gamma)),$ and view $\ZZ(\Gamma)^{\oplus n}=M_{n,1}(\ZZ(\Gamma)).$ Lastly, set
$D=\{f_{1},\cdots,f_{m}\}.$
We now apply Lemma \ref{L:equimicro} for this $A,B,\rho,D.$ For $\xi\in \Xi_{\delta}(\sigma_{i}(f)),$ define
\[\phi_{\xi}\colon \{1,\cdots,d_{i}\}\to (\TT^{\Gamma})^{n}\]
by
\[\phi_{\xi}(j)(k)(g)=\xi(\sigma_{i}(g)^{-1}j)(k)+\ZZ.\]

By a simple computation
\[\ip{\phi_{\xi}(j),f_{l}}=(\sigma_{i}(f_{l})\xi)(j)+\ZZ\]
(recall that we are viewing $(\TT^{\Gamma})^{n}$ as the dual of $\ZZ(\Gamma)^{\oplus n}$). As
\[\frac{1}{m}\sum_{l=1}^{m}\|\sigma_{i}(f_{l})\xi\|_{2,\ZZ^{d_{i}}}^{2}=\|\sigma_{i}(f)\xi\|_{2,(\ZZ^{d_{i}})^{\oplus m}}^{2},\]
\[\rho_{2}(\phi_{\xi},\phi_{\xi'})=\|\xi-\xi'\|_{2,(\ZZ^{d_{i}})^{\oplus n}},\]
by the preceding Lemma we find that
\[\sup_{\varespilon>0}\inf_{\delta>0}\limsup_{i\to \infty}\frac{1}{d_{i}}\log S_{\varespilon}(\Xi_{\delta}(\sigma_{i}(f)),\theta_{2,(\ZZ^{d_{i}})^{\oplus n}})\leq h_{\Sigma}(X_{f},\Gamma).\]
For the reverse inequality, given $\phi\in \Map(\rho|D,F,\delta,\sigma_{i}),$ let
$\zeta_{\phi}\in(\TT^{d_{i}})^{n}$
be given by
\[\zeta_{\phi}(k)(j)=\phi(j)(k)(e),\]
and let
$\xi_{\phi}\in (\RR^{d_{i}})^{n}$
be any lift of $\zeta_{\phi}$ under the quotient map
\[(\RR^{d_{i}})^{n}\to (\TT^{d_{i}})^{n}.\]

	Viewing $(\TT^{\Gamma})^{n}$ as the dual of $\ZZ(\Gamma)^{\oplus n},$
\begin{align*}
\ip{\phi(j),f_{l}}&=\sum_{k=1}^{n}\sum_{g\in \Gamma}\widehat{f_{l}}(k,g)\phi(j)(k)(g)\\
&=\sum_{k=1}^{n}\sum_{g\in \Gamma}\widehat{f_{l}}(k,g)[g^{-1}\phi](j)(k)(e),
\end{align*}
and
\[(\sigma_{i}(f_{l})\zeta_{\phi})(j)=\sum_{k=1}^{n}\sum_{g\in \Gamma}\widehat{f_{l}}(k,g)\phi(\sigma_{i}(g)^{-1}(j))(k)(e).\]
Thus it is not hard to see that
\[\frac{1}{d_{i}}\sum_{j=1}^{d_{i}}|\ip{\phi(j),f_{l}}-(\sigma_{i}(f_{l})\zeta_{\phi})(j)|^{2}\leq \eta(F,\delta),\]
with
\[\lim_{(F,\delta)}\eta(F,\delta)=0,\]
(Here the pairs are ordered by $(F,\delta)\leq (F',\delta')$ if $\delta'\leq \delta,F'\supseteq F).$

	Since
\[\frac{1}{d_{i}}\sum_{j=1}^{d_{i}}|\ip{\phi(j),f_{l}}|^{2}<\delta^{2},\]
we have that
\[\|\sigma_{i}(f_{l})\xi_{\phi}\|_{2,\ZZ^{d_{i}}}\leq \delta+\eta(F,\delta)^{1/2}.\]
As
\[\|\sigma_{i}(f)\xi_{\phi}\|_{2,(\ZZ^{d_{i}})^{\oplus m}}^{2}=\frac{1}{m}\sum_{l=1}^{m}\|\sigma_{i}(f_{l})\xi_{\phi}\|_{2,\ZZ^{d_{i}}}^{2},\]
if we are given a $\delta'>0,$ we can find a finite $F\subseteq\Gamma$, and a $\delta>0$ so that
\[\xi_{\phi}\in \Xi_{\delta'}(\sigma_{i}(f))\]
for all $\phi\in\Map(\rho|D,F,\delta,\sigma_{i}).$ As
\[\|\xi_{\phi}-\xi_{\psi}\|_{2,(\ZZ^{d_{i}})^{\oplus n}}=\rho_{2}(\phi,\psi),\]
for all $\psi,\phi\in \Map(\rho|D,F,\delta,\sigma_{i})$ we have
\[S_{\varepsilon}(\Map(\rho|D,F,\delta,\sigma_{i}),\rho_{2})=S_{\varepsilon}(\{\xi_{\phi}:\phi\in \Map(\rho|D,F,\delta,\sigma_{i})\},\theta_{2,(\ZZ^{d_{i}})^{\oplus n}}).\]
By our choice of $F,\delta$ we have
\[\{\xi_{\phi}:\phi\in \Map(\rho|D,F,\delta,\sigma_{i})\}\subseteq \Xi_{\delta'}(\sigma_{i}(f)),\]
 so for any $\varepsilon>0$ we have (by (\ref{E:sepspan})):
\[ S_{\varepsilon}(\Map(\rho|D,F,\delta,\sigma_{i}),\rho_{2})=S_{\varepsilon}(\{\xi_{\phi}:\phi\in \Map(\rho|D,F,\delta,\sigma_{i})\},\theta_{2,(\ZZ^{d_{i}})^{\oplus n}})\leq S_{\varepsilon/2}(\Xi_{\delta'}(\sigma_{i}(f)),\theta_{2,(\ZZ^{d_{i}})^{\oplus n}}).\]
Thus the reverse inequality follows.

\end{proof}

The above proposition will be our main tool to evaluate the topological entropy of $\Gamma\actson X_{f}.$ Let us remark on the advantage of our approach. Previously, the techniques in sofic entropy of algebraic actions have been as follows: take $\xi \in \sigma_{i}(f)^{-1}(\ZZ^{d_{i}}),$ and consider $\phi_{\xi}$ as in the above proposition. If $\Gamma$ is residually finite, and $\sigma_{i}$ comes from a sequence of finite quotients, then $\phi_{\xi}$ maps into $X_{f}$ instead of just $(\TT^{\Gamma})^{n}.$ Similar remarks apply if $\Gamma$ is amenable (and not necessarily residually finite). Now, one is led to estimate
\[\left|\frac{\sigma_{i}(f)^{-1}(\ZZ^{d_{i}})}{\ZZ^{d_{i}}}\right|.\]
If $\sigma_{i}(f)$ is invertible, we will see  later that this is
\[|\det(\sigma_{i}(f))|,\]
and if we have reasonable control over $\ker(\sigma_{i}(f)),$ then this expression is close to
\[\Det^{+}(\sigma_{i}(f)).\]

	Now to get the lower bound, one has to establish
\begin{equation}\label{E:hard}
\lim_{i\to \infty}\Det^{+}(\sigma_{i}(f))^{1/d_{i}}=\Det_{L(\Gamma)}(f),\
\end{equation}
when $n=m.$ Equivalently,
\[\int_{(0,\infty)}\log(t)\,d\mu_{|\sigma_{i}(f)|}(t)\to \int_{(0,\infty)}\log(t)\,d\mu_{|f|}(t).\]
 However, this is far from obvious given that all we know is that
\[\mu_{|\sigma_{i}(f)|}\to \mu_{|f|}\]
weak$^{*}.$  In the case when $\Gamma$ is residually finite and $\sigma_{i}$ come from a sequence of finite quotients, such a statement amounts to counting entropy as a growth rate of periodic points. In \cite{KLi}, David Kerr and Hanfeng Li assume an invertibility hypothesis on $f,$ which implies that $\mu_{|\sigma_{i}(f)|},\mu_{|f|}$ have support inside $[C,M],$ for some $C,M>0,$ and so $(\ref{E:hard})$ does follow by weak$^{*}$ convergence.  In \cite{BowenLi}, Bowen-Li consider the case when $f$ is a Laplacian operator, $\Gamma$ is residually finite, and $\sigma_{i}$ come from a sequence of finite quotients. This specific structure allows them to control the kernel of $f$ and  the asymptotics of distributions of periodic points is true by some nontrivial graph-theoretic facts. In \cite{LiThom}, Li-Thom prove a result analogous to (\ref{E:hard}) for amenable groups, using quasi-tiling arguments and a statement similar to the Ornstein-Weiss Lemma. This argument is very special to the case of amenable groups.

	It is our opinion that the approximation results $(\ref{E:hard})$ are too difficult to establish in the nonamenable case  without an invertibility hypothesis (for which the entropy has already been computed),  or very specific information about $f.$ Further, we do not expect (\ref{E:hard}) is true for general sofic approximations. Thus we seek a method of proof avoiding such an approximation result. This is the main advantage of our approach: first to produce microstates one can use not only vectors $\xi\in \sigma_{i}(f)^{-1}(\ZZ^{d_{i}}),$ but also vectors $\xi\in (\RR^{d_{i}})^{\oplus n}$ so that $\|\sigma_{i}(f)\xi\|_{2}<\delta.$ This creates more elements in our  microstates space, making it simpler to get a lower bound. Further since the methods are perturbative in nature, one is allowed more flexibility in perturbing the operator and this will allow greater control over the kernel. We remark that the idea of perturbing $\sigma_{i}(f)$ has already seen some applications to entropy of algebraic actions, see e.g. Theorem 7.1 of \cite{Li2}. We now make precise the notion of perturbation that we are using. For notation, if $j,k\in \NN$ with $j\leq k,$ we use $e_{j}\in \RR^{k}$ for the vector which is $1$ in the $j^{th}$ coordinate and zero elsewhere. If $V$ is a vector space $l,n\in \NN,$ with $l\leq n,$ and $v\in V,$ we use $v\otimes e_{l}\in V^{\oplus n},$ for the element which is $v$ in the $l^{th}$ coordinate and zero elsewhere.
	
\begin{definition}\emph{Let $\Gamma$ be a countable discrete sofic group with sofic approximation $\sigma_{i}\colon\Gamma\to S_{d_{i}}$. Extend $\sigma_{i}$ to $\sigma_{i}\colon M_{m,n}(\ZZ(\Gamma))\to M_{m,n}(M_{d_{i}}(\ZZ))$ by}
\[(\sigma_{i}(f))_{st}=\sum_{g\in\Gamma}\widehat{f_{st}}(g)\sigma_{i}(g).\]
\emph{ Fix $f\in M_{m,n}(\ZZ(\Gamma)),$ a sequence $x_{i}\in M_{m,n}(M_{d_{i}}(\ZZ))$ is said to be a} rank perturbation of $\sigma_{i}(f)$ \emph{if}
\[\sup_{i}\|x_{i}\|_{\infty}<\infty,\]
\[u_{d_{i}}(\{1\leq j\leq d_{i}:x_{i}(e_{l}\otimes e_{j})=\sigma_{i}(f)(e_{l}\otimes e_{j})\mbox{ for all $1\leq l\leq n$}\})\to 1.\]
\end{definition}

The main relevance of rank perturbations is that often we have a much stronger control on images or kernels of rank perturbations than we do of the original operator. This will be clear from the following proposition. For the proof, we define
\[m_{\phi}\colon \ell^{2}(n,u_{n})\to \ell^{2}(n,u_{n})\]
for $\phi\in \ell^{\infty}(n)$ by
\[m_{\phi}(\xi)(l)=\phi(l)\xi(l).\]

\begin{proposition}\label{P:rankperturabtion} Let $\Gamma$ be a countable discrete sofic group, and let $\sigma_{i}\colon\Gamma\to S_{d_{i}}$ be a sofic approximation. Let $f\in M_{m,n}(\ZZ(\Gamma)).$

(i): If $\lambda(f)$ has dense image as on operator on $\ell^{2}(\Gamma)^{\oplus n}$ (so that necessarily $n\geq m$ by Proposition \ref{P:Left/RightIssues}), then there is a rank perturbation $x_{i}$ of $\sigma_{i}(f)$ so that $\im(x_{i})=\RR^{A_{i}},$ for some $A_{i}\subseteq\{1,\dots,d_{i}\}^{m}$ with
\[\lim_{i\to\infty}\frac{|A_{i}|}{d_{i}}=m.\]

(ii): If $\lambda(f)$ has dense image as on operator on  $\ell^{2}(\Gamma)^{\oplus n}$ and is injective as an operator on $\ell^{2}(\Gamma)^{\oplus n}$ (so that necessarily $m=n$ by Proposition \ref{P:Left/RightIssues}), then there is a rank perturbation $x_{i}$ of $\sigma_{i}(f)$ with $x_{i}\in GL_{n}(M_{d_{i}}(\RR)).$

(iii) If $x_{i}$ is \emph{any} rank perturbation of $\sigma_{i}(f)$ we have
\[\sup_{\xi\in \RR^{d_{i}}}\|(x_{i}-\sigma_{i}(f))\xi\|_{2,(\ZZ^{d_{i}})^{\oplus m}}\to_{i\to\infty} 0.\]
\end{proposition}

\begin{proof}
(i) Let $A_{i}\subseteq\{1,\dots,d_{i}\}^{m}$ be such that
\[\Proj_{\im(\sigma_{i}(f))^{\perp}}\big|_{\RR^{A_{i}^{c}}}\]
is an isomorphism onto $\im(\sigma_{i}(f))^{\perp}.$ By Lemma \ref{L:injective} we have
\[\lim_{i\to\infty}\frac{|A_{i}^{c}|}{d_{i}}=\lim_{i\to\infty}\frac{\dim_{\RR}(\ker(\sigma_{i}(f)^{*})\cap (\RR^{d_{i}})^{\oplus m})}{d_{i}}=\dim_{L(\Gamma)}(\ker\lambda(f^{*}))=0.\]
Set $x_{i}=m_{\chi_{A_{i}}}\sigma_{i}(f),$ then $x_{i}$ is a rank perturbation of $\sigma_{i}(f)$ and
\[\frac{|A_{i}|}{d_{i}}\to m.\]
Clearly, $x_{i}((\RR^{d_{i}})^{\oplus n})\subseteq \RR^{A_{i}}.$ Suppose that $\xi\in \RR^{A_{i}},$ choose a $\zeta\in \RR^{A_{i}^{c}}$ so that
\[\Proj_{\im(\sigma_{i}(f))^{\perp}}(\zeta)=\Proj_{\im(\sigma_{i}(f))^{\perp}}(\xi).\]
Thus there is an $\eta\in (\RR^{d_{i}})^{\oplus n}$ so that $\sigma_{i}(f)\eta=\xi-\zeta.$ So
$x_{i}\eta=\xi$
and therefore $x_{i}((\RR^{d_{i}})^{\oplus n})=\RR^{A_{i}}.$

(ii): Let $A_{i}\subseteq\{1,\dots,d_{i}\}^{n}, B_{i}\subseteq\{1,\dots,d_{i}\}^{n}$ be such that
\[\Proj_{\im(\sigma_{i}(f))^{\perp}}\big|_{\RR^{A_{i}^{c}}},\]
\[\Proj_{\ker(\sigma_{i}(f))^{\perp}}\big|_{\RR^{B_{i}}},\]
are isomorphisms onto $\im(\sigma_{i}(f))^{\perp},\ker(\sigma_{i}(f))^{\perp},$ respectively. Set
\[x_{i}^{o}=m_{\chi_{A_{i}}}\sigma_{i}(f)m_{\chi_{B_{i}}}.\]
As in $(i)$ we have that $x_{i}^{o}$ is a rank perturbation of $\sigma_{i}(f).$ Moreover, $|B_{i}|=|A_{i}|.$ We claim that $\ker(x_{i}^{o})\cap (\RR^{d_{i}})^{\oplus n}=\RR^{B_{i}^{c}},x_{i}^{o}((\RR^{d_{i}})^{\oplus n})=\RR^{A_{i}}.$ To see this, suppose that $\xi\in (\RR^{d_{i}})^{\oplus n}$ and that $x_{i}^{o}\xi=0.$ So
\[\sigma_{i}(f)(\chi_{B_{i}}\xi)\in \RR^{A_{i}^{c}},\]
but
\[\Proj_{\im(\sigma_{i}(f))^{\perp}}(\sigma_{i}(f)(\chi_{B_{i}}\xi))=0,\]
so our choice of $A_{i}$ forces $\sigma_{i}(f)(\chi_{B_{i}}\xi)=0.$ So
\[\Proj_{\ker(\sigma_{i}(f))^{\perp}}(\chi_{B_{i}}\xi)=0,\]
and our choices of $B_{i}$ forces $\chi_{B_{i}}\xi=0,$ i.e. $\xi\in \RR^{B_{i}^{c}}.$ The proof that $x_{i}^{o}((\RR^{d_{i}})^{\oplus n})=\RR^{A_{i}}$ is similar to $(i).$  Let $V_{i}\in M_{d_{i}}(\ZZ)$ be such that
\[V_{i}^{*}V_{i}=m_{\chi_{B_{i}^{c}}},\]
\[V_{i}V_{i}^{*}=m_{\chi_{A_{i}^{c}}},\]
(e.g. let $V_{i}$ be the natural operator induced by a bijection $B_{i}^{c}\to A_{i}^{c}$). Set
\[x_{i}=x_{i}^{o}+V_{i}.\]
It is easy to check that $x_{i}$ is a rank perturbation of $\sigma_{i}(f),$ and that $x_{i}\in GL_{d_{i}}(M_{d_{i}}(\RR)).$

(iii): Let
\[J_{i}=\{j:x_{i}(e_{l}\otimes e_{j})=\sigma_{i}(f)(e_{l}\otimes e_{j})\mbox{ for $1\leq l\leq n$}\}.\]
For $\xi\in (\RR^{d_{i}})^{\oplus n}$ we have
\begin{align*}
\|(x_{i}-\sigma_{i}(f))\xi\|_{2,(\ZZ^{d_{i}})^{\oplus n}}^{2}&=\frac{1}{n}\sum_{l=1}^{n}\|((x_{i}-\sigma_{i}(f))\xi)(l)\|_{2,\ZZ^{d_{i}}}^{2}  \\
&=\frac{1}{n}\sum_{l=1}^{n}\|[(x_{i}-\sigma_{i}(f))(\chi_{J_{i}^{c}}(\xi))](l)\|_{2,\ZZ^{d_{i}}}^{2}.
\end{align*}
For any $\zeta\in \RR^{d_{i}}$ we have
\[\|\chi_{J_{i}^{c}}\zeta\|_{2,\ZZ^{d_{i}}}^{2}\leq u_{d_{i}}(J_{i}^{c}).\]
So
\[\|(x_{i}-\sigma_{i}(f))\xi\|_{2,(\ZZ^{d_{i}})^{\oplus n}}^{2}\leq (\|x_{i}\|_{\infty}+\|\widehat{f}\|_{1})^{2}u_{d_{i}}(J_{i}^{c})\]
from the definition of rank perturbation we have
\[\sup_{i}\|x_{i}\|_{\infty}<\infty,\]
\[u_{d_{i}}(J_{i}^{c})\to 0,\]
which proves the proposition.

\end{proof}

\section{Fuglede-Kadison Determinants and Topological Entropy}\label{S:upperbound}

\subsection{The Entropy of $X_{f}$ When $f\in M_{n}(\ZZ(\Gamma))$}
	
	The aim of this section is to prove that when $f\in M_{n}(\ZZ(\Gamma))$ with $\lambda(f)$ injective, then
\[h_{\Sigma}(X_{f},\Gamma)= \log\Det_{L(\Gamma)}(f).\]
As stated before, when $f\in M_{m,n}(\ZZ(\Gamma))$ is injective as an operator on $\ell^{2}(\Gamma)^{\oplus n}$ we have the upper bound
\[h_{\Sigma}(X_{f},\Gamma)\leq \log \Det_{L(\Gamma)}^{+}(f),\]
however the proof of this uses more advanced operator algebraic techniques so we will  postpone it until the next subsection.

	 Before we proceed to the proof of the main theorem, we will need a few more technical lemmas. The first of these Lemmas is essentially the same as  Lemma 7.10 in \cite{BowenLi}, and is equivalent to the statement of Lemma 4.6 in \cite{LiThom}. To state it, let $\mathcal{H},\mathcal{K}$ be finite dimensional Hilbert spaces, and $T\in B(\mathcal{H},\mathcal{K})$ be invertible. If $\delta>0,$ we let $\det_{\delta}(T)$ be the product of eigenvalues of $|T|$ in the interval $(0,\delta],$ counted with multiplicity. If $\mathcal{H}$ is a Hilbert space, we let $\Ball(\mathcal{H})$ be the closed unit ball of $\mathcal{H}.$

\begin{lemma}\label{L:approximatekernel}  Let $\mathcal{H},\mathcal{K}$ be finite dimensional Hilbert spaces, and $T\in B(\mathcal{H},\mathcal{K})$ be injective. Let $\delta,\varepsilon>0.$  If $4\delta<\varespilon,$ then
\[N_{\varepsilon}(T^{-1}(\delta\Ball(\mathcal{K})),\|\cdot\|_{\mathcal{H}})\leq \Det_{\frac{4\delta}{\varepsilon}}(T)^{-1}.\]
\end{lemma}

\begin{proof} Since
\[\||T|\xi\|=\|T\xi\|,\]
we have
\[T^{-1}(\delta \Ball(\mathcal{K}))=|T|^{-1}(\delta\Ball(\mathcal{H})).\]
The Lemma is now equivalent to Lemma 4.6 in \cite{LiThom}.
\end{proof}

The next Lemma goes back to Rita Solomyak in \cite{RSolomyak}.  It has subsequently been used many times in the computation of entropy of algebraic actions (see e.g. Lemma 3.1 in \cite{Li2}).

\begin{lemma}\label{L:smith} Let $n\in \NN,$ and $T\in M_{n}(\ZZ)\cap GL_{n}(\RR).$ Then
\[|T^{-1}(\ZZ^{n})/\ZZ^{n}|=|\det(T)|.\]

\end{lemma}

We  need some good control on the number of small integers.

\begin{lemma}\label{L:smallintegers}
	We have that
\[\inf_{\varespilon>0}\limsup_{n\to \infty}\frac{1}{n}\log|\ZZ^{n}\cap \varepsilon\Ball(\ell^{1}(n,u_{n}))|=0.\]

\end{lemma}

\begin{proof}

	Let
\[\Omega=\left\{(F_{1},\cdots,F_{\lfloor{n\varepsilon\rfloor}}):F_{j}\in \NN,\sum_{j} jF_{j}\leq n\varespilon\right\},\]
for $F\in \Omega,$ set
\[F_{0}=n-\sum_{j}F_{j}.\]
Given $x\in \ZZ^{n}\cap \varepsilon \Ball(\ell^{1}(n,u_{n})),$ and $1\leq j\leq n\varespilon,j\in \NN,$ set
\[F_{j}(x)=|\{l:|x(l)|=j\}|.\] For $x\in \ZZ^{n}\cap \varepsilon\Ball(\ell^{1}(n,u_{n})),$ we have that $(F_{1}(x),\cdots,F_{\lfloor{n\varespilon\rfloor}}(x))\in \Omega,$ and given $(F_{1},\cdots,F_{\lfloor{n\varepsilon\rfloor}})\in \Omega,$ there are at most
\[2^{n\varepsilon}\frac{n!}{F_{0}!F_{1}!\cdots F_{\lfloor{n\varepsilon\rfloor}}!}\]
possible $x\in \ZZ^{n}\cap \varespilon \Ball(\ell^{1}(n,u_{n}))$ with $F_{j}(x)=F_{j},$ for $j=1,\dots,n.$ Thus
\[|\ZZ^{n}\cap \varepsilon\Ball(\ell^{1}(n,u_{n}))|\leq 2^{n\varepsilon}\sum_{F\in \Omega}\frac{n!}{F_{0}!F_{1}!\cdots F_{\lfloor{n\varepsilon\rfloor}}!}.\]
We have that
\[\Omega\subseteq \prod_{j=1}^{\lfloor{n\varepsilon\rfloor}}\left\{F\in \ZZ:0\leq F\leq \left \lfloor \frac{n\varespilon}{j} \right \rfloor\right\}.\]
So
\begin{align*}
|\Omega|&\leq (n\varespilon+1)\left(\frac{n\varepsilon}{2}+1\right)\left(\frac{n\varespilon}{3}+1\right)\cdots \left(\frac{n\varepsilon}{\lfloor{n\varepsilon\rfloor}}+1\right)\\
&\leq 2^{n\varespilon}(n\varespilon)\left(\frac{n\varepsilon}{2}\right)\left(\frac{n\varespilon}{3}\right)\cdots \left(\frac{n\varepsilon}{\lfloor{n\varepsilon\rfloor}}\right)\\
&\leq 2^{n\varespilon}\frac{(n\varespilon)^{n\varepsilon}}{(\lfloor{n\varepsilon\rfloor})!},
\end{align*}
and by  Stirling's formula  there is some $C>0$ so that
\begin{equation}\label{E:estimatesize}
|\Omega|\leq \frac{Ce^{n\varepsilon}2^{n\varepsilon}}{\sqrt{4\pi n\varepsilon}}.
\end{equation}

	We thus only have to bound
\[\frac{n!}{F_{0}!F_{1}!\cdots F_{\lfloor{n\varepsilon\rfloor}}!}\]
for $F\in \Omega.$ By elementary calculus
\[\frac{1}{k!}\leq k^{-k}e^{k}e^{-1},\]
for all $k\in \NN.$
Thus by Stirling's Formula, there is a $\kappa(n)$ with
\[\lim_{n\to\infty}\frac{1}{n}\log \kappa(n)=0\]
so that
\begin{align}\label{E:boundterm}
\frac{n!}{F_{0}!F_{1}!\cdots F_{\lfloor{n\varepsilon\rfloor}}!}&\leq \kappa(n)n^{n}e^{-\lfloor{n\varepsilon\rfloor}}\prod_{j=0}^{\lfloor{n\varepsilon\rfloor}}F_{j}^{-F_{j}}\\ \nonumber
&\leq \kappa(n)e^{-n\varepsilon+1}\left(\frac{F_{0}}{n}\right)^{-F_{0}}\exp\left(-n\sum_{j=1}^{\lfloor{n\varepsilon\rfloor}}\frac{F_{j}}{n}\log \frac{F_{j}}{n}\right)\\ \nonumber
&\leq \kappa(n)e^{-n\varepsilon+1}(1-\varepsilon)^{-n(1-\varepsilon)}\exp\left(-n\sum_{j=1}^{\lfloor{n\varepsilon\rfloor}}\frac{F_{j}}{n}\log \frac{F_{j}}{n}\right), \nonumber
\end{align}
as $F_{0}\geq n(1-\varepsilon).$ It thus suffices to estimate
\[-\sum_{j=1}^{\lfloor{n\varepsilon\rfloor}}\frac{F_{j}}{n}\log\frac{F_{j}}{n}.\]
It is then enough to estimate the maximum of
\[\phi(x)=-\sum_{j=1}^{\lfloor{n\varepsilon\rfloor}}x_{j}\log x_{j},\]
on
\[D=\left\{x\in \RR^{\lfloor{n\varepsilon\rfloor}}:x_{j}\geq 0,\sum_{j}jx_{j}\leq \varespilon\right\}.\]
It is easy to see that the maximum occurs at a point where
\[\sum_{j}jx_{j}=\varepsilon.\]
Let $x$ be a point where $\phi$ achieves its maximum. By the method of Lagrange multipliers, we see that if $\varespilon<\frac{1}{e},$ then there is a $\lambda>0$ so that
\[-\log x_{j}=\lambda j+1,\]
whenever $x_{j}\ne 0.$ Thus
\begin{align*}
\phi(x)&=-\sum_{j:x_{j}\ne 0}x_{j}\log(x_{j})\\
&=\sum_{j:x_{j}\ne 0}(\lambda j+1)x_{j}\\
&=\lambda\varespilon+\sum_{j}x_{j}\\
&\leq (\lambda+1)\varepsilon.
\end{align*}
But,
\begin{align*}\label{E:boundlambda}
\varespilon&=\sum_{j:x_{j}\ne 0}jx_{j}\\
&=e^{-1}\sum_{j:x_{j}\ne 0}je^{-\lambda j}\\
&\leq e^{-1}\sum_{j=1}^{\infty}je^{-\lambda j}\\
&=e^{-1}\frac{e^{-\lambda}}{(1-e^{-\lambda})^{2}}.
\end{align*}
We need to get an upper bound on $\lambda.$ If $e^{-\lambda}\geq 1/2,$ then we are done. Otherwise,
\[e\varepsilon\leq 4e^{-\lambda},\]
which implies that
\[\lambda\leq \log(4)-1-\log(\varepsilon).\]
So
\[\phi(x)\leq \max\left(\varepsilon\log(4)-\varepsilon\log(\varepsilon),\varepsilon+\varepsilon\log(2)\right).\]
Therefore
\[\max_{F\in \Omega}-\sum_{j=1}^{\lfloor{n\varepsilon\rfloor}}\frac{F_{j}}{n}\log\frac{F_{j}}{n}\leq \max\left(\varepsilon\log(4)-\varepsilon\log(\varepsilon),\varepsilon+\log(2)\varepsilon\right).\]

Now applying $(\ref{E:estimatesize})$ and $(\ref{E:boundterm})$  we see that
\begin{align*}
\frac{1}{n}\log|\ZZ^{n}\cap \varepsilon\Ball(\ell^{1}(n,u_{n}))|&\leq \delta(n)+ \log(2)\varepsilon-(1-\varespilon)\log(1-\varepsilon)\\
&+\max\left(\varepsilon\log(4)-\varespilon\log(\varepsilon),\varepsilon+2\log(2)\varepsilon\right),
\end{align*}
where
\[\lim_{n\to \infty}\delta(n)=0.\]
This estimate is good enough to prove the Lemma.

\end{proof}

	We are now ready to evaluate the  topological entropy of $\Gamma\actson X_{f},$ when $f\in M_{n}(\ZZ(\Gamma)).$

\begin{theorem}\label{T:squarematrix} Let $\Gamma$ be a countable discrete sofic group with sofic approximation $\Sigma.$ Let $f\in M_{n}(\ZZ(\Gamma))$ and suppose that $\lambda(f)$ is injective. Then
\[h_{\Sigma}(X_{f},\Gamma)=\log \Det_{L(\Gamma)}(f).\]
\end{theorem}

\begin{proof}

	We will apply Proposition \ref{P:mainreduction}. In order to use Lemma \ref{L:smith} we need to have good control over the kernel of $\sigma_{i}(f),$ for this we perturb $\sigma_{i}(f)$ slightly. By Proposition \ref{P:rankperturabtion} we may find a rank perturbation $x_{i}$ of $\sigma_{i}(f)$ with $x_{i}\in GL_{d_{i}}(M_{n}(\RR)).$ By Proposition \ref{P:rankperturabtion} we have
\[\sup_{\xi\in (\RR^{d_{i}})^{\oplus n}}\|x_{i}\xi-\sigma_{i}(f)\xi\|_{2,(\ZZ^{d_{i}})^{\oplus n}}\to 0.\]
 So
\[h_{\Sigma}(X_{f},\Gamma)=\sup_{\varepsilon>0}\inf_{\delta>0}\limsup_{i\to \infty}\frac{\log S_{\varespilon}(\Xi_{\delta}(x_{i}),\theta_{2,(\ZZ^{d_{i}})^{\oplus n}})}{d_{i}}.\]
Let $M=\sup_{i}\|x_{i}\|_{\infty}.$

	 Let
\[\mathcal{N}\subseteq x_{i}^{-1}((\ZZ^{d_{i}})^{\oplus n})\cap (\RR^{d_{i}})^{\oplus n}\]
be a section for the quotient map
\[ x_{i}^{-1}((\ZZ^{d_{i}})^{\oplus n})\to \frac{x_{i}^{-1}((\ZZ^{d_{i}})^{\oplus n})\cap (\RR^{d_{i}})^{\oplus n}}{(\ZZ^{d_{i}})^{\oplus n}}.\]
 By Lemma \ref{L:smith} we have
\begin{equation}\label{E:count1sec4}
|\mathcal{N}|=|\Det(x_{i})|.
\end{equation}
Let
\[\mathcal{M}\subseteq x_{i}^{-1}(\delta \Ball(\ell^{2}_{\RR}(d_{i}n,u_{d_{i}n})))\cap (\RR^{d_{i}})^{\oplus n}\]
be a maximal $\varespilon$-separated subset. Then by Lemma \ref{L:approximatekernel}, we know that if $4\delta<\varespilon,$ then
\begin{equation}\label{E:count2sec4}
|\mathcal{M}|\leq \Det_{4\delta/\varespilon}(x_{i})^{-1}.
\end{equation}

	Now let $\xi\in \Xi_{\delta}(x_{i}).$ Perturbing $\xi$ by an integer point we may assume that $\|\xi\|_{2}\leq 1.$
As $\xi\in \Xi_{\delta}(x_{i}),$  we can find some $l\in (\ZZ^{d_{i}})^{\oplus n}$ so that
\[\|x_{i}\xi-l\|_{2}<\delta.\]
 Let $\zeta\in \mathcal{N},k\in (\ZZ^{d_{i}})^{\oplus n}$ be such that
\[x_{i}\zeta+x_{i}k=l.\]
Then
\[\|x_{i}(\xi-\zeta-k)\|_{2}<\delta.\]
Hence, we may find some $\eta\in \mathcal{M}$ so that
\[\|\xi-\zeta-\eta-k\|_{2}\leq\varepsilon.\]
So
\[\|\xi-\zeta-\eta\|_{2,(\ZZ^{d_{i}})^{\oplus n}}\leq 2\varespilon.\]
Hence, by inequalities (\ref{E:count1sec4}),(\ref{E:count2sec4}), we know that
\[\frac{1}{d_{i}}\log S_{2\varespilon}(\Xi_{\delta}(x_{i}),\theta_{2,(\ZZ^{d_{i}})^{\oplus n}})\leq \frac{1}{d_{i}}\log(|\mathcal{M}||\mathcal{N}|)\leq \int_{[4\delta/\varespilon,\infty)}\log(t)\,d\mu_{|x_{i}|}(t).\]
Since $\mu_{|x_{i}|}\to \mu_{|f|}$ weak$^{*},$
\[\limsup_{i\to \infty}\frac{1}{d_{i}}\log S_{2\varespilon}(\Xi_{\delta}(x_{i}),\theta_{2,(\ZZ^{d_{i}})^{\oplus n}})\leq \int_{\left(4 \delta/\varepsilon,\infty\right)}\log(t)\,d\mu_{|f|}(t).\]
Letting $\delta\to 0$ and applying the Monotone convergence theorem, we find that
\[h_{\Sigma}(X_{f},\Gamma)\leq \log \Det_{L(\Gamma)}(f).\]

	We now turn to the proof of the lower bound. For this, fix $\varespilon,\delta>0.$  Let $\eta>0$ depend upon $\varepsilon$ in a manner to be determined shortly. Let
\[\mathcal{N}\subseteq x_{i}^{-1}((\ZZ^{d_{i}})^{\oplus n})\cap (\RR^{d_{i}})^{\oplus n}\]
be such that $\{x_{i}\xi\}_{\xi\in \mathcal{N}}$ is a maximal $\eta$-separated family with respect to $\theta_{2,x_{i}((\ZZ^{d_{i}})^{\oplus n})}.$
Let $p=\chi_{(0,\delta/\varespilon]}(|x_{i}|),$ set $W=p(\RR^{d_{i}})^{\oplus n},$ and let
\[\mathcal{M}\subseteq W\cap x_{i}^{-1}(\delta \Ball(\ell^{2}(d_{i}n,u_{d_{i}n})))\]
be a maximal $\varepsilon$-separated subset with respect to $\theta_{2,(\ZZ^{d_{i}})^{\oplus n}}.$

	Suppose $\xi_{1},\xi_{2}\in \mathcal{N},\zeta_{1},\zeta_{2}\in \mathcal{M}$ and
\begin{equation}\label{E:separated}
\|\xi_{1}+\zeta_{1}-\xi_{2}-\zeta_{2}\|_{2,(\ZZ^{d_{i}})^{\oplus n}}\leq \varespilon,
\end{equation}
then
\[\|x_{i}(\xi_{1}-\xi_{2})\|_{2,x_{i}(\ZZ^{d_{i}})^{\oplus n}}\leq \varespilon M+2\delta.\]
Hence if we choose $\eta=2\varespilon M,$ and $\delta$ is sufficiently small we find that
$\xi_{1}=\xi_{2}$
by our choice of $\mathcal{N}.$ Then by inequality $(\ref{E:separated}),$ and our choice of $\mathcal{M},$ we find that $\zeta_{1}=\zeta_{2}.$ Therefore,
\begin{equation}\label{E:gotcha}
N_{\varespilon}(\Xi_{\delta}(x_{i}),\theta_{2,(\ZZ^{d_{i}})^{\oplus n}})\geq |\mathcal{N}||\mathcal{M}|.
\end{equation}

	We now have to get a lower bound on $|\mathcal{M}|,|\mathcal{N}|.$ For this, set for $r>0$
\[\omega_{n}(r)=|\ZZ^{n}\cap r \Ball(\ell^{1}(n,u_{n}))|.\]
For all $\xi\in x_{i}^{-1}((\ZZ^{d_{i}}))^{\oplus n}\cap(\RR^{d_{i}})^{\oplus n},$ there exists a $\zeta\in \mathcal{N},k\in(\ZZ^{d_{i}})^{\oplus n}$ such that
\[\|x_{i}\xi-x_{i}\zeta-x_{i}k\|_{2}\leq 2\varespilon M.\]
Hence there exists  a section $S\subseteq x_{i}^{-1}((\ZZ^{d_{i}})^{\oplus n})\cap (\RR^{d_{i}})^{\oplus n}$  of the quotient map
\[x_{i}^{-1}((\ZZ^{d_{i}})^{\oplus n})\cap (\RR^{d_{i}})^{\oplus n}\to \frac{x_{i}^{-1}((\ZZ^{d_{i}})^{\oplus n})\cap (\RR^{d_{i}})^{\oplus n}}{(\ZZ^{d_{i}})^{\oplus n}}\]
such that for all $\xi\in S,$ there is a $\zeta\in \mathcal{N}$ with
\[\|x_{i}(\xi-\zeta)\|_{2}\leq 2\varespilon M.\]
From this, it is easy to see that
\begin{equation}\label{E:lowerboundinverseimage}
|\det(x_{i})|=|S|\leq \omega_{d_{i}n}\left(2\varepsilon M \right)|\mathcal{N}|.
\end{equation}

	To bound $|\mathcal{M}|$ note that for all $\xi\in W\cap x_{i}^{-1}(\delta\Ball(\ell^{2}(d_{i}n,u_{d_{i}n})),$ there is a $\zeta\in\mathcal{M},k\in (\ZZ^{d_{i}})^{\oplus n}$ so that
\[\|\xi-\zeta-k\|_{2}\leq \varepsilon.\]
This implies that
\begin{equation}\label{E:smallimageinteger}
\|x_{i}k\|_{2}\leq \varespilon M+2\delta,
\end{equation}
and
\begin{equation}\label{E:projectionfact}
\|\xi-p\zeta-pk\|_{2}\leq \varepsilon.
\end{equation}
Let
\[\mathcal{T}\subseteq (\ZZ^{d_{i}})^{n}\cap x_{i}^{-1}((\varepsilon M+2\delta)\Ball(\ell^{2}(d_{i}n,u_{d_{i}n})))\]
be a section of the map
\[(\ZZ^{d_{i}})^{n}\cap x_{i}^{-1}((\varepsilon M+2\delta)\Ball(\ell^{2}(d_{i}n,u_{d_{i}n})))\to p[(\ZZ^{d_{i}})^{n}\cap x_{i}^{-1}((\varepsilon M+2\delta)\Ball(\ell^{2}(d_{i}n,u_{d_{i}n})))],\]
given by multiplication by $p.$ As $x_{i}\big|_{\mathcal{T}}$ is injective, we have
\[|\mathcal{T}|\leq \omega_{d_{i}n}(\varepsilon M+2\delta)\]
and by $(\ref{E:projectionfact}),$
\[W\cap x_{i}^{-1}(\delta\Ball(\ell^{2}(d_{i}n,u_{d_{i}n}))\subseteq\bigcup_{\substack{\zeta \in \mathcal{M},\\ l\in \mathcal{T}}}p\zeta+pl+\varespilon\Ball(W,\|\cdot\|_{2}).\]
Computing volumes,
\begin{equation}\label{E:lowerapproxkernelbound2}
\Det_{\delta/\varespilon}(x_{i})^{-1}\delta^{\Tr(p)}\leq |\mathcal{M}|\omega_{d_{i}n}(\varepsilon M+2\delta)\varespilon^{\Tr(p)}.
\end{equation}

	Applying $(\ref{E:lowerboundinverseimage}),(\ref{E:lowerapproxkernelbound2}),(\ref{E:gotcha}),$ we see that
\begin{align*}
\frac{1}{d_{i}}\log N_{\varepsilon}(\Xi_{\delta}(x_{i}),\theta_{2,(\ZZ^{d_{i}})^{\oplus n}})&\geq \int_{(\delta/\varepsilon,\infty)}\log(t)\,d\mu_{|x_{i}|}(t)+\log(\delta/\varepsilon)\mu_{|x_{i}|}((0,\delta/\varespilon])\\
&-\frac{1}{d_{i}}\log \omega_{d_{i}n}(\varepsilon M+2\delta)-\frac{1}{d_{i}}\log\omega_{d_{i}n}\left(2\varepsilon M\right).
\end{align*}
Using Lemma \ref{L:smallintegers} and that $\mu_{|x_{i}|}\to \mu_{|f|}$ weak$^{*},$
\[h_{\Sigma}(X_{f},\Gamma)\geq \log \Det_{L(\Gamma)}(f)+\sup_{\varespilon>0}\inf_{\delta>0}\log(\delta/\varepsilon)\mu_{|f|}((0,\delta/\varespilon)).\]
But,
\[-\log(\delta/\varepsilon)\mu_{|f|}((0,\delta/\varespilon))\leq \int_{(0,\delta/\varepsilon)}-\log(t)\,d\mu_{|f|}(t)\to 0,\]
as $\delta \to 0,$ by the dominated convergence theorem and Corollary \ref{C:integrability}.

\end{proof}

\subsection{The General Upper Bound For $h_{\Sigma}(X_{f},\Gamma)$ By Ultraproduct Techniques}

	We will now establish that for $f\in M_{m,n}(\ZZ(\Gamma))$ with $\lambda(f)$ injective,
\begin{equation}\label{E:resultofthissection}
h_{\Sigma}(X_{f},\Gamma)\leq \log \Det^{+}_{L(\Gamma)}(f).
\end{equation}
 We could have given a proof of the upper bound in Theorem \ref{T:squarematrix}, but we decided to postpone the proof until now, as the operator algebra machinery involved in the proof of (\ref{E:resultofthissection}) is considerably more technical than in the proof of Theorem \ref{T:squarematrix}.

\begin{definition}\emph{ Let $\mathcal{H}$ be a Hilbert space. A} von Neumann algebra \emph{is a weak operator topology closed, unital, subalgebra of $B(\mathcal{H})$ which is closed under taking adjoints. A} tracial von Neumann algebra \emph{is a pair $(M,\tau)$ where $M$ is a von Neumann algebra, and $\tau\colon M\to \CC$ is a linear functional such that}\end{definition}

\begin{list}{ \arabic{pcounter}:~}{\usecounter{pcounter}}
\item $\tau(1)=1,$\\
\item $\tau(x^{*}x)\geq 0,$  with equality if and only if $x=0$,\\
\item $\tau(xy)=\tau(yx),$  for all $x,y\in M,$
\item $\tau\big|_{\{x\in M:\|x\|_{\infty}\leq 1\}}$ is weak operator topology continuous.
\end{list}

As before $\|x\|_{\infty}$ is the operator norm of $x.$ The pairs $(L(\Gamma),\tau),$ and $(M_{d_{i}}(\CC),\tr_{d_{i}})$ are the most natural examples for our purposes. The definition of $\Tr\otimes\tau,$  spectral measure, and Fuglede-Kadison determinant as in Section \ref{S:spectralmeasure} work for elements in $M_{m,n}(M).$  For $x\in M_{m,n}(M),$ we use
\[\|x\|_{2}=(\Tr\otimes \tau(x^{*}x))^{1/2}.\]

\begin{definition}\emph{ Let $(M_{n},\tau_{n})$ be a sequence of tracial von Neumann algebras, and let $\omega\in \beta\NN\setminus \NN$ be a free ultrafilter. Set}
\[M=\frac{\{(x_{n})_{n=1}^{\infty}:x_{n}\in M_{n},\sup_{n}\|x_{n}\|_{\infty}<\infty\}}{\{(x_{n}):x_{n}\in M_{n},\sup_{n}\|x_{n}\|_{\infty}<\infty,\lim_{n\to \omega}\|x_{n}\|_{2}=0\}}.\]
\emph{If $x_{n}\in M_{n},$ and $\sup_{n}\|x_{n}\|_{\infty}<\infty,$ we use $(x_{n})_{n\to \omega}$ for the image in $M$ of the sequence $(x_{n})_{n=1}^{\infty}$ under the quotient map. Let}
\[\tau_{\omega}\colon M\to \CC\]
\emph{be given by}
\[\tau_{\omega}((x_{n})_{n\to \omega})=\lim_{n\to \omega}\tau_{n}(x_{n}).\]
\emph{ We call the pair $(M,\tau_{\omega})$ the} tracial ultraproduct of $(M_{n},\tau_{n})$ \emph{and we will denote it by}
\[\prod_{n\to \omega}(M_{n},\tau_{n}).\]
\end{definition}

	A remark about the definition, it is not hard to show that $M$ is a $*$-algebra with the operations being given coordinate-wise. We clearly have a inner product on $M$ given by
\[\ip{x,y}=\tau_{\omega}(y^{*}x),\]
let $L^{2}(M,\tau_{\omega})$ be Hilbert space completion of $M$ under this inner product. We have a representation
\[\lambda\colon M\to B(L^{2}(M,\tau_{\omega}))\]
given by left multiplication. It turns out that this representation is faithful, and that $\lambda(M)$ is weak-operator topology closed in $B(L^{2}(M,\tau_{\omega})),$ (see \cite{BO} Lemma A.9) so we may regard $(M,\tau_{\omega})$ as a tracial von Neumann algebra.

	Here is the main example of relevance for us. Let $\Gamma$ be a countable discrete sofic group with sofic approximation $\Sigma=(\sigma_{i}\colon \Gamma\to S_{d_{i}}),$ and let $\omega\in \beta\NN\setminus \NN$ be a free ultrafilter. Then we have a trace-preserving injective $*$-homomorphism (i.e. preserving multiplication and adjoints)
\[\sigma\colon \CC(\Gamma)\to \prod_{i\to \omega}(M_{d_{i}}(\CC),\tr_{d_{i}})\]
given by
\[\sigma(f)=(\sigma_{i}(f))_{i\to \omega}.\]
It turns out that because $\CC(\Gamma)$ is weak operator topology dense,  this embedding extends uniquely to a trace-preserving injective $*$-homomorphism
\[\sigma\colon L(\Gamma)\to \prod_{i\to \omega}(M_{d_{i}}(\CC),\tr_{d_{i}}).\]

\begin{lemma} Let $\Gamma$ be a countable discrete sofic group with sofic approximation $\sigma_{i}\colon \Gamma\to S_{d_{i}},$ and let $f\in M_{m,n}(\ZZ(\Gamma))$ with $\lambda(f)$ injective. Let $A_{i}\subseteq \{1,\cdots,d_{i}\}^{m},B_{i}\subseteq \{1,\cdots,d_{i}\}^{n}$ be subsets so that
\[P_{\im(\sigma_{i}(f))^{\perp}}\big|_{\RR^{A_{i}^{c}}},P_{\ker(\sigma_{i}(f))^{\perp}}\big|_{\RR^{B_{i}}}\]
are isomorphisms onto $\im(\sigma_{i}(f))^{\perp},\ker(\sigma_{i}(f))^{\perp}.$ Set $x_{i}=\chi_{A_{i}}\sigma_{i}(f)\chi_{B_{i}},$ then
\[\inf_{0<\delta<1}\limsup_{i\to \infty}\int_{(\delta,\infty)}\log(t)\,d\mu_{|x_{i}|}(t)\leq \log\Det^{+}_{L(\Gamma)}(f).\]
\end{lemma}

\begin{proof} We claim that it suffices to show that for all $\omega\in \beta\NN\setminus\NN,$
\begin{equation}\label{E:ultrafilter}
\inf_{0<\delta<1}\lim_{i\to \omega}\int_{(\delta,\infty)}\log(t)\,d\mu_{|x_{i}|}(t)\leq \log\Det^{+}_{L(\Gamma)}(f).\end{equation}
Suppose we can show $(\ref{E:ultrafilter}),$ but that
\[\alpha=\inf_{0<\delta<1}\limsup_{i\to\infty}\int_{(\delta,\infty)}\log(t)\,d\mu_{|x_{i}|}(t)>\log\Det^{+}_{L(\Gamma)}(f).\]
Let $\delta_{n}<1$ be a decreasing sequence of positive real numbers converging to zero. Choose a strictly increasing sequence of natural numbers $i_{n}$ so that
\[\int_{(\delta_{n},\infty)}\log(t)\,d\mu_{|x_{i_{n}}|}(t)\geq \alpha-2^{-n}.\]
Let $\omega\in \beta\NN\setminus\NN$ be such that  $\{i_{n}:n\in \NN\}\in\omega.$ Note that if $m\geq n,$ then
\begin{align*}
\alpha-2^{-m}&\leq \int_{(\delta_{m},\infty)}\log(t)\,d\mu_{|x_{i_{m}}|}(t)\\
&=\int_{(\delta_{n},\infty)}\log(t)\,d\mu_{|x_{i_{m}}|}(t)+\int_{(\delta_{m},\delta_{n}]}\log(t)\,d\mu_{|x_{i_{m}}|}(t)\\
&\leq \int_{(\delta_{n},\infty)}\log(t)\,d\mu_{|x_{i_{m}}|}(t).
\end{align*}
Hence,
\[\lim_{i\to \omega}\int_{(\delta_{n},\infty)}\log(t)\,d\mu_{|x_{i}|}(t)\geq \alpha.\]
So
\[\inf_{0<\delta<1}\lim_{i\to \omega}\int_{(\delta,\infty)}\log(t)\,d\mu_{|x_{i}|}(t)=\lim_{n\to \infty}\lim_{i\to \omega}\int_{(\delta_{n},\infty)}\log(t)\,d\mu_{|x_{i}|}(t)\geq \alpha>\log \Det^{+}_{L(\Gamma)}(f),\]
and this contradicts $(\ref{E:ultrafilter}).$

	We now proceed to prove $(\ref{E:ultrafilter}).$ Fix $\omega\in \beta\NN\setminus\NN,$ and let
\[(M,\tau)=\prod_{i\to \omega}(M_{d_{i}}(\CC),\tr_{d_{i}}),\]
let
\[\sigma\colon L(\Gamma)\to M\]
be defined as before the statement of the Lemma, and extend by the usual methods to a map
\[\sigma\colon M_{m,n}(L(\Gamma))\to M_{m,n}(M).\]
Let
\[x=(x_{i})_{i\to \omega}\in M_{m,n}(M).\]
First note that for every $\phi\in C(\RR),$
\begin{equation}\label{E:weakultrafilter}
\int \phi\,d\mu_{|x|}(t)=\Tr\otimes \tau_{\omega}(\phi(|x|))=\lim_{i\to \omega}\Tr\otimes \tr_{d_{i}}(\phi(|x_{i}|))=\lim_{i\to \omega}\int \phi\,d\mu_{|x_{i}|}(t).
\end{equation}
By Lemma $\ref{L:injective},$
\[\frac{|B_{i}^{c}|}{d_{i}}=\frac{\dim(\ker(\sigma_{i}(f)))}{d_{i}}\to 0\]
as $i\to \infty,$ so
\[x=(\chi_{A_{i}}\sigma_{i}(f))_{i\to \omega}.\]
Thus
\begin{equation}\label{E:determinantestimate1}
0\leq x^{*}x\leq \sigma(f)^{*}\sigma(f)\in M_{n}(M).
\end{equation}
By operator monotonicity of logarithms and the Monotone Convergence Theorem we have:
\begin{align*}
\log\Det_{M}^{+}(x)=\frac{1}{2}\lim_{\varepsilon\to 0}\tau_{\omega}(\log(x^{*}x+\varepsilon))&\leq \frac{1}{2}\lim_{\varespilon\to 0}\tau_{\omega}(\log(\sigma(f)^{*}\sigma(f)+\varepsilon))\\
&=\frac{1}{2}\log\Det_{M}^{+}(\sigma(f)^{*}\sigma(f))\\
&=\frac{1}{2}\log \Det_{L(\Gamma)}^{+}(f^{*}f)\\
&=\log\Det_{L(\Gamma)}^{+}(|f|)\\
&=\log\Det_{L(\Gamma)}^{+}(f),
\end{align*}
here we have used that the inclusion $L(\Gamma)\to M$ is trace-preserving. By (\ref{E:weakultrafilter}), we have
\[\log \Det_{M}^{+}(x)=\inf_{0<\delta<1}\lim_{i\to \omega}\int_{(\delta,\infty)}\log(t)\,d\mu_{|x_{i}|}(t),\]
so we have proved (\ref{E:ultrafilter}).

\end{proof}

\begin{theorem}\label{T:upperbound} Let $\Gamma$ be a countable discrete sofic group with sofic approximation $\Sigma.$ Let $f\in M_{m,n}(\ZZ(\Gamma))$ be injective as an operator on $\ell^{2}(\Gamma)^{\oplus n},$ then
\[h_{\Sigma}(X_{f},\Gamma)\leq \log \Det^{+}_{L(\Gamma)}(f).\]

\end{theorem}

\begin{proof} Let $x_{i}$ be defined as in the preceding Lemma. It follows as in Proposition \ref{P:rankperturabtion} that
\[\ker(x_{i})\cap (\RR^{d_{i}})^{\oplus n}=\RR^{B_{i}^{c}},\]
\[x_{i}((\RR^{d_{i}})^{\oplus n})=\RR^{A_{i}},\]
\[\frac{|B_{i}|}{d_{i}}\to 0.\]

	Let $\xi\in \Xi_{\delta}(\sigma_{i}(f)).$ We may assume that $\|\xi\|_{2}\leq 1.$ As
\[\frac{|B_{i}^{c}|}{d_{i}}\to 0,\]
it follows that
\[\sup_{\zeta\in (\RR^{d_{i}})^{\oplus n}}\|\sigma_{i}(f)\chi_{B_{i}}\zeta-\sigma_{i}(f)\zeta\|_{2,(\ZZ^{d_{i}})^{\oplus m}}\to 0.\]
So for all large $i,$ we have
\[\|\sigma_{i}(f)\chi_{B_{i}}\xi\|_{2,(\ZZ^{d_{i}})^{\oplus m}}\leq 2\delta.\]
So we can find an $l\in (\ZZ^{d_{i}})^{\oplus m}$ so that
\[\|\sigma_{i}(f)\chi_{B_{i}}\xi-l\|_{2}\leq 2\delta.\]
Thus,
\begin{equation}\label{E:aorhakdngoa}
\|x_{i}\xi-\chi_{A_{i}}l\|_{2}\leq 2\delta.
\end{equation}
 Let
\[\mathcal{N}\subseteq x_{i}^{-1}(\ZZ^{A_{i}})\cap \RR^{B_{i}}\]
be a section for the quotient map
\[ x_{i}^{-1}(\ZZ^{A_{i}})\to \frac{x_{i}^{-1}((\ZZ^{A_{i}})\cap \RR^{B_{i}}}{\ZZ^{B_{i}}}.\]
 By Lemma \ref{L:smith} we have
\begin{equation}\label{E:count1}
|\mathcal{N}|=|\Det^{+}(x_{i})|.
\end{equation}
Let
\[\mathcal{M}\subseteq x_{i}^{-1}(2\delta \Ball(\ell^{2}_{\RR}(d_{i}m,u_{d_{i}m})))\cap \RR^{B_{i}},\]
be a maximal $\varespilon$-separated subset. By Lemma \ref{L:approximatekernel} we know that if $8\delta<\varespilon,$ then
\begin{equation}\label{E:count2}
|\mathcal{M}|\leq \Det_{8\delta/\varespilon}(x_{i})^{-1}.
\end{equation}
Since
\[\|\xi-\chi_{B_{i}}\xi\|_{2,(\ZZ^{d_{i}})^{\oplus n}}\leq \frac{|B_{i}^{c}|}{d_{i}}\to 0,\]
\[x_{i}((\RR^{d_{i}})^{\oplus n})=\RR^{A_{i}}\]
inequalities (\ref{E:aorhakdngoa}),(\ref{E:count1}),(\ref{E:count2}) allows us to follow the first half of the proof of Theorem \ref{T:squarematrix} to see that
\[S_{8\varepsilon}(\Xi_{\delta}(\sigma_{i}(f))\leq |\mathcal{M}||\mathcal{N}|.\]
It now follows as in the proof of the first half of Theorem \ref{T:squarematrix} that
\[h_{\Sigma}(X_{f},\Gamma)\leq \inf_{0<\delta<1}\limsup_{i\to \infty}\int_{(\delta,\infty)}\log(t)\,d\mu_{|x_{i}|}(t),\]
and so the Theorem follows automatically from the preceding Lemma.

\end{proof}

\section{Fuglede-Kadison Determinants and  Measure-Theoretic Entropy}\label{S:measureentropy}

If $X$ is a compact  group, we use $m_{X}$ for the Haar measure on $X.$  The following is the main result of this section.

\begin{theorem}\label{T:invertibleinvNA} Let $\Gamma$ be a countable discrete non-amenable sofic group with sofic approximation $\Sigma.$ Let $f\in M_{n}(\ZZ(\Gamma))$ be injective as a left multiplication operator on $\ell^{2}(\Gamma)^{\oplus n}.$  Then,
\[h_{\Sigma,m_{X_{f}}}(X_{f},\Gamma)=\log \Det_{L(\Gamma)}(f).\]
\end{theorem}

Our techniques will be general enough to be adaptable to a slightly different situation. Namely, suppose that $f\in M_{m,n}(\ZZ(\Gamma))$ and that $\lambda(f)$ has dense image as an operator on $\ell^{2}(\Gamma)^{\oplus n},$ and that $m\ne n$ (it necessarily follows that $m<n$ by Proposition \ref{P:Left/RightIssues}) we will be able to show that the measure-theoretic entropy of $\Gamma\actson (X_{f},m_{X_{f}})$ is infinite.

	We focus solely on the non-amenable case, the amenable case is covered by the results in \cite{LiThom}. Let us recall the definition of measure-theoretic entropy in the case of a topological model.

\begin{definition}\emph{ Let $X$ be a compact metrizable space and $\Gamma \actson X$ by homeomorphisms. Let $\mu$ be a Borel probability measure on $X$ preserved by $\Gamma.$ Let $\rho$ be a dynamically generating pseudometric on $X.$ For $F\subseteq \Gamma$ finite, $L\subseteq C(X)$ finite, and $\delta>0,$ we let $\Map(\rho,F,L,\delta,\sigma_{i})$ be set of all $\phi \in \Map(\rho,F,\delta,\sigma_{i})$ so that}
\[\left|\frac{1}{d_{i}}\sum_{j=1}^{d_{i}}f(\phi(j))-\int f\,d\mu\right|<\delta.\]
\emph{Define the measure-theoretic entropy of $\Gamma\actson (X,\mu)$ by}
\[h_{\Sigma,\mu}(\rho,F,L,\delta,\varepsilon)=\limsup_{i\to \infty}\frac{1}{d_{i}}\log S_{\varespilon}(\Map(\rho,F,L,\delta,\sigma_{i}),\rho_{2}),\]
\[h_{\Sigma,\mu}(\rho,\varepsilon)=\inf_{\substack{F\subseteq \Gamma \mbox{\emph{ finite, } }\\ L\subseteq C(X) \mbox{\emph{ finite, } }\\ \delta>0}}h_{\Sigma,\mu}(\rho,F,L,\delta,\varepsilon)\]
\[h_{\Sigma,\mu}(X,\Gamma)=\sup_{\varespilon>0}h_{\Sigma,\mu}(\rho,\varepsilon).\]
\end{definition}
By Proposition 3.4 in \cite{KLi2}, and Proposition 5.4, we know measure-theoretic entropy does not depend upon the pseudometric and if $\Gamma\actson (X,\mu),\Gamma\actson (Y,\nu)$ are probability measure preserving actions, and $\phi\colon X\to Y$ is a bimeasurable bijection  such that $\phi_{*}\mu=\nu,$ then $h_{\Sigma,\mu}(X,\Gamma)=h_{\Sigma,\nu}(Y,\Gamma).$

\subsection{Main Technical Lemmas}

	We begin to collect a few technical lemmas needed for the proof. The idea of the proof of Theorem \ref{T:invertibleinvNA} is that we can already produce enough microstates for the topological action $\Gamma\actson X_{f}$ to get the lower bound on topological entropy. So if we can prove that ``most" of these are microstates for the measure-preserving action $\Gamma\actson (X_{f},m_{X_{f}}),$ this will be enough to get the lower bound on measure-theoretic entropy. Since topological entropy always dominates measure-theoretic entropy this will prove the main theorem. To ease the work involved in this probabilistic argument, it will be helpful to prove a ``concentration'' result (albeit a soft one) which will essentially reduce our work to prove that our microstates approximately pushforward  the uniform measure to the Haar measure ``on average''. Similar techniques have been used by Lewis Bowen (see the proof of Theorem 4.1 in \cite{BowenEntropy}), as well as Lewis Bowen and Hanfeng Li (see \cite{BowenLi} Lemma 7.3).  However, we wish to first formulate the technique in a more abstract setting. This setting is close to that of Lemma 6.1 in \cite{DabBi}. In fact, the following Lemma may be regarded as a mild generalization of Lemma 6.1 in \cite{DabBi}.
	
		We first need to recall some facts about integrals of vector valued functions in a locally convex space. Suppose that $X$ is a separable, locally convex space, and that $K\subseteq X$ is a compact, metrizable, convex set. If $\mu\in \Prob(K),$ then there is a unique point $p\in K$ so that for all $\phi\in X^{*}$
\[p(x)=\int_{K}\phi(x)\,d\mu(x),\]
see \cite{RudFA} Theorem 3.27. We write
\[p=\int_{K}x\,d\mu(x),\]
the point $p$ is called the \emph{barycenter} of $\mu.$ The gist of the following lemma is that if a probability measure on such a compact, convex set has a barycenter which is close to an extreme point, then ``most'' of the mass of the measure is concentrated near the extreme point.

\begin{lemma}[Abstract Automatic Concentration]\label{L:AbstractConc} Let $X$ be a separable, locally convex space, let $K\subseteq C$ be  compact, metrizable, convex subsets of $X$ and $p$ an extreme point of $K.$ Then for any  open neighborhood $U$ of $0$ in $X$ and $\varepsilon>0,$ there is a  neighborhood $V$ in $X$ of $0$ so that if $\mu\in \Prob(C)$ and
\[\mu((K+V)\cap C)=1,\]
\[p-\int_{C}x\,d\mu(x)\in V,\]
then
\[\mu(\{x\in C:x-p\in U\})\geq 1-\varepsilon.\]
\end{lemma}

\begin{proof}   Fix an open neighborhood $U$ of $0$ in $X.$ By metrizability, we may find a decreasing sequence $V_{n}$  of open neighborhoods of $0$ with
\[K=\bigcap_{n=1}^{\infty}(K+V_{n})\cap C.\]
Assuming the lemma is false for this $U,$ we can find an $\varespilon>0$ and a sequence $\mu_{n}$ of Borel probability measures on $C$ so that
\[\mu_{n}((K+V_{n})\cap C)=1,\]
\[p-\int_{C}\,x\,d\mu_{n}(x)\in V_{n},\]
but
\[\mu_{n}(\{x\in C:x-p\in U\})\leq 1-\varepsilon.\]
We may assume, by passing to a subsequence, that there is a $\mu\in \Prob(C)$ with $\mu_{n}\to \mu$ in the weak$^{*}$-topology. It is easy to see that $\int_{C}\,x\,d\mu_{n}(x)\to p,$ and thus for every $\phi\in X^{*}$ we have
\[\phi(p)=\lim_{n\to\infty}\int_{C}\phi(x)\,d\mu_{n}(x)=\int_{C}\phi(x)\,d\mu(x),\]
so
\[p=\int_{C}x\,d\mu(x).\]
Suppose that $W$ is a neighborhood of $K$ in $C$ and choose a neighborhood $W_{0}$ of $K$ in $C$ with $C\subseteq \overline{W_{0}}\subseteq W.$ Since $\mu_{n}\to \mu$ weak$^{*},$ we have
\[\mu(W)\geq \mu(\overline{W_{0}})\geq \limsup_{n\to\infty}\mu_{n}(\overline{W_{0}})=1,\]
since for all sufficiently large $n$ it is true that $(K+V_{n})\cap C\subseteq W_{0}.$  Infimizing over all neighborhoods $W$ of $K$ we see that $\mu(K)=1$ and thus
\[p=\int_{K}x\,d\mu(x).\]
By extremality of $p$ we find that $\mu=\delta_{p}.$ Since $\mu_{n}\to \delta_{p}$ weak$^{*}$ we must have that
\[1=\delta_{p}(p+U)\leq \liminf_{n\to\infty}\mu_{n}(p+U)\leq 1-\varespilon,\]
a contradiction.

\end{proof}
	
		We use the Lemma to state a more technical version of a concentration Lemma, which is more specific to our situation.

\begin{lemma}[The Automatic Concentration Lemma]\label{L:AutomaticConcentration} Let $\Gamma$ be a countable discrete sofic group with sofic approximation $\Sigma.$ Let $X$ be a compact metrizable space and $\Gamma\actson X$ by homeomorphisms, let $\rho$ be a dynamically generating pseudometric on $X.$ Let $I$ be a directed set, and  let $(\Omega_{i,\alpha},\PP_{i,\alpha})_{i\in \NN,\alpha\in I}$ be standard probability spaces. Let
\[\Phi_{i,\alpha}\colon \{1,\cdots,d_{i}\}\times \Omega_{\alpha,i}\to X\]
be Borel measurable maps, and for $\xi\in \Omega_{\alpha,i}$ define
\[\phi_{\xi}\colon \{1,\cdots,d_{i}\}\to X\]
by
\[\phi_{\xi}(j)=\Phi_{i,\alpha}(j,\xi).\]
Suppose that for all $g\in \Gamma,$
\[\lim_{\alpha}\limsup_{i\to \infty}\|\rho_{2}(\phi_{\xi}\circ \sigma_{i}(g),g\phi_{\xi})\|_{L^{\infty}(\xi)}=0,\]
and that there is a Borel probability measures $\mu$ on $X$ so that for all $f\in C(X)$
\[\lim_{\alpha}\limsup_{i\to\infty}\left|\int f\,d(\Phi_{i,\alpha})_{*}(u_{d_{i}}\otimes \PP_{i,\alpha})-\int f\,d\mu\right|=0.\]
Then

(a) $\mu$ is $\Gamma$-invariant,

(b) if $\Gamma\actson (X,\mu)$ is ergodic, then for all $F\subseteq \Gamma$ finite, $L\subseteq C(X)$ finite, $\delta>0$ we have
\[\lim_{\alpha}\liminf_{i\to \infty}\PP_{i,\alpha}(\{\xi:\phi_{\xi}\in \Map(\rho,F,L,\delta,\sigma_{i})\})=1.\]
\end{lemma}

\begin{proof}

(a): This will be left as an exercise.

(b) Let $C=\Prob(X)\subseteq C(X)^{*},$ and $K\subseteq C$ be the weak$^{*}$ compact, convex set consisting of $\Gamma$-invariant measures. The ergodicity of $\Gamma\actson (X,\mu)$ is equivalent to saying that $\mu$ is an extreme point of $K.$ Given  finite $F\subseteq\Gamma,L\subseteq C(X),$ and $\delta,\kappa>0,$ set
\[U=\bigcap_{f\in L}\left\{\eta\in \Prob(X):\left|\int_{X}f\,d\eta-\int_{X}f\,d\mu\right|<\delta\right\}.\]
 Given a finite $L'\subseteq C(X),\delta'>0,$ let $V_{L,\delta'}\subseteq C(X)^{*}$ consist of all (complex) measures $\eta$ so that
 \[\max_{f\in L'}\left|\int_{X}f\,d\eta\right|<\delta'.\]
 The preceding Lemma allows us to find a finite $L'\subseteq C(X)$ and $\kappa'>0$ so that if $\nu\in \Prob(C),$ and
\[\mu-\int_{C}\eta\,d\nu(\eta)\in V_{L',\delta'},\]
\[\nu((K+V_{L',\delta'})\cap C)=1,\]
then
\[\nu(\{\eta\in \Prob(X):\eta \in U\})\geq 1-\kappa.\]
By a compactness argument, we may find a finite $F'\subseteq G$ with $F\subseteq F'$ and a $\delta\in (0,\delta')$  so that for all $i\in \NN$ and all $\phi\in \Map(\rho,F,\delta,\sigma_{i})$ we have $\phi_{*}(u_{d_{i}})\in K+V_{L',\delta'}.$

	Our assumptions allow us to find an $\alpha_{0}\in I$ so that if $\alpha\geq\alpha_{0},$ then
\[\max_{g\in F'}\limsup_{i\to \infty}\|\rho_{2}(\phi_{\xi}\circ \sigma_{i}(g),g\phi_{\xi})\|_{L^{\infty}(\xi)}<\delta',\]
\[\max_{f\in L'}\limsup_{i\to\infty}\left|\int f\,d(\Phi_{i,\alpha})_{*}(u_{d_{i}}\otimes \PP_{i,\alpha})-\int f\,d\mu\right|<\kappa'.\]
Fix $\alpha\geq \alpha_{0}.$ As $F$ and $L'$ are finite, the above limiting statements allow us to find a $i_{0}\in \NN$ so that if $i\geq i_{0},$ then
\begin{equation}\label{E:happensallthetime}
\max_{g\in F'}\|\rho_{2}(\phi_{\xi}\circ \sigma_{i}(g),g\phi_{\xi})\|_{L^{\infty}(\xi)}<\delta',
\end{equation}
and
\[\max_{f\in L'}\left|\int f\,d(\Phi_{i,\alpha})_{*}(u_{d_{i}}\otimes \PP_{i,\alpha})-\int f\,d\mu\right|<\kappa'.\]
Define
\[\Psi_{i,\alpha}\colon \Omega_{\alpha,i}\to C\]
by
\[\Psi_{i,\alpha}(\xi)=(\phi_{\xi})_{*}(u_{d_{i}}),\]
and let
\[\eta_{i,\alpha}=(\Psi_{i,\alpha})_{*}(\PP_{i,\alpha}).\]
Then for all $f\in C(X),$
\[\int_{C}\int_{X}f\,d\nu\,d\eta_{i,\alpha}(\nu)=\int_{\Omega_{i,\alpha}}\int_{X}f\,d(\phi_{\xi})_{*}(u_{d_{i}})\,d\PP_{i,\alpha}(\xi)=\int f\,d(\Phi_{i,\alpha})_{*}(u_{d_{i}}\otimes \PP_{i,\alpha}).\]
Thus for all $i\geq i_{0},$
\[\max_{f\in L'}\left|\int_{C}\int_{X}f\,d\nu\,d\eta_{i,\alpha}-\int_{X}f\,d\mu\right|<\kappa'.\]
By $(\ref{E:happensallthetime}),$ and our choice of $F',\delta',$ we have for all $i\geq i_{0}$
\[\PP_{i,\alpha}(\{\xi:\phi_{\xi}\in \Map(\rho,F,\delta,\sigma_{i})\})=1,\]
\[\eta_{i,\alpha}((K+V_{L',\delta'})\cap C)=1.\]
So for all $i\geq i_{0}$
\[\PP_{i,\alpha}(\{\xi:\phi_{\xi}\in \Map(\rho,F,L,\delta,\sigma_{i})\})=\PP_{i,\alpha}(\{\xi:(\phi_{\xi})_{*}(u_{d_{i}})\in U\})=\eta_{i,\alpha}(\{\nu:\nu \in U\})\geq 1-\kappa,\]
by choice of $\kappa',L'.$ Thus
\[\liminf_{\alpha}\liminf_{i\to\infty}\PP_{i,\alpha}(\{\xi:\phi_{\xi}\in \Map(\rho,F,L,\delta,\sigma_{i})\})\geq 1-\kappa.\]
The lemma is completed by letting $\kappa\to 0.$

\end{proof}

	Let us explain why we call this Lemma the ``Automatic Concentration Lemma." The assumption
\[(\Phi_{i,\alpha})_{*}(u_{d_{i}}\times \PP_{i,\alpha})\to \mu\]
can be thought of as saying that
\[(\phi_{\xi})_{*}(u_{d_{i}})\to \mu\]
``in expectation." When one is dealing with limits of large probability spaces (for example, uniform probability measures on large finite sets, or normalized Lebesgue measure on the ball of a large finite-dimensional Banach space), there is a principle called the ``concentration of measure phenomenon." This roughly says that (nice enough) functions have a very small deviation from their expectation. This principle is used to great effect, e.g. in the theory of  random graphs, geometric functional analysis, and random matrix theory. The conclusion of the Lemma, i.e. that
\[\PP_{i,\alpha}\left(\left\{\xi:\left|\int f\,d(\phi_{\xi})_{*}u_{d_{i}}-\int f\,d\mu\right|<\delta\right\}\right)\to 1,\]
can be interpreted as the statement that $(\phi_{\xi})_{*}u_{d_{i}}$ ``concentrates" near its expectation. So this Lemma can be thought of as an automatic concentration result, provided that $\Gamma\actson (X,\mu)$ is ergodic.

	Note that this differs in spirit from most concentration results. Typically one needs some concrete estimates (e.g. a log Sobolev inequality) and gets explicit results on the deviation from the mean (typically exponential decay). For this lemma, the techniques, results, and assumptions are much softer. No concrete estimates is need for the proof of this concentration result, however the result also gives no concrete estimates on deviations from the mean. To apply this result, we need ergodicity as well as weak$^{*}$ convergence. For ergodicity, we state a Theorem due to Hanfeng Li, Klaus Schmidt, and Jesse Peterson. They only state this result for $f\in \ZZ(\Gamma),$ however  the proof works for  $f\in M_{m,n}(\ZZ(\Gamma)).$

\begin{theorem}[Theorem 1.3 in \cite{LiSchmidtPet}]\label{T:automaticergodicity} Let $\Gamma$ be a non-amenable group, and let $f\in M_{m,n}(\ZZ(\Gamma))$ be such that $\lambda(f)$ has dense image. Then $\Gamma\actson (X_{f},m_{X_{f}})$ is ergodic.
\end{theorem}

	We thus focus on proving the weak$^{*}$ convergence necessary to apply the Automatic Concentration Lemma. In many respects, this is the most difficult and technical part of the paper. To help illuminate the ideas, let us roughly outline the proof. Fix $f\in M_{n}(\ZZ(\Gamma))$ such that $\lambda(f)$ is injective. First, let us analyze the proof of  Theorem \ref{T:squarematrix}. Let $x_{k}\in GL_{n}(M_{d_{k}}(\RR))$ be the rank perturbation considered in the proof of Theorem \ref{T:squarematrix}. We see that the $\xi\in \Xi_{\delta}(x_{k})$ we used to compute the topological entropy come in two types. Setting
	\[G_{k}=x_{k}^{-1}((\ZZ^{d_{k}})^{\oplus n})/((\ZZ^{d_{k}})^{\oplus n}),\]
	\[\Omega_{\delta,\varepsilon}=x_{k}^{-1}(\delta\Ball(\ell^{2}_{\RR}(d_{k}n,u_{d_{k}n})))\cap \chi_{[0,\delta/\varepsilon]}(|x_{k}|)((\RR^{d_{k}})^{\oplus n}),\]
every $v\in\Xi_{\delta}(x_{k})$ is of the form
\[v+(\ZZ^{d_{k}})^{\oplus n}=\xi+\zeta+(\ZZ^{d_{k}})^{\oplus n},\]
for some $\xi\in G_{k},\zeta\in\Omega_{\delta}.$ For $\xi\in(\TT^{d_{i}})^{\oplus n}$ define
\[\phi_{\xi}\colon\{1,\dots,d_{i}\}\to (\TT^{\Gamma})^{\oplus n},\]
by
\[\phi_{\xi}(j)(l)(g)=\xi(l)(\sigma_{i}(g)^{-1}(j)).\]
Let $m_{G_{k}}$ be the Haar measure on $G_{k}$ and $\mu_{\delta,\varepsilon}$ the normalized Lebesgue measure on $\chi_{[0,\delta/\varepsilon]}(|x_{k}|)((\RR^{d_{k}})^{\oplus n})$ chosen so that $\mu_{\delta}(\Omega_{\delta,\varepsilon})=1.$ In order to apply the Automatic Concentration Lemma, we are left trying to argue that
\[\int_{\Omega_{\delta,\varepsilon}}\int_{G_{k}}(\phi_{\xi+\zeta+(\ZZ^{d_{k}})^{\oplus n}})_{*}(u_{d_{k}})\,dm_{G_{k}}(\xi)\,d\mu_{\delta,\varepsilon}(\zeta)\approx m_{X_{f}}\]
in the weak$^{*}$ topology.  By abstract Fourier analysis, it is enough to verify this condition by integrating both sides of this approximate equality against the continuous function
\[\ev_{\alpha}\colon(\TT^{\Gamma})^{\oplus n}\to \CC\]
for $\alpha\in(\ZZ^{\Gamma})^{\oplus n}$ given by
\[\ev_{\alpha}(\Theta)=\exp(2\pi i\ip{\Theta,\alpha}).\]
For $A\in M_{s,t}(L(\Gamma)),$ let $\widetidle{A}\in M_{t,s}(L(\Gamma))$ be given by $(\widetilde{A})_{ij}=A_{ji}.$ We are then naturally led to show that
\[\frac{1}{d_{k}}\sum_{j=1}^{d_{k}}\int_{\Omega_{\delta,\varepsilon}}\int_{G_{k}}\exp(2\pi i(\sigma_{k}(\widetilde{\alpha})\xi)(j))\exp(2\pi i \sigma_{k}(\widetilde{\alpha})\zeta)(j))\,dm_{G_{k}}(\xi)\,d\mu_{\delta,\varepsilon}(\zeta)\approx 0\]
for $\alpha\in\ZZ(\Gamma)^{\oplus n}\setminus r(f)(\ZZ(\Gamma)^{\oplus m}).$
The  integral
\[\int_{G_{k}}\exp(2\pi i (\sigma_{k}(\widetilde{\alpha})\xi)(j))\,dm_{G_{k}}(\xi)\]
 will be zero if and only if $\xi\mapsto (\sigma_{k}(\widetilde{\alpha})\xi)(j)+\ZZ$ is a nontrivial homomorphism on $G_{k}$, and will be one otherwise. It is simple to describe when the homomorphism $\xi\mapsto (\sigma_{k}(\widetilde{\alpha})\xi)(j)+\ZZ$ is trivial. By Lemma \ref{L:BlerghHadag} below this will occur only when $\sigma_{k}(\widetilde{\alpha})^{*}e_{j}=x_{k}^{*}r_{j,k}$ for some $r_{j,k}\in (\ZZ^{d_{k}})^{\oplus n}.$ We would like to argue that this forces $\alpha\in r(f)(\ZZ(\Gamma)^{\oplus m}),$ but can only do this when we have a uniform bound on $\|r_{j,k}\|_{\ell^{2}(d_{k}n)}$ (see Lemma \ref{L:idealtest}).

We will then see that in order to prove the statement
\[\int_{\Omega_{\delta,\varepsilon}}\int_{G_{k}}(\phi_{\xi+\zeta+(\ZZ^{d_{k}})^{\oplus n}})_{*}(u_{d_{i}})\,dm_{G_{k}}(\xi)\,d\mu_{\delta,\varpesilon}(\zeta)\approx m_{X_{f}}\]
we are left with analyzing
\[\int_{\Omega_{\delta,\varepsilon}}\exp(2\pi i (\sigma_{k}(\widetilde{\alpha})\zeta)(j))\,d\mu_{\delta,\varepsilon}(\zeta)=\int_{\Omega_{\delta,\varepsilon}}\exp(2\pi i\ip{x_{k}\zeta,r_{j,k}}_{\ell^{2}(d_{k}n)})\,d\mu_{\delta,\varepsilon}(\zeta)\]
for $\alpha\in\ZZ(\Gamma)^{\oplus n}\setminus r(f)(\ZZ(\Gamma)^{\oplus m}),$ and where $\|r_{j,k}\|_{\ell^{2}(d_{k}n)}$ is large. The fact that  $\|r_{j,k}\|_{\ell^{2}(d_{k}n)}$  is large makes the integrand $\exp(2\pi i \ip{x_{k}\xi,r_{j,k}}_{\ell^{2}(d_{k}n)}))$ highly oscillatory and we exploit this to argue that
\[\int_{\Omega_{\delta,\varepsilon}}\exp(2\pi i (\sigma_{k}(\widetilde{\alpha})\zeta)(j))\,d\mu_{\delta,\varepsilon}(\zeta)\approx 0.\]
In order to give certain examples of algebraic actions with infinite measure-theoretic entropy (see Theorem \ref{T:infinitemeasureentropyexample}) we will  state our methods so that they work in a slightly more general setting than having $\lambda(f)$ be injective on $\ell^{2}(\Gamma)^{\oplus n}.$
	
\begin{lemma}\label{L:BlerghHadag} Let $T\in M_{m,n}(\ZZ)$ and $v\in \ZZ^{n}.$ The homomorphism
\[T^{-1}(\ZZ^{m})/\ZZ^{n}\to\RR/\ZZ\]
given by
\[\xi\mapsto \ip{\xi,v}_{\ell^{2}(n)}+\ZZ\]
is identically zero if and only if
\[v\in T^{*}(\ZZ^{m}).\]

\end{lemma}

\begin{proof} One direction is obvious. We thus focus on showing that if $\xi\mapsto \ip{\xi,v}_{\ell^{2}(n)}+\ZZ$
is identically zero, then $v\in T^{*}(\ZZ^{m}).$ So suppose that $v\in \ZZ^{n}$ and the homomorphisms $\xi\mapsto \ip{\xi,v}_{\ell^{2}(n)}+\ZZ$ is identically zero for $\xi\in T^{-1}(\ZZ^{m}).$  We clearly have that $\ker(T)\subseteq T^{-1}(\ZZ^{m})$ and thus for every $\xi\in\ker(T)\cap \RR^{n}$ we have
\[\ip{\xi,v}_{\ell^{2}(n)}\in\ZZ.\]
Thus the linear functional on $\ker(T)\cap \RR^{n}$ defined by $\xi\mapsto\ip{\xi,v}_{\ell^{2}(n)}$ is $\ZZ$-valued. A linear functional on a real vector space can only be $\ZZ$-valued if it is zero and this proves that $v\in\ker(T)^{\perp}\cap \RR^{n}=T^{*}(\RR^{m}).$ So we can let $r\in \RR^{m}$ be such that $v=T^{*}r.$ Thus
\[\ip{T\xi,r}_{\ell^{2}(m)}\in\ZZ\]
for all $\xi\in T^{-1}(\ZZ^{m}).$ Equivalently,
\[\ip{\xi,r}_{\ell^{2}(m)}\in\ZZ\]
for all $\xi\in \ZZ^{m}\cap T(\RR^{n}).$

  By Smith normal formal, we may find matrices $A\in GL_{m}(\ZZ),B\in GL_{n}(\ZZ),D\in M_{m,n}(\ZZ)$ so that
\[ATB=D,\]
\[D=\begin{bmatrix}
\alpha_{1}&0&0&0&\cdots& 0 &\cdots &0\\
0& \alpha_{2}&0&\cdots&0 &0& \cdots &0\\
0&0 &\alpha_{3}&\cdots&0& 0&\cdots &0 \\
\vdots&\vdots&\ddots&\cdots&\vdots& \vdots &\cdots  &\vdots\\
0&0& \cdots&\cdots &\alpha_{d}& 0& \cdots&0 \\
0&0&0&0&\cdots&0 & \cdots &0 \\
\vdots&\vdots&\vdots&\ddots&\cdots&\vdots &\ddots &\vdots \\
0&0&0&0&\cdots&0 &\cdots &0
\end{bmatrix},\]
where $\alpha_{1},\dots,\alpha_{d}$ are not zero. It is easy to see that $\ZZ^{m}\cap T(\RR^{n})=A^{-1}(\ZZ^{d}\oplus\{0\}).$ Thus for all $l\in\ZZ^{d}$ we have
\[\ip{A^{-1}(l\oplus 0),r}_{\ell^{2}(n)}\in \ZZ\]
equivalently
\[\ip{l\oplus 0,(A^{-1})^{*}r}_{\ell^{2}(n)}\in\ZZ\]
for all $l\in\ZZ^{d}.$ Thus $(A^{-1})^{*}r=s+t$ where $s\in \ZZ^{d}\oplus \{0\},$ and $t\in \{0\}\oplus \RR^{n-d}.$ So
\[r=A^{*}s+A^{*}t.\]
Thus
\[v=T^{*}r=T^{*}A^{*}s+T^{*}A^{*}t.\]
Since
\[B^{*}T^{*}A^{*}=D\]
it is not hard to show that $T^{*}A^{*}t=0.$ Thus
\[v=T^{*}A^{*}s\in T^{*}(\ZZ^{m}).\]

\end{proof}

The above Lemma will be useful when combined with the following Lemma.  The following Lemma can be thought of as giving a way to ``test'', in terms of a given sofic approximation, if a given $\alpha\in\ZZ(\Gamma)^{\oplus n}$ is an element of the submodule $r(f)(\ZZ(\Gamma)^{\oplus m}).$ Recall that if $A\in M_{s,t}(L(\Gamma))$ we let $\widetidle{A}\in M_{t,s}(L(\Gamma))$ be defined by
\[(\widetilde{A})_{ij}=A_{ji}.\]

\begin{lemma}[The Submodule Test]\label{L:idealtest} Let $\Gamma$ be a countable discrete sofic group with sofic approximation $\Sigma=(\sigma_{i}\colon \Gamma\to S_{d_{i}}).$ Extend $\Sigma$ to a map $\sigma_{i}\colon M_{s,t}(\ZZ(\Gamma))\to M_{s,t}(M_{d_{i}}(\ZZ))$ in the usual way. Let $f\in M_{m,n}(\ZZ(\Gamma))$ and suppose $x_{i}$ is a rank perturbation of $\sigma_{i}(f).$
Let $\alpha\in \ZZ(\Gamma)^{\oplus n}$ and regard $\ZZ(\Gamma)^{\oplus n}=M_{n,1}(\ZZ(\Gamma)).$ Suppose that there is some $C>0$ so that
\[\limsup_{i\to \infty}\frac{|\{1\leq j\leq d_{i}:\sigma_{i}(\widetilde{\alpha})^{*}e_{j}\in x_{i}^{*}((\ZZ^{d_{i}})^{\oplus m}\cap C\Ball(\ell^{2}(d_{i}m)))\}|}{d_{i}}>0.\]
Then $\alpha\in r(f)(\ZZ(\Gamma)^{\oplus m}).$
\end{lemma}

\begin{proof} Passing to a subsequence, we may assume that there is some $\beta>0$ so that
\[\beta\leq \frac{|\{j:\sigma_{i}(\widetilde{\alpha})^{*}e_{j}\in x_{i}^{*}((\ZZ^{d_{i}})^{\oplus m}\cap C\Ball(\ell^{2}(d_{i}m)))\}|}{d_{i}}\]
for all $i.$ For $x\in L(\Gamma)^{\oplus p},$ set
\[\|x\|_{2}^{2}=\Tr\otimes \tau(x^{*}x)\]
 for $\theta\in \ell^{2}(\Gamma)^{\oplus p}$ we use the norm
\[\|\theta\|_{2}^{2}=\sum_{j,g}|\theta(j)(g)|^{2}.\]
We claim the following is true.

\emph{Claim: For every finite $E\subseteq \Gamma,$ there is an $R\in\ZZ(\Gamma)^{\oplus m}$ so that}
\[((f^{*}R)(l))^{\widehat{}}(g)=\widehat{\alpha(l)^{*}}(g)\mbox{ for all $g\in E$, $1\leq l\leq n$,}\]
\[\|R\|_{2}\leq C.\]

Suppose we grant the claim for now. Let $E_{k}$ be an increasing sequence of non-empty finite subsets of $\Gamma$ whose union is $\Gamma.$ Choose an $R_{k}\in \ZZ(\Gamma)^{\oplus m}$ so that
\[((f^{*}R_{k})(l))^{\widehat{}}(g)=\widehat{\alpha(l)^{*}}(g)\mbox{ for all $g\in E_{k},1\leq l\leq n$},\]
\[\|R_{k}\|_{2}\leq C.\]
Passing to a subsequence, we may assume that $\widehat{R_{k}(l)}(g)$ converges pointwise to some  $r(l,g).$ By Fatou's Lemma, we see that
\[\|r\|_{\ell^{2}(\Gamma)^{\oplus m}}\leq C.\]
Thus $r\in \ell^{2}(\Gamma,\ZZ)^{\oplus m}=c_{c}(\Gamma,\ZZ)^{\oplus m}.$ So we may define $R\in \ZZ(\Gamma)^{\oplus m}$ by
\[\widehat{R(l)}(g)=r(l,g).\]
It is easy to see that
\[\widehat{(f^{*}R)(l)}^{}(g)=\widehat{\alpha(l)^{*}}(g)\]
and hence if we let $Q\in \ZZ(\Gamma)^{\oplus m}$ be given as $Q(l)=R(l)^{*},$ we have that
\[r(f)Q=\alpha.\]
Thus $\alpha\in r(f)(\ZZ(\Gamma)^{\oplus m}),$ which proves the Lemma.

We now turn to the proof of the claim. So fix a finite subset $E$ of $\Gamma.$ Let
\[S=\bigcup_{1\leq r\leq m,1\leq s\leq n}\supp(\widehat{f_{rs}})\cup\{e\}\cup\supp(\widehat{f_{rs}})^{-1}\]
\[K=(E\cup E^{-1}\cup\{e\})^{(2015)!}S(E\cup E^{-1}\cup \{e\})^{(2015)!}.\]
We may find $C_{i}\subseteq\{1,\dots,d_{i}\}$ so that
\[u_{d_{i}}(C_{i})\to 1,\]
\[\sigma_{i}(g_{1})^{\varepsilon_{1}}\cdots \sigma_{i}(g_{l})^{\varepsilon_{l}}(j)=\sigma_{i}(g_{1}^{\varepsilon_{1}}\cdots g_{l}^{\varepsilon_{l}})(j)\mbox{ for all $j\in C_{i},$ $1\leq l\leq (2015)!,\varepsilon_{1},\cdots,\varepsilon_{l}\in \{-1,1\},g_{1},\dots,g_{l}\in K$},\]
\[x_{i}(e_{r}\otimes e_{\sigma_{i}(g)(j)})=\sigma_{i}(f)(e_{r}\otimes e_{\sigma_{i}(g)(j)})\mbox{for all $j\in C_{i},$ $1\leq r\leq n,g\in E$},\]
\[\sigma_{i}(g)(j)\ne \sigma_{i}(h)(j)\mbox{ if $g,h\in (K\cup\{e\}\cup K^{-1})^{(2015)!}$, and $g\ne h$.}\]
Let
\[B_{i}=\{j:\sigma_{i}(\widetilde{\alpha})^{*}e_{j}\in x_{i}^{*}((\ZZ^{d_{i}})^{\oplus m}\cap C\Ball(\ell^{2}(d_{i}m)))\},\]
so that
\[\liminf_{i\to\infty}u_{d_{i}}(B_{i}\cap C_{i})\geq \beta>0,\]
and hence for all large $i,$ we have $B_{i}\cap C_{i}$ is not empty. For $j\in B_{i}\cap C_{i},$ choose $r_{j,i}\in (\ZZ^{d_{i}})^{\oplus m}$ so that
\[\sigma_{i}(\widetilde{\alpha})^{*}e_{j}=x_{i}^{*}r_{j,i}\]
\[\|r_{j,i}\|_{\ell^{2}(d_{i}m)}\leq C.\]

Fix a $j\in B_{i}\cap C_{i},$ by our previous comments such a $j$ exists if $i$ is sufficiently large. Define $R\in \ZZ(\Gamma)^{\oplus m}$ by
\[\widehat{R(l)}(g)=\begin{cases}
0, & \textnormal{ if $g\in\Gamma\setminus K$}\\
\ip{r_{j,i},e_{l}\otimes e_{\sigma_{i}(g)(j)}}_{\ell^{2}(d_{i}m)}, & \textnormal{ if $g\in K$}
\end{cases}.\]
Since $\sigma_{i}(g)(j)\ne \sigma_{i}(h)(j)$ for $g\ne h$ in $K$ we have
\[\|R\|_{2}\leq C.\]
For $g\in E,j\in C_{i}$ and $1\leq l\leq n$ we have
\[\ip{\sigma_{i}(\widetilde{\alpha})^{*}e_{j},e_{l}\otimes e_{\sigma_{i}(g)(j)}}_{\ell^{2}(d_{i}n)}=\widehat{\alpha(l)^{*}}(g).\]
As $K\supseteq SE,$ we have for $j\in C_{i}\cap B_{i},g\in E$ that
\begin{align*}
\widehat{\alpha(l)^{*}}(g)&=\ip{\sigma_{i}(\widetilde{\alpha})^{*}e_{j},e_{l}\otimes e_{\sigma_{i}(g)(j)}}_{\ell^{2}(d_{i}n)}\\
&=\ip{x_{i}^{*}r_{j,i},e_{l}\otimes e_{\sigma_{i}(g)(j)}}_{\ell^{2}(d_{i}n)}\\
&=\ip{r_{j,i},x_{i}(e_{l}\otimes e_{\sigma_{i}(g)(j)})}_{\ell^{2}(d_{i}m)}\\
&=\ip{r_{j,i},\sigma_{i}(f)(e_{l}\otimes e_{\sigma_{i}(g)(j)})}_{\ell^{2}(d_{i}m)}\\
&=\sum_{1\leq s\leq m}\sum_{h\in\Gamma}\widehat{(f^{*})_{ls}}(h^{-1})\ip{r_{j,i},e_{s}\otimes e_{\sigma_{i}(hg)(j)}}_{\ell^{2}(d_{i}m)}\\
&=\sum_{1\leq s\leq m}\sum_{h\in\Gamma}\widehat{(f^{*})_{ls}}(h^{-1})\widehat{R(s)}(hg)\\
&=((f^{*}R)(l))^{\widehat{}}(g).
\end{align*}

\end{proof}

As stated before, the last step in our weak$^{*}$ convergence argument will be to argue that a certain oscillatory integral is close to zero. The following lemma will be key in the proof of this fact.

\begin{lemma}\label{L:equatorconcentration} Let $n_{k}$ be a sequence of positive integers, and let $c_{k},R_{k}$ be a sequence of positive real numbers with $c_{k}\to 0.$ Then
\[\lim_{k\to \infty}\sup_{\substack{\xi \in \RR^{n_{k}},\\ \|\xi\|_{\ell^{2}(n_{k})}\leq \frac{c_{k}R_{k}}{\sqrt{n_{k}}}}}\frac{\vol(R_{k}\Ball(\ell^{2}_{\RR}(n_{k}))\setminus (R_{k}\Ball(\ell^{2}_{\RR}(n_{k}))+\xi))}{\vol(R_{k}\Ball(\ell^{2}_{\RR}(n_{k}))}=0.\]
\end{lemma}
\begin{proof} The claim is easy if $n_{k}$ is bounded, so passing to a subsequence we may assume that $n_{k}\to \infty.$ Without loss of generality $R_{k}=1$ for all $k.$ Let $m_{k}$ be the normalized Lebesgue measure on $\Ball(\ell^{2}_{\RR}(n_{k}))$ chosen so that $m_{k}(\Ball(\ell^{2}_{\RR}(n_{k})))=1.$ Let $\nu_{k}$ be the unique probability measure on $S^{n_{k}-1}$ invariant under the action of $O(n_{k},\RR).$ Then for all measurable $f\colon \Ball(\ell^{2}_{\RR}(n))\to [0,\infty)$ we have
\[\int_{\Ball(\ell^{2}_{\RR}(n_{k}))}f(x)\,dm_{k}(x)=n_{k}\int_{0}^{1}\int_{S^{n_{k}-1}}r^{n_{k}-1}f(rx)\,d\nu_{k}(x)\,dr.\]

	We have for any $T\in M_{n_{k}}(\CC),$
\[\int_{S^{n-1}}\ip{T\xi,\xi}_{\ell^{2}(n_{k})}\,d\nu_{k}(x)=\tr_{n}(T).\]
Indeed, if we consider the left-hand side of the above equality to be a linear functional on $M_{n_{k}}(\CC)$ then it is  invariant under conjugation by elements of $O(n_{k},\RR)$ and takes the value $1$ on $\id,$ and this uniquely defines $\tr_{n}(T).$ Thus if $v\in \RR^{n_{k}},$ $\|v\|_{\ell^{2}(n_{k})}=1,$ we may apply the above discussion  to the operator $T\xi=\ip{\xi,\zeta}\zeta,$ and see that
\begin{equation}\label{E:equatorconcentration}
\int_{S^{n_{k}-1}}|\ip{\xi,v}_{\ell^{2}(n_{k})}|^{2}\,d\nu_{k}(\xi)=\frac{1}{n_{k}}.
\end{equation}
Now let $\xi\in \RR^{n_{k}},\|\xi\|_{\ell^{2}_{\RR}(n_{k})}\leq \frac{c_{k}}{\sqrt{n_{k}}},$ set
\[v=\frac{\xi}{\|\xi\|_{\ell^{2}(n_{k})}}.\]
Let $C>1$ and let
\[B=\left\{\zeta \in \Ball(\ell^{2}_{\RR}(n_{k})):\|\zeta-\xi\|_{\ell^{2}(n_{k})}\geq 1,|\ip{\zeta,v}_{\ell^{2}(n_{k})}|\leq \frac{C}{\sqrt{n_{k}}}\|\zeta\|_{\ell^{2}_{\RR}(n_{k})}\right\}.\]
Then by $(\ref{E:equatorconcentration}),$
\begin{equation}\label{E:Festimate}
m_{k}(\Ball(\ell^{2}_{\RR}(n_{k}))\setminus (\Ball(\ell^{2}_{\RR}(n_{k}))+\xi))\leq \frac{1}{C^{2}}+m_{k}(B).
\end{equation}

	Now if $\zeta\in B,$ then
\begin{align*}
1&\leq \|\zeta-\xi\|_{\ell^{2}(n_{k})}^{2}\\
&=\|\zeta\|_{\ell^{2}(n_{k})}^{2}-2\|\xi\|_{2}\Rea(\ip{\zeta,v}_{\ell^{2}(n_{k})})+\|\xi\|_{\ell^{2}(n_{k})}^{2}\\
&\leq \|\zeta\|_{\ell^{2}(n_{k})}^{2}+2\frac{c_{k}C}{n_{k}}\|\zeta\|_{\ell^{2}(n_{k})}+\frac{c_{k}^{2}}{n_{k}}\\
&\leq \left(\|\zeta\|_{\ell^{2}(n_{k})}+\frac{c_{k}C}{n_{k}}\right)^{2}+\frac{c_{k}^{2}}{n_{k}}.
\end{align*}
Thus
\[\|\zeta\|_{\ell^{2}(n_{k})}\geq \omega(k)\]
where
\[\omega(k)=\left(1-\frac{c_{k}^{2}}{n_{k}}\right)^{1/2}-\frac{c_{k}C}{n_{k}}.\]
So
\begin{equation}\label{E:FAman}
m_{k}(B)\leq n_{k} \int_{\omega(k)}^{1}r^{n_{k}-1}\,dr=1-\omega(k)^{n_{k}}.
\end{equation}
We have
\[n_{k}\omega(k)-n_{k}=(n_{k}^{2}-c_{k}^{2}n_{k})^{1/2}-n_{k}-c_{k}C=-c_{k}C-\frac{c_{k}^{2}n_{k}}{(n_{k}^{2}-c_{k}^{2}n_{k})^{1/2}+n_{k}}\to 0\]
since $c_{k}\to 0,$ and $n_{k}\to \infty.$

	Thus given $\varespilon>0,$ we have
\[\omega(k)\geq 1-\frac{\varepsilon}{n_{k}}\]
for all large $k.$ Applying the above to $(\ref{E:Festimate}),$ and $(\ref{E:FAman}),$ and using that $n_{k}\to \infty,$ we find that
\[\limsup_{k\to \infty}\sup_{\substack{\xi \in \RR^{n_{k}},\\ \|\xi\|_{\ell^{2}(n_{k})}\leq \frac{c_{k}}{\sqrt{n_{k}}}}}\frac{\vol(\Ball(\ell^{2}_{\RR}(n_{k}))\setminus (\Ball(\ell^{2}_{\RR}(n_{k}))+\xi))}{\vol(\Ball(\ell^{2}_{\RR}(n_{k}))}\leq \frac{1}{C^{2}}+1-e^{-\varespilon}.\]
As $C>1,\varespilon>0$ are arbitrary the Lemma is proved.

\end{proof}

	We now prove the necessary weak$^{*}$ convergence result needed to apply the Automatic Concentration Lemma.

\begin{lemma}\label{L:proofofweak^{*}convergence} Let $\Gamma$ be a countable discrete sofic group with sofic approximation $\Sigma=(\sigma_{k}\colon \Gamma\to S_{d_{k}}).$ Let $f\in M_{m,n}(\ZZ(\Gamma))$ be such that the image of $\lambda(f)$ is dense. Extend $\sigma_{k}$ to a map $\sigma_{k}\colon M_{s,t}(\ZZ(\Gamma))\to M_{s,t}(M_{d_{k}}(\ZZ))$ in the usual way. Let $x_{k}\in M_{m,n}(M_{d_{k}}(\ZZ))$ be a rank perturbation of $\sigma_{k}(f)$ such that there exists $A_{k}\subseteq\{1,\dots,d_{k}\}^{m}$ with $x_{k}((\RR^{d_{k}})^{\oplus n})=\RR^{A_{k}}$ and
\[\lim_{k\to\infty}\frac{|A_{k}|}{d_{k}m}=1\]
 (this is possible by Proposition \ref{P:rankperturabtion} (i)).

	For $\delta,\varespilon>0,$ let
\[G_{k}=\frac{x_{k}^{-1}((\ZZ^{d_{k}})^{\oplus m})}{(\ZZ^{d_{k}})^{\oplus n}},\]
\[p_{\delta/\varepsilon}=\chi_{(0,\delta/\varepsilon]}(|x_{k}|),\]
\[W_{k,\delta,\varepsilon}=p_{\delta/\varepsilon}(\ell^{2}_{\RR}(d_{k}n,u_{d_{k}n}))\cap x_{k}^{-1}(\delta \Ball(\ell^{2}_{\RR}(d_{k}m,u_{d_{k}m}))).\]
Let $\mu_{\delta,\varepsilon}$ be the normalized Lebesgue measure on $p_{\delta/\varepsilon}(\ell^{2}_{\RR}(d_{k}n,u_{d_{k}n})),$ chosen so that $\mu_{\delta,\varespilon}(W_{k,\delta,\varepsilon})=1.$ Let $m_{G_{k}}$ be the Haar measure on $G_{k}.$ Define
\[\Phi_{k}\colon \{1,\cdots,d_{k}\}\times G_{k}\times W_{k,\delta,\varepsilon}\to (\TT^{\Gamma})^{\oplus n}\]
by
\[\Phi_{k}(j,\xi,\zeta)(l)(x)=\xi(l)(\sigma_{k}^{-1}(x)(j))+\zeta(l)(\sigma_{k}^{-1}(x)(j))+\ZZ,\mbox{ for $1\leq l\leq n,1\leq j\leq d_{k},x\in\Gamma.$}\]
Then, for all $g\in C((\TT^{\Gamma})^{n})$ and for all $\varepsilon>0,$
\[\limsup_{\delta\to 0}\limsup_{k\to \infty}\left|\int g\,d(\Phi_{k})_{*}(u_{d_{k}}\otimes u_{G_{k}}\otimes \mu_{\delta,\varespilon})-\int g\,dm_{X_{f}}\right|=0.\]

\end{lemma}

\begin{proof}
We use $\eta=\frac{\delta}{\varespilon}.$ For $\alpha\in \ZZ(\Gamma)^{\oplus n},$ let $\ev_{\alpha}\in C((\TT^{\Gamma})^{n})$ be defined by
\[\ev_{\alpha}(\zeta)=e^{2\pi i \ip{\zeta,\alpha}}\]
where
\[\ip{\zeta,\alpha}=\sum_{1\leq l\leq n}\sum_{g\in \Gamma}\zeta(l)(g)\widehat{\alpha}(l)(g).\]
Identify $\ZZ(\Gamma)^{\oplus p}$ with $M_{p,1}(\ZZ(\Gamma)).$ We let $M_{s,t}(M_{d_{i}}(\ZZ))$ act as operators $(\TT^{d_{i}})^{t}\to (\TT^{d_{i}})^{s}$ in the usual way.  It is straightforward to see that
\[\ev_{\alpha}(\Phi_{k}(j,\xi,\zeta))=\exp(2\pi i[\sigma_{k}(\widetilde{\alpha})(\xi+\zeta)](j)).\]

	By abstract Fourier analysis, it is enough to assume that $g=\ev_{\alpha}$ for some $\alpha\in \ZZ(\Gamma)^{\oplus n}.$ We are thus required to show that
\[\limsup_{\delta\to 0}\limsup_{k\to \infty}\left|1-\int \ev_{\alpha}\,d(\Phi_{k})_{*}(u_{d_{k}}\otimes u_{G_{k}}\otimes \mu_{\delta,\varespilon})\right|=0,\]
if $\alpha\in r(f)(\ZZ(\Gamma)^{\oplus m}),$ and
\[\limsup_{\delta\to 0}\limsup_{k\to \infty}\left|\int \ev_{\alpha}\,d(\Phi_{k})_{*}(u_{d_{k}}\otimes u_{G_{k}}\otimes \mu_{\delta,\varespilon}	)\right|=0,\]
when $\alpha\in \ZZ(\Gamma)^{\oplus n}\setminus r(f)(\ZZ(\Gamma)^{\oplus m}).$
The desired claim is easy when $\alpha\in r(f)(\ZZ(\Gamma)^{\oplus m}),$ so we will focus on the case when $\alpha\in \ZZ(\Gamma)^{\oplus n}\setminus r(f)(\ZZ(\Gamma)^{\oplus m}).$

	If $A$ is a compact abelian group and $\theta\in \Hom(A,\RR/\ZZ)$ is not identically zero, then
\begin{equation}\label{E:thinkaboutit}
\int _{A}\exp(2\pi i \theta(a))\,dm_{A}(a)=0.
\end{equation}
Let
\[\mathcal{B}_{k}=\{j:\xi\mapsto[\sigma_{k}(\widetilde{\alpha})\xi](j)\mbox{is trivial on $G_{k}$}\}.\]
Then by $(\ref{E:thinkaboutit})$
\begin{equation}\label{E:crucialcharactercomputation}
\int f\,d(\Phi_{k})_{*}(u_{d_{k}}\otimes u_{G_{k}}\otimes \mu_{\delta,\varepsilon})=\frac{1}{d_{k}}\sum_{j\in \mathcal{B}_{k}}\int_{W_{k,\delta,\varepsilon}}e^{2\pi i [\sigma_{k}(\widetilde{\alpha})\zeta](j)}\,d\mu_{\delta,\varepsilon}(\zeta).
\end{equation}
If
\[\lim_{k\to \infty}\frac{|\mathcal{B}_{k}|}{d_{k}}=0,\]
then our clam is easy. Passing to a subsequence we assume that
\begin{equation}\label{E:asymptoticdistribution}
\frac{|\mathcal{B}_{k}|}{d_{k}}\to \beta>0.
\end{equation}
By Lemma \ref{L:BlerghHadag} we have
\[\mathcal{B}_{k}=\{j:\sigma_{k}(\widetilde{\alpha})^{*}e_{j}\in x_{k}^{*}((\ZZ^{d_{k}})^{\oplus m})\}.\]
For $j\in \mathcal{B}_{k},$ let $r_{j,k}\in (\ZZ^{d_{k}})^{\oplus m}$ be such that
\begin{equation}\label{E:puttingourselvesinagoodposition}
\sigma_{k}(\widetilde{\alpha})^{*}e_{j}=x_{k}^{*}r_{j,k}.
\end{equation}
Since $\ker(x_{k}^{*})^{\perp}=\im(x_{k})=\RR^{A_{k}},$ replacing $r_{j,k}$ with $\chi_{A_{k}}r_{j,k}$ we may assume that $r_{j,k}\in\ZZ^{A_{k}}.$ Since $\alpha\notin r(f)(\ZZ(\Gamma)^{\oplus m}),$ Lemma \ref{L:idealtest} and (\ref{E:puttingourselvesinagoodposition}) implies that we can find a sequence of positive real numbers $C_{k}$ tending to $\infty,$ so that
\[\frac{|\{j\in \mathcal{B}_{k}:\|r_{j,k}\|_{\ell^{2}(d_{k}n)}\leq C_{k}\}|}{d_{k}}\to 0.\]
We claim that for all large $k,$ $p_{\eta}\ne 0.$ To prove this, set
\[q_{\eta}=\chi_{(0,\eta]}(|x_{k}^{*}|).\]
By functional calculus and the fact that $r_{j,k}\perp \ker(x_{k}^{*}),$
\begin{align}\label{E:alidhgalkdhglkad}
\|r_{j,k}-q_{\eta}r_{j,k}\|_{\ell^{2}(d_{k}m)}^{2}&=\|\chi_{(\eta,\infty)}(|x_{k}^{*}|)r_{j,k}\|_{\ell^{2}(d_{k}m)}^{2}\\ \nonumber
&=\ip{\chi_{(\eta,\infty)}(|x_{k}^{*}|)r_{j,k},r_{j,k}}_{\ell^{2}(d_{k}m)}\\ \nonumber
&\leq \frac{1}{\eta^{2}}\ip{|x_{k}^{*}|^{2}r_{j,k},r_{j,k}}_{\ell^{2}(d_{k}m)}\\ \nonumber
&=\frac{1}{\eta^{2}}\|x_{k}^{*}r_{j,k}\|_{\ell^{2}(d_{k}n)}^{2}\\ \nonumber
&=\frac{1}{\eta^{2}}\|\sigma_{k}(\widetidle{\alpha})^{*}e_{j}\|_{\ell^{2}(d_{k}n)}^{2}\\ \nonumber
&\leq \frac{\|\widehat{\alpha}\|_{1}^{2}}{\eta^{2}}. \nonumber
\end{align}
Since $C_{k}\to\infty$ and $\mathcal{B}_{k}\ne \varnothing$ for all large $k,$ we see that $q_{\eta}r_{j,k}\ne 0$ for all large $k$ then  and thus that $q_{\eta}\ne 0.$ Let $x_{k}=u_{k}|x_{k}|$ be the polar decomposition. We leave it as an exercise to show that for all Borel $f\colon [0,\infty)\to \RR$ with
\[f(0)=0\]
we have
\[f(|x_{k}^{*}|)=u_{k}f(|x_{k}|)u_{k}^{*}.\]
Indeed, this is easy when $f(t)=t^{2n}$ for some $n\in\NN,$  and the general case follows by approximation. Thus
\[q_{\eta}=\chi_{(0,\eta]}(|x_{k}^{*}|)=u_{k}p_{\eta}u_{k}^{*}.\]
So
\[\tr(q_{\eta})=\tr(u_{k}p_{\eta}u_{k}^{*})=\tr(p_{\eta}),\]
as
\[u_{k}^{*}u_{k}=\Proj_{\ker(x_{k})^{\perp}}\geq p_{\eta}.\]
Thus the fact that $q_{\eta}\ne 0$ for all large $k$ implies that $p_{\eta}\ne 0$ for all large $k.$
Let
\[V_{k,\delta,\varepsilon}=x_{k}p_{\eta}(\ell^{2}_{\RR}(d_{k}n,u_{d_{k}}n))\cap \delta\Ball(\ell^{2}_{\RR}(d_{k}n,u_{d_{k}m})),\]
and let $\nu_{\delta,\varepsilon}$ be the normalized Lebesgue measure on $V_{k,\delta,\varepsilon}$ chosen so that $\nu_{\delta,\varepsilon}(V_{k,\delta,\varepsilon})=1.$
By the change of variables $\zeta\mapsto x_{k}\zeta$, it is not hard to see that
\begin{equation}\label{E:changeofvariables}
\int_{W_{k,\delta,\varepsilon}}e^{2\pi i [\sigma_{k}(\widetilde{\alpha})\zeta](j)}\,d\mu_{\delta,\varepsilon}(\zeta)=\int_{V_{k,\delta,\varepsilon}}e^{2\pi i \ip{\zeta,r_{j,k}}_{\ell^{2}(d_{k}n)}}\,d\nu_{\delta,\varepsilon}(\zeta).
\end{equation}
Since
\[x_{k}p_{\eta}(\ell^{2}_{\RR}(d_{k}n,u_{d_{k}n}))=u_{k}|x_{k}|\chi_{(0,\eta]}(|x_{k}|)(\ell^{2}_{\RR}(d_{k}n,u_{d_{k}n}))=u_{k}\chi_{(0,\eta]}(|x_{k}|)(\ell^{2}_{\RR}(d_{k}n,u_{d_{k}n})),\]
it is not hard to see that
\begin{equation}\label{E:itsthesubspaceyoualdghal}
\Proj_{x_{k}p_{\eta}(\ell^{2}_{\RR}(d_{k}n,u_{d_{k}n}))}=u_{k}\chi_{(0,\eta]}(|x_{k}|)u_{k}^{*}=q_{\eta}.
\end{equation}

	By (\ref{E:alidhgalkdhglkad}),
\[\|q_{\eta}r_{j,k}-r_{j,k}\|_{\ell^{2}(d_{k}m)}\leq \frac{\|\widehat{\alpha}\|_{1}}{\eta}.\]
So if we set $M_{k}=C_{k}-\frac{\|\widehat{\alpha}\|_{1}}{\eta},$ it follows that
\begin{equation}\label{E:projectedoccilation}
\frac{|\{j\in \mathcal{B}_{k}:\|q_{\eta}r_{j,k}\|_{\ell^{2}(d_{k}m)}\leq M_{k}\}|}{d_{k}}\to 0.
\end{equation}
Set
\[\mathcal{M}_{k}=\{j\in \mathcal{B}_{k}:\|q_{\eta}r_{j,k}\|_{\ell^{2}(d_{k}m)}\geq M_{k}\},\]
then
\[\label{E:bigthings}
\lim_{k\to \infty}\frac{1}{d_{k}}\sum_{j\in \mathcal{B}_{k}}\int_{V_{k,\delta,\varepsilon}}e^{2\pi i \ip{\zeta,r_{j,k}}_{\ell^{2}(d_{k}m)}}\,d\nu_{\delta,\varepsilon}(\zeta)=\lim_{k\to \infty}\frac{1}{d_{k}}\sum_{j\in \mathcal{M}_{k}}\int_{V_{k,\delta,\varepsilon}}e^{2\pi i \ip{\zeta,r_{j,k}}_{\ell^{2}(d_{k}m)}}\,d\nu_{\delta,\varepsilon}(\zeta).\]
Applying (\ref{E:itsthesubspaceyoualdghal}) and the fact that $V_{k,\delta,\varepsilon}\subseteq x_{k}p_{\eta}(\ell^{2}_{\RR}(d_{k}n,u_{d_{k}n}))$ we have
\begin{align}\label{E:bigthings}\lim_{k\to \infty}\frac{1}{d_{k}}\sum_{j\in \mathcal{B}_{k}}\int_{V_{k,\delta,\varepsilon}}e^{2\pi i \ip{\zeta,r_{j,k}}_{\ell^{2}(d_{k}m)}}\,d\nu_{\delta,\varepsilon}(\zeta)&=\lim_{k\to \infty}\frac{1}{d_{k}}\sum_{j\in \mathcal{M}_{k}}\int_{V_{k,\delta,\varepsilon}}e^{2\pi i \ip{q_{\eta}\zeta,r_{j,k}}_{\ell^{2}(d_{k}m)}}\,d\nu_{\delta,\varepsilon}(\zeta)\\
&=\lim_{k\to \infty}\frac{1}{d_{k}}\sum_{j\in \mathcal{M}_{k}}\int_{V_{k,\delta,\varepsilon}}e^{2\pi i \ip{\zeta,q_{\eta}r_{j,k}}_{\ell^{2}(d_{k}m)}}\,d\nu_{\delta,\varepsilon}(\zeta) \nonumber
\end{align}
For $j\in \mathcal{M}_{k},$ let
\[v_{j,k}=\frac{q_{\eta}r_{j,k}}{\|q_{\eta}r_{j,k}\|_{\ell^{2}_{\RR}(d_{k}m)}},\]
\[\zeta_{j,k}=\frac{1}{2\|q_{\eta}r_{j,k}\|_{\ell^{2}_{\RR}(d_{k}m)}}v_{j,k}.\]
Since $\ip{\zeta_{j,k},q_{\eta}r_{j,k}}=1/2,$
\[\int_{V_{k,\eta}}e^{2\pi i \ip{\zeta,q_{\eta}r_{j,k}}_{\ell^{2}(d_{k}m)}}\,d\nu_{\eta}(\zeta)=\frac{1}{2}\left(\int_{V_{k,\eta}}e^{2\pi i \ip{\zeta,q_{\eta}r_{j,k}}_{\ell^{2}(d_{k}m)}}\,d\nu_{\eta}(\zeta)-\int_{V_{k,\eta}+\zeta_{j,k}}e^{2\pi i \ip{\zeta,q_{\eta}r_{j,k}}_{\ell^{2}(d_{k}m)}}\,d\nu_{\eta}(\zeta)\right).\]
Hence,
\begin{equation}\label{E:osciallatoryintegral}
\left|\int_{V_{k,\delta,\varepsilon}}e^{2\pi i \ip{\zeta,q_{\eta}r_{j,k}}_{\ell^{2}(d_{k}m)}}\,d\nu_{\delta,\varepsilon}(\zeta)\right|=
\end{equation}
\[\frac{1}{2}\left|\int_{V_{k,\delta,\varepsilon}}e^{2\pi i \ip{\zeta,q_{\eta}r_{j,k}}_{\ell^{2}(d_{k}m)}}\,d\nu_{\delta,\varepsilon}(\zeta)-\int_{V_{k,\delta,\varepsilon}+\zeta_{j,k}}e^{2\pi i \ip{\zeta,q_{\eta}r_{j,k}}_{\ell^{2}(d_{k}m)}}\,d\nu_{\delta,\varepsilon}(\zeta)\right|\leq\nu_{\delta,\varepsilon}(V_{k,\delta,\varepsilon}\setminus (V_{k,\delta,\varepsilon}+\zeta_{j,k})).\]

	Let $w_{k}=\Tr(p_{\eta}).$ Then
\[\|\zeta_{j,k}\|_{\ell^{2}(d_{k}m)}\leq \frac{1}{2M_{k}}=\frac{\delta\sqrt{d_{k}m}}{\sqrt{w_{k}}}\cdot\left(\frac{1}{2\delta M_{k}}\sqrt{\frac{w_{k}}{d_{k}m}}\right).\]
Since
\[w_{k}>0, \mbox{ (as we have already shown that $p_{\eta}\ne 0$)},\]
\[\lim_{k\to \infty}\frac{w_{k}}{d_{k}}\leq \mu_{|f|}((0,\eta]),\]
and
\[M_{k}\to \infty,\]
we may apply Lemma \ref{L:equatorconcentration}  with
\[n_{k}=w_{k},R_{k}=\delta\sqrt{d_{k}m},c_{k}=\left(\frac{1}{2\delta M_{k}}\sqrt{\frac{w_{k}}{d_{k}m}}\right)\]
and find that
\[\sup_{j\in \mathcal{M}_{k}}\nu_{\delta,\varepsilon}(V_{k,\delta,\varepsilon}\setminus (V_{k,\delta,\varepsilon}+\zeta_{j,k}))\to 0.\]
Applying $(\ref{E:bigthings}),(\ref{E:osciallatoryintegral})$ completes the proof.

\end{proof}

\subsection{The Main Result}

	We now prove that the measure theoretic entropy of $\Gamma\actson (X_{f},m_{X_{f}})$ is $\log \Det_{L(\Gamma)}(f)$ if $f\in M_{n}(\ZZ(\Gamma))$ is injective as a left multiplication operator on $\ell^{2}(\Gamma)^{\oplus n}.$

\begin{theorem}\label{T:measureentropy} Let $\Gamma$ be a countable discrete non-amenable sofic group.  Let $f\in M_{n}(\ZZ(\Gamma))$ be injective as a left-multiplication operator on $\ell^{2}(\Gamma)^{\oplus n}.$ Then for any sofic approximation $\Sigma$ of $\Gamma$ we have
\[h_{\Sigma,m_{X_{f}}}(X_{f},\Gamma)=\log \Det_{L(\Gamma)}(f).\]
\end{theorem}

\begin{proof}

	From Theorem \ref{T:squarematrix} and the variational principle, it is enough to show that
\[h_{\Sigma,m_{X_{f}}}(X_{f},\Gamma)\geq \log \Det_{L(\Gamma)}(f).\]
Identify $X_{f}\subseteq (\TT^{\Gamma})^{n}$ as usual. We may regard $m_{X_{f}}$ as a measure on $(\TT^{\Gamma})^{n}.$ Trivially we have
\[\Gamma\actson (X_{f},m_{X_{f}})\cong \Gamma \actson ((\TT^{\Gamma})^{n},m_{X_{f}})\]
as measure-preserving actions of $\Gamma.$ We use the pseudometric $\rho$ on $X_{f}$ given by
\[\rho(\chi_{1},\chi_{2})^{2}=\frac{1}{n}\sum_{1\leq l\leq n}|\chi_{1}(l)(e)-\chi_{2}(l)(e)|^{2},\]
let $\theta_{2,(\ZZ^{m})^{\oplus n}}$ be the pseudometric defined on $(\TT^{m})^{n}$ as before. Let
\[\Sigma=(\sigma_{k}\colon \Gamma\to S_{d_{k}}),\]
be a sofic approximation. By Proposition \ref{P:rankperturabtion}, we may find a rank  perturbation $x_{k}\in GL_{n}(M_{d_{k}}(\RR)).$ Let
\[M=\sup_{k}\|x_{k}\|_{\infty}.\]
 Set
\[G_{k}=x_{k}^{-1}((\ZZ^{d_{k}})^{\oplus n})/((\ZZ^{d_{k}})^{\oplus n}),\]
by Lemma \ref{L:smith} we know that $G_{k}$ is a finite abelian group and that
\[|G_{k}|=|\det(x_{k})|.\]

Let $\varepsilon>0,$ for $\delta>0,k\in \NN,$ let $G_{k},W_{k,\delta,\varespilon},\mu_{\delta,\varespilon},p_{\delta/\varespilon},\Phi_{k}$ be as in Lemma \ref{L:proofofweak^{*}convergence}. For $(\xi,\zeta)\in G_{k}\times W_{k,\delta,\varespilon}$ set
\[\phi_{\xi+\zeta}(j)=\Phi_{k}(j,\xi,\zeta).\]
Let $L\subseteq C((\TT^{\Gamma})^{n}),F\subseteq \Gamma$ be finite, and $\eta>0.$ By Proposition \ref{P:Left/RightIssues}, we know that $\lambda(f)$ has dense image. By Lemma \ref{L:proofofweak^{*}convergence}, Theorem \ref{T:automaticergodicity} and the Automatic Concentration Lemma (Lemma \ref{L:AutomaticConcentration})
\[\liminf_{\delta\to 0}\liminf_{k\to \infty}(u_{G_{k}}\otimes \mu_{\delta,\varespilon})(\{(\xi,\zeta):\phi_{\xi+\zeta}\in \Map(\rho,F,L,\eta,\sigma_{i})\})=1.\]

	Fix $\kappa>0,$ and choose $\delta_{0}>0$ sufficiently small so that
\[\liminf_{k\to \infty}(u_{G_{k}}\otimes \mu_{\delta,\varespilon})(\{(\xi,\zeta):\phi_{\xi+\zeta}\in \Map(\rho,F,L,\eta,\sigma_{i})\})>1-\kappa\]
if $\delta<\delta_{0}.$ Let
\[E_{k}=\{\xi\in G_{k}:\mu_{\delta,\varespilon}(\{\zeta:\phi_{\xi+\zeta}\in \Map(\rho,F,L,\eta,\sigma_{i})\})\geq 1-\sqrt{\kappa}\}.\]
Then for all $\delta<\delta_{0},$ and all large $k$ we have
\[\frac{|E_{k}|}{|G_{k}|}\geq 1-\sqrt{\kappa},\]
so as in Theorem \ref{T:squarematrix}
\[|E_{k}|\geq (1-\sqrt{\kappa})|\det(x_{k})|.\]
For $\xi\in E_{k},$ let
\[\Omega_{\xi}=\{\zeta\in W_{k,\delta,\varespilon}:\phi_{\xi+\zeta}\in \Map(\rho,F,L,\eta,\sigma_{i})\},\]
so that
\[\mu_{\delta,\varepsilon}(\Omega_{\xi})\geq 1-\sqrt{\kappa}.\]

	Let $\mathcal{N}$  be such that $\{x_{k}\xi\}_{\xi\in \mathcal{N}}$ is a maximal $2\varespilon M$ separated subset of $E_{k}$ with respect to $\theta_{2,x_{k}((\ZZ^{d_{i}})^{\oplus n}}.$ And for $\xi\in \mathcal{N},$ let $\mathcal{M}_{\xi}$ be a maximal $\varepsilon$-separated subset of $\Omega_{\xi}$ with respect to $\theta_{2,(\ZZ^{d_{k}})^{\oplus n}}.$ Let $\omega_{n}(\cdot)$ be defined as in the proof of Theorem \ref{T:squarematrix}. It follows as in the proof of Theorem \ref{T:squarematrix}, that for $\delta<\delta_{0},$
\[N_{\varepsilon}(\Map(\rho,F,L,\eta,\sigma_{i}),\rho_{2})\geq \sum_{\xi \in\mathcal{N}}|\mathcal{M}_{\xi}|,\]
\[|\mathcal{M}_{\xi}|\geq (1-\sqrt{\kappa})\frac{\Det_{\delta/\varespilon}(x_{k})^{-1}\delta^{\Tr(p)}}{\varespilon^{\Tr(p)}\omega_{n}(\varespilon M+\delta)},\]
\[|\mathcal{N}|\geq (1-\sqrt{\kappa})\frac{|\det(x_{k})|}{\omega_{n}\left(2\varepsilon M\right)}.\]
The rest of the proof follows as in Theorem \ref{T:squarematrix}.

\end{proof}

We can also use our techniques to give examples of algebraic actions with infinite measure-theoretic entropy.

\begin{theorem}\label{T:infinitemeasureentropyexample} Let $\Gamma$ be a countable discrete non-amenable sofic group, and $\Sigma=(\sigma_{k}\colon\Gamma\to S_{d_{k}})$ a sofic approximation. Let $f\in M_{m,n}(\ZZ(\Gamma))$ and suppose that $\lambda(f)$ is not injective, but has dense image  (so necessarily $m<n$ by Proposition  \ref{P:Left/RightIssues}). Then
\[h_{\Sigma,m_{X_{f}}}(X_{f},\Gamma)=\infty.\]

\end{theorem}

\begin{proof} Let $\rho$ be the dynamical generating pseudometric on $(\TT^{\Gamma})^{\oplus n}$ given by
\[\rho(\theta_{1},\theta_{2})^{2}=\frac{1}{n}\sum_{l=1}^{n}|\theta_{1}(l)(e)-\theta_{2}(l)(e)|^{2}.\]
For $E\subseteq \Gamma$ finite, and $\theta_{1},\theta_{2}\in (\TT^{\Gamma})^{\oplus n}$ set
\[\rho_{E}(\theta_{1},\theta_{2})=\max_{g\in E}\rho_{2}(g\theta_{1},g\theta_{2}).\]
 Fix $0<\varepsilon<1.$ By Theorem \ref{T:automaticergodicity} the action $\Gamma\actson (X_{f},m_{X_{f}})$ is ergodic, and so we may apply the Automatic Concentration Lemma. By Proposition \ref{P:rankperturabtion}, we may find a rank  perturbation $x_{k}\in M_{m,n}(M_{d_{k}}(\ZZ))$ so that $x_{k}((\RR^{d_{k}})^{\oplus n})=\RR^{A_{k}}$ for some
\[A_{k}\subseteq\{1,\dots,d_{k}\}^{m}\]
with
\[\lim_{k\to\infty}\frac{|A_{k}|}{d_{k}}=m.\]
Let $G_{k},\mu_{\delta,\varepsilon},W_{k,\delta,\varepsilon},\Phi$ be as in Lemma \ref{L:proofofweak^{*}convergence}. For $\xi\in(\TT^{d_{i}})^{\oplus n}$ define
\[\phi_{\xi}\colon \{1,\dots,d_{i}\}\to (\TT^{\Gamma})^{\oplus n}\]
by
\[\phi_{\xi}(j)(l)(g)=\xi(l)(\sigma_{k}(g)^{-1}(j)).\]
For $1\leq j\leq d_{k},$ let $\Phi_{j}\colon G_{k}\times W_{k,\delta,\varepsilon}\to (\TT^{\Gamma})^{\oplus n}$ be given by
\[\Phi_{j}(\xi,\zeta)=\Phi(j,\xi,\zeta).\]
We remark that the proof of Lemma \ref{L:proofofweak^{*}convergence}  shows that for $\alpha\in\ZZ(\Gamma)^{\oplus n}$ and $\eta>0$ we have
\begin{equation}\label{E:strongconcentration}\lim_{\delta\to 0}\lim_{k\to\infty}u_{d_{k}}\left(\left\{j:\left|\int \ev_{\alpha}d(\Phi_{j})_{*}(m_{G_{k}}\otimes \mu_{\delta,\varepsilon})-\int_{X_{f}}\ev_{\alpha}\,dm_{X_{f}}\right|>\eta\right\}\right)=0.
\end{equation}

	Let $\Omega_{k}=\ker(x_{k})\cap \Ball(\ell^{2}_{\RR}(d_{k}n,u_{d_{k}n})),$ and let $\nu$ be the normalized Lebesgue measure on $\Omega_{k}$ so that $\nu(\Omega_{k})=1.$ Define
\[\Psi\colon\{1,\dots,d_{k}\}\times \Omega_{k}\times G_{k}\times W_{k,\delta,\varepsilon}\to (\TT^{\Gamma})^{\oplus n}\]
by
\[\Psi(j,v,\xi,\zeta)(l)(g)=v(l)(\sigma_{k}(g)^{-1}(j))+\xi(l)(\sigma_{k}(g)^{-1}(j))+\zeta(l)(\sigma_{k}(g)^{-1}(j))+\ZZ.\]
Note that for all $g\in \Gamma$ and finite $E\subseteq\Gamma$
\begin{equation}\label{E:itsalwayscloseetc}
\sup_{v\in \Omega_{k}}\rho_{E,2}(\phi_{v},X_{f}^{\oplus d_{k}})\to_{k\to\infty}0.
\end{equation}
This follows as any $v\in\Omega_{k}$  is in $\Xi_{\delta}(x_{k})$ for any $\delta>0,$ and $x_{k}$ is rank perturbation of $\sigma_{k}(f).$  From (\ref{E:strongconcentration}) and $(\ref{E:itsalwayscloseetc})$ and the fact that
\[\int_{X_{f}}\ev_{\alpha}\,dm_{X_{f}}=\begin{cases}
0,& \mbox{ if $\alpha\notin r(f)(\ZZ(\Gamma)^{\oplus m})$,}\\
1,& \mbox{ if $\alpha\in r(f)(\ZZ(\Gamma)^{\oplus m})$,}
\end{cases}\]
it is not hard to see that
\[\limsup_{\delta\to 0}\limsup_{k\to\infty}\left|\int h\,d\Psi_{*}(u_{d_{k}}\otimes \nu\otimes m_{G_{k}}\otimes \mu_{\delta,\varepsilon})-\int h\,dm_{X_{f}}\right|=0,\]
for all $h\in C(X_{f}).$ Thus by the Automatic Concentration Lemma,
\[\limsup_{\delta\to 0}\limsup_{k\to\infty}(\nu\otimes m_{G_{k}}\otimes \mu_{\delta,\varepsilon})\left(\{(v,\xi,\zeta):\phi_{v+\xi+\zeta}\in\Map(\rho,F,\delta,L,\sigma_{k})\}\right)=1.\]
So we may fix a $\delta_{0}>0$ so that if $0<\delta<\delta_{0},$ then
\[\limsup_{k\to\infty}(\nu\otimes m_{G_{k}}\otimes \mu_{\delta,\varepsilon})\left(\{(v,\xi,\zeta):\phi_{v+\xi+\zeta}\in\Map(\rho,F,\delta,L,\sigma_{k})\}\right)>1/2.\]
Thus for all sufficiently large $k,$ we can find $ (\xi,\zeta)\in G_{k}\times W_{k,\delta,\varepsilon}$ so that
\[\nu(\{v:\phi_{v+\xi+\zeta}\in \Map(\rho,F,\delta,L,\sigma_{k})\})>1/2.\]
Set
\[B=\{v:\phi_{v+\xi+\zeta}\in \Map(\rho,F,\delta,L,\sigma_{k})\}.\]
Let $S\subseteq B$ be $\varepsilon$-dense with respect to $\Theta_{2,(\ZZ^{d_{k}})^{\oplus n}}$ of minimal cardinality. Since
\[\rho_{2}(\phi_{\xi},\phi_{\zeta})=\theta_{2,(\ZZ^{d_{k}})^{\oplus n}}(\xi,\zeta)\]
for all $\xi,\zeta\in(\TT^{d_{i}})^{\oplus n}$ we have
\[N_{\varepsilon/4}(\Map(\rho,F,\delta,L,\sigma_{i}),\rho_{2})\geq |S|.\]
 For all $v\in B,$ we can find a $w\in S$ and an $l\in (\ZZ^{d_{k}})^{\oplus n}$ so that
\[\|v-w-l\|_{2,(\ZZ^{d_{k}})^{\oplus n}}<\varepsilon.\]
Thus
\[B+p\xi+p\zeta\subseteq \bigcup_{\substack{w\in S,\\ l\in (\ZZ^{d_{k}})^{\oplus n}\cap 3\Ball(\ell^{2}_{\RR}(d_{k}n,u_{d_{k}n}))}}w+\xi+\zeta+l+\varepsilon\Ball(\ell^{2}_{\RR}(d_{k},u_{d_{k}n}).\]
Let $p$ be the projection onto the kernel of $x_{k}.$ Thus
\[B+\xi+\zeta\subseteq \bigcup_{\substack{w\in S,\\ l\in (\ZZ^{d_{k}})^{\oplus n}\cap 3\Ball(\ell^{2}_{\RR}(d_{k}n,u_{d_{k}n}))}}w+p\xi+p\zeta+pl+\varepsilon\Ball(\ker(x_{k})\cap (\RR^{d_{k}})^{\oplus n},\|\cdot\|_{\ell^{2}(d_{k}n,u_{d_{k}}n)}).\]
Computing volumes, we have
\[\frac{1}{2}\leq \nu(B)\leq \varepsilon^{\dim_{\RR}(\ker(x_{k})\cap (\RR^{d_{k}})^{\oplus n})}|S||(\ZZ^{d_{k}})^{\oplus n}\cap 3\Ball(\ell^{2}_{\RR}(d_{k}n,u_{d_{k}n}))|.\]
Note that
\[4\Ball(\ell^{2}_{\RR}(d_{k}n,u_{d_{k}n}))\supseteq\bigcup_{l\in (\ZZ^{d_{k}})^{\oplus n}\cap 3\Ball(\ell^{2}_{\RR}(d_{k}n,u_{d_{k}n}))}l+[-1/2,1/2)^{d_{k}n}\]
and this union is a disjoint union. From this it follows that
\[4^{d_{k}n}\vol(\Ball(\ell^{2}_{\RR}(d_{k}n,u_{d_{k}n})))\geq | (\ZZ^{d_{k}})^{\oplus n}\cap 3\Ball(\ell^{2}_{\RR}(d_{k}n,u_{d_{k}n}))|.\]
and so (see page 11 of \cite{Pis}) we may find a $C>0$ with
\[| (\ZZ^{d_{k}})^{\oplus n}\cap 3\Ball(\ell^{2}_{\RR}(d_{k}n,u_{d_{k}n}))|\leq C^{d_{k}n}.\]
Hence, we see that
\[|S|\geq \frac{1}{2}\varepsilon^{-\dim_{\RR}(\ker(x_{k})\cap (\RR^{d_{k}})^{\oplus n})}C^{-d_{k}n}.\]
Therefore, by Lemma \ref{L:injective}
\[h_{\Sigma,m_{X_{f}}}(\rho,F,L,\delta,\varepsilon/4,\sigma_{k})\geq \dim_{L(\Gamma)}(\ker\lambda(f))\log(1/\varepsilon)-n\log C.\]
Since $\lambda(f)$ is not injective,
\[\dim_{L(\Gamma)}(\ker\lambda(f))>0.\]
Thus we can take the infimum over  $F,L,\delta$ and let $\varepsilon\to 0$ to complete the proof.

\end{proof}

\section{Applications of the Results}\label{S:apply}

\subsection{Random Sofic Entropy and $f$-Invariant Entropy of Algebraic Actions}

	In this section, we extend our results on sofic entropy of algebraic actions to Lewis Bowen's $f$-invariant entropy. We use  $\FF_{r}$ for the free group on $r$ generators. Typically the following notions are defined in terms of countable partitions, we will prefer to use measurable maps $\alpha\colon X\to A,$ where $A$ is a countable set. Such a map induces a partition with elements $\alpha^{-1}(\{a\})$
for $a\in A.$ We leave it as an exercise to the informed reader to check that all our definitions agree with the usual definitions for partitions.

\begin{definition} \emph{Let $(X,\mu)$ be a standard probability space, and $\FF_{r}\actson (X,\mu)$ a probability measure-preserving action. Let $K$ be a standard Borel space, a measurable map $\alpha\colon X\to K$ is said to be a} generator \emph{if for every measurable subset $E\subseteq X,$ $\varepsilon>0,$ there exists a finite subset $F\subseteq \Gamma,$ a natural number $n$ and $(A_{g,})_{g\in F,1\leq j\leq n}$ measurable subsets of $K$ so that}
\[\mu\left(E\Delta\bigcup_{j=1}^{n}\bigcap_{g\in F}g\alpha^{-1}(A_{g,j})\right)<\varepsilon.\]
\end{definition}

Let $(X,\mu)$ be a standard probability space. An \emph{observable} is defined to be a  measurable map $\alpha\colon X\to A,$ with $A$ a countable set. If $A$ is finite, we say that $\alpha$ is a \emph{finite observable}.

If $\alpha\colon X\to A,\beta\colon X\to B$ are two observables, we say that $\alpha\leq \beta,$ if there is a $\pi\colon B\to A$ so that $\alpha(x)=\pi(\beta(x))$
for almost every $x\in X.$

 Recall that if $A$ is a countable set, $(X,\mu)$ is a standard probability space, and $\alpha\colon X\to A$ is a measurable observable then the entropy of $\alpha$ is defined by
\[H(\alpha)=-\sum_{a\in A}\mu(\alpha^{-1}(\{a\}))\log \mu(\alpha^{-1}(\{a\})).\]
If $\alpha\colon X\to A,$  $\beta\colon X\to B$ are observables we define
\[(\alpha\vee \beta)\colon X\to A\times B\]
by
\[(\alpha\vee \beta)(x)=(\alpha(x),\beta(x)),\]
and we define $H(\alpha|\beta)$ by
\[H(\alpha|\beta)=H(\alpha\vee \beta)-H(\beta).\]
If $\Gamma\actson (X,\mu)$ and $g\in \Gamma$ define $g\alpha \colon X\to A$ by
$(g\alpha)(x)=\alpha(g^{-1}x).$
Let us recall the definition of Lewis Bowen's $f$-invariant entropy.  For $n\in \NN,$ we use $B(e,n)$ for the ball of radius $n$ in $\FF_{r}$ with respect to the standard word metric.

\begin{definition}\emph{ Let $(X,\mu)$ be a standard probability space and $\FF_{r}\actson (X,\mu)$ a measure preserving action, let $a_{1},\cdots,a_{r}$ be free generators for $\FF_{r}.$  Let $\alpha\colon X\to A$ be an observable with finite Shannon entropy. Define}
\[F(\alpha)=(1-2r)H(\alpha)+\sum_{i=1}^{r}H(\alpha\vee a_{i}\alpha),\]
\[f(\alpha)=\inf_{n>0}F\left(\bigvee_{g\in B(e,n)}g\alpha\right).\]
\end{definition}

Lewis Bowen proved (see \cite{Bowenfinvariant} Theorem 1.3) that if $\alpha,\beta$ are two finite generators, then $f(\alpha)=f(\beta).$ So we can define
\[f_{\mu}(X,\FF_{r})=f(\alpha)\]
if there is a finite generator  $\alpha$ for the action. We  prove in this section that if $h\in M_{m,n}(\ZZ(\FF_{r}))$ is injective as a left multiplication operator on $\ell^{2}(\FF_{r})^{\oplus n},$ then assuming the action $\FF_{r}\actson (X_{h},m_{X_{h}})$ has a finite generator,
\[f_{m_{X_{h}}}(X_{h},\FF_{r})\leq \log \Det^{+}_{L(\FF_{r})}(h)\]
with equality if $m=n.$ To do this, we will use a relation due to Lewis Bowen between $f$-invariant entropy and sofic entropy with respect to a random sofic approximation. Let us recall the definition of a random sofic approximation. If $X$ is a standard Borel space, we use $\Prob(X)$ for the space of Borel probability measures on $X.$

\begin{definition}\emph{Let $\Gamma$ be a countable discrete group, a sequence $\kappa_{i}\in \Prob(S_{d_{i}}^{\Gamma})$ is said to be a} random sofic approximation  of $\Gamma$ \emph{if}
\begin{list}{ \arabic{pcounter}:~}{\usecounter{pcounter}}
\item $d_{i}\to \infty$
\item  \emph{for all $g,h\in \Gamma,\delta>0,$ there is an $i_{0}$ so that $i\geq i_{0}$ implies}  \[\kappa_{i}(\{\sigma:u_{d_{i}}(\{j:(\sigma(g)\sigma(h))(j)=\sigma(gh)(j)\})\geq 1-\delta\})=1,\]
\item $\kappa_{i}\otimes u_{d_{i}}(\{(\sigma,j):\sigma(g)(j)\ne \sigma(h)(j)\})\to 1,$ \emph{for all $g,h\in \Gamma$ with $g\ne h$}
\end{list}
\end{definition}

Property 2 may seems like an unnaturally strong assumption, but it turns out to be necessary for the definition of random sofic entropy to be an invariant. We now state the definition of entropy of an action with respect to a random sofic approximation.

\begin{definition}\emph{Let $\Gamma$ be a countable discrete group, and $\kappa=(\kappa_{i})$ a random sofic approximation of $\Gamma$ with $\kappa_{i}\in \Prob(S_{d_{i}}^{\Gamma})$. Let $X$ be a compact metrizable space and $\Gamma\actson X$ by homeomorphisms. Let $\rho$ be a dynamically generating pseudometric on $X.$  Define the topological entropy with respect to $\kappa$ by}
\[h_{\kappa}(\rho,F,\delta,\varepsilon)=\limsup_{i\to \infty}\frac{1}{d_{i}}\log \int_{S_{d_{i}}^{\Gamma}}S_{\varepsilon}(\Map(\rho,F,\delta,\sigma),\rho_{2})\,d\kappa_{i}(\sigma),\]
\[h_{\kappa}(\rho,\varepsilon)=\inf_{\substack{F\subseteq \Gamma\mbox{\emph{ finite,}}\\ \delta>0}}h_{\kappa}(\rho,F,\delta,\varespilon),\]
\[h_{\kappa}(X,\Gamma)=\sup_{\varepsilon>0}h_{\kappa}(\rho,\varepsilon).\]
\emph{If $\mu$ is a $\Gamma$-invariant Borel probability measure on $X,$ define the measure theoretic entropy of $\Gamma\actson (X,\mu)$ by}
\[h_{\kappa,\mu}(\rho,F,L,\delta,\varepsilon)=\limsup_{i\to \infty}\frac{1}{d_{i}}\log \int S_{\varespilon}(\Map(\rho,F,L,\delta,\sigma),\rho_{2})\,d\kappa_{i}(\sigma),\]
\[h_{\kappa,\mu}(\rho,\varepsilon)=\inf_{\substack{F\subseteq \Gamma\mbox{\emph{ finite,}}\\L\subseteq C(X)\mbox{\emph{ finite,}}\\ \delta>0}}h_{\kappa}(\rho,F,L,\delta,\varepsilon)\]
\[h_{\kappa,\mu}(X,\Gamma)=\sup_{\varepsilon>0}h_{\kappa,\mu}(\rho,\varepsilon).\]
\end{definition}
It was shown by Lewis Bowen in \cite{BowenGroupoid} Theorem 6.7 and Theorem 9.4 that $h_{\kappa},h_{\kappa,\mu}$ are independent of the choice of dynamically generating pseudometric. We will also need the definition of measure-theoretic entropy without using a topological model (essentially due to David Kerr in \cite{KerrPartition} as a generalization of Lewis Bowen's definition in \cite{Bow}).

\begin{definition} \emph{Let $\Gamma$ be a countable discrete group and $\sigma\colon \Gamma\to S_{d}$ a function. Let $(X,\mathcal{B},\mu)$ be a standard probability space, and $\mathcal{S}\subseteq \mathcal{B}$ a subalgebra of $\mathcal{B}.$ Let $\alpha\colon X\to A$ be a finite $\mathcal{S}$-measurable observable. For $\phi\in A^{d},$ and a finite $F\subseteq \Gamma,$ let $\alpha^{F}\colon X\to A^{F},\phi^{F}_{\sigma}\colon \{1,\cdots,d\}\to A^{F}$ be defined by}
\[\alpha^{F}(x)(g)=\alpha(g^{-1}x)\]
\[\phi^{F}_{\sigma}(x)(g)=\phi(\sigma(g)^{-1}x).\]
 \emph{For $F\subseteq \Gamma$ finite, $\varespilon>0,$ let $\AP(\alpha,F,\varespilon,\sigma)$ be all $\phi\in  A^{d}$ so that}
\[\sum_{a\in A^{F}}|(\alpha^{F})_{*}(\mu)(\{a\})-(\phi^{F}_{\sigma})_{*}(u_{d})(\{a\})|<\varespilon.\]
\end{definition}

\begin{definition}\emph{Let $\Gamma$ be a countable discrete group with random sofic approximation $\kappa=(\kappa_{i})$ with $\kappa_{i}\in M(S_{d_{i}}^{\Gamma}).$ Let $(X,\mathcal{B},\mu)$ be a standard probability space with $\Gamma\actson (X,\mathcal{B},\mu)$ by measure-preserving transformations. Let $\mathcal{S}\subseteq \mathcal{B}$ be a subalgebra which is generating under the action of $\Gamma.$ Let $\alpha\colon X\to A,\beta\colon X\to B$ be finite $\mathcal{S}$-measurable observables (i.e. $\alpha^{-1}(\{a\})\in \mathcal{S},\beta^{-1}(\{b\})\in \mathcal{S}$ for $a\in A,b\in B$) and assume that $\alpha\leq \beta,$ and $\pi$ is as in the definition of $\alpha\leq\beta,$ set}
\[h_{\kappa,\mu}(\alpha;\beta,F,\varespilon)=\limsup_{i\to \infty}\frac{1}{d_{i}}\log\int_{S_{d_{i}}^{\Gamma}}|\pi^{d_{i}}(\AP(\beta,F,\varespilon,\sigma))|\,d\kappa_{i}(\sigma),\]
\[h_{\kappa,\mu}(\alpha;\beta)=\inf_{\substack{F\subseteq\mbox{\emph{ finite,}}\\ \varespilon>0}}h_{\kappa,\mu}(\alpha;\beta,F,\varepsilon),\]
\[h_{\kappa,\mu}(\alpha)=\inf_{\beta}h_{\kappa,\mu}(\alpha;\beta)\]
\[h_{\kappa,\mu}(X,\Gamma)=\sup_{\alpha}h_{\kappa,\mu}(\alpha).\]
	\emph{Where the infimum and supremum in the last two lines are over all finite $\mathcal{S}$-measurable observables $\alpha,\beta$ with $\alpha\leq \beta.$ }
\end{definition}
It is show in \cite{BowenGroupoid} Theorem 7.5 that this definition is independent of the choice of generating subalgebra, and in Theorem 9.5 that it agrees with the definition in the case of a topological model.

\begin{theorem}[Theorem 1.3 in \cite{Bowenfrandom}]\label{T:finvarianttheorem} Let $(X,\mu)$ be a standard probability space, and $\FF_{r}\actson (X,\mu)$ by measure-preserving transformations. Suppose that the action has a finite generator. Let $U=(u_{\Hom(\FF_{r},S_{n})}),$ then
\[h_{U,\mu}(X,\FF_{r})=f_{\mu}(X,\FF_{r}).\]
\end{theorem}

	Using additional results of Lewis Bowen, one can replace the assumption of having a generator with having a generator with finite Shannon entropy. Since we cannot find a proof in the literature, we include the simple proof below.
\begin{proposition}\label{P:finvarianttheorem} Let $(X,\mu)$ be a standard probability space, and $\FF_{r}\actson (X,\mu)$ by measure-preserving transformations. Suppose that the action has a generator with finite Shannon entropy. Let $U=(u_{\Hom(\FF_{r},S_{n})}),$ then
\[h_{U,\mu}(X,\FF_{r})=f_{\mu}(X,\FF_{r}).\]
\end{proposition}

\begin{proof}

	Let $\alpha$ be a generator with finite Shannon entropy. By the preceding Theorem, we may assume that $\alpha$ has range $\NN.$ Let
\[\pi_{n}\colon \NN\to \{1,\cdots,n\}\]
be defined by $\pi_{n}(k)=k$ if $1\leq k\leq n,$ and $\pi_{n}(k)=n$ if $k\geq n.$  Let $\alpha_{n}=\pi_{n}\circ \alpha.$  Let $\Sigma_{m}$ be the smallest $\Gamma$-invariant $\sigma$-algebra of measurable sets containing $\alpha_{m}^{-1}(\{j\})$ for $1\leq j\leq m,$  and $S$ the subalgebra of measurable sets (not necessarily a sub-sigma algebra) generated by $\{\alpha_{n}^{-1}(\{k\}):n,k\in \NN,k\leq n\}.$ It suffices to show that
\[f_{\mu}(\alpha_{n})\to f_{\mu}(X,\FF_{r})\]
\[f_{\mu}(\alpha_{n})\to h_{U,\mu}(X,\FF_{r}).\]

	For the first result, we know by \cite{Bowenfinvariant} that $f$ is upper semi-continuous on the space of partitions, in particular that
\[f_{\mu}(X,\FF_{r})=f_{\mu}(\alpha)\geq \limsup_{n\to \infty}f_{\mu}(\alpha_{n}).\]
On the other hand by \cite{BowenGun} Theorem 1.3,
\[f_{\mu}(\alpha)=f_{\mu}(\alpha_{n})+f_{\mu}(\alpha|\Sigma_{n}),\]
(see \cite{BowenGun} 1.2 for the definition of $f_{\mu}(\alpha|\Sigma_{n})$) and by \cite{BowenGun} Proposition 5.1,
\[f_{\mu}(\alpha|\Sigma_{n})\leq H(\alpha|\Sigma_{n})\leq H(\alpha|\alpha_{n})\to 0.\]
So
\[f_{\mu}(\alpha)\leq \liminf_{n\to \infty}f_{\mu}(\alpha_{n}),\]
and thus
\[f_{\mu}(\alpha_{n})\to f_{\mu}(\alpha)=f_{\mu}(X,\FF_{r}).\]

	For the second claim,  note that the smallest $\Gamma$-invariant complete sub-sigma algebra of measurable sets containing $S$ is all measurable subsets of $X.$ So we may use $S$ to compute $h_{U,\mu}(X,\FF_{r}).$ By a method of proof analogous  to Lemma 5.1 in  \cite{Bow}, we have
\begin{equation}\label{E:relativeentropyestimate}
h_{U,\mu}(\alpha;\gamma)\leq h_{U,\mu}(\beta;\gamma)+H(\alpha|\beta)
\end{equation}
if $\beta\leq \alpha\leq \gamma$ are finite observables. Thus if $\alpha\leq \gamma,\beta\leq \gamma$ are observables, then
\[h_{U,\mu}(\alpha;\gamma)\leq h_{U,\mu}(\alpha\vee \beta;\gamma)\leq h_{U,\mu}(\beta;\gamma)+H(\alpha|\beta).\]
 From this it follows that
\[h_{U,\mu}(X,\FF_{r})=\sup_{m}h_{U,\mu}(\alpha_{m}).\]
Now
\[h_{U,\mu}(\alpha_{m})\leq h_{U,\mu}(\alpha_{m};\alpha_{m})=f_{\mu}(\alpha_{m})\]
by the preceding Theorem. Thus,
\[h_{U,\mu}(X,\FF_{r})\leq \liminf_{m\to \infty}f_{\mu}(\alpha_{m}).\]
On the other hand, for $n\geq m$ we have by $(\ref{E:relativeentropyestimate}),$
\[h_{U,\mu}(\alpha_{n};\alpha_{n})\leq h_{U,\mu}(\alpha_{m};\alpha_{n})+H(\alpha_{n}|\alpha_{m}).\]
If $\beta$ is any finite $S$-measurable observable, then $\beta\leq \alpha_{n}$ for some $n,$  so
\[h_{U,\mu}(\alpha_{m})=\inf_{n}h_{U,\mu}(\alpha_{m};\alpha_{n}).\]
Thus
\[\limsup_{n\to \infty}f_{\mu}(\alpha_{n})\leq h_{U,\mu}(\alpha_{m})+H(\alpha|\alpha_{m}),\]
and letting $m\to \infty$ proves that
\[\limsup_{n\to \infty}f_{\mu}(\alpha_{n})\leq h_{U,\mu}(X,\FF_{r})\]
thus
\[f_{\mu}(\alpha_{n})\to h_{U,\mu}(X,\FF_{r}).\]

\end{proof}

	We first relate random sofic entropy to deterministic sofic entropy.

\begin{proposition}\label{P:randomlowerbound} Let $\Gamma$ be a countable discrete group, and $X$ a compact metrizable space with $\Gamma\actson X$ by homeomorphisms. Let $\kappa$ be a random sofic approximation of $\Gamma,$ and $\rho$ a dynamically generating pseudometric on $X.$ Then for any $\varespilon>0$
\[\inf_{\Sigma}h_{\Sigma}(\rho,\varepsilon)\leq h_{\kappa}(\rho,\varepsilon)\]
where the infimum is over all sofic approximations $\Sigma$ of $\Gamma.$

	If $\mu$ is a $\Gamma$-invariant Borel probability measure on $X,$ then
\[\inf_{\Sigma}h_{\Sigma,\mu}(\rho,\varepsilon)\leq h_{\kappa,\mu}(\rho,\varepsilon),\]
again with the infimum being over all sofic approximations $\Sigma$ of $\Gamma.$
\end{proposition}

\begin{proof}

	We do the proof only in the topological case, the proof for the measure-theoretic case is the same. Let $\kappa=(\kappa_{i})$ with $\kappa_{i}\in \Prob(S_{d_{i}}^{\Gamma}).$
Since
\[\inf_{\Sigma}h_{\Sigma}(\rho,\varepsilon)=\inf_{F,\delta}\inf_{\Sigma}h_{\Sigma}(\rho,F,\delta,\varespilon)\]
it suffices to show that if $F\subseteq \Gamma$ is finite, $\delta>0$ then
\[\inf_{\Sigma}h_{\Sigma}(\rho,F,\delta,\varepsilon)\leq h_{\kappa}(\rho,F,\delta,\varespilon).\]
Let $F_{n}$ be an increasing sequence of finite subsets of $\Gamma$ so that
\[\Gamma=\bigcup_{n=1}^{\infty}F_{n}.\]
	Let $i_{n}$ be an increasing sequence of integers so that
\[\frac{1}{d_{i}}\log \int S_{\varepsilon}(\Map(\rho,F,\delta,\sigma),\rho_{2})\,d\kappa_{i}(\sigma)\leq h_{\kappa}(\rho,F,\delta,\varepsilon)+2^{-n},\mbox{ for $i\geq i_{n}$}\]
\[u_{d_{i}}(\{j:\sigma(g)\sigma(h)(j)=\sigma(gh)(j)\})\geq 1-2^{-n},\mbox{ for all $g,h\in F_{n},$ for all $i\geq i_{n}$ and $\kappa_{i}$-almost every $\sigma$}\]
\[\kappa_{i}(\{\sigma:u_{d_{i}}(\{j:\sigma(g)(j)\ne \sigma(h)(j)\})\geq 1-2^{-n}\mbox{ for all $g\ne h$ in $F_{n}$}\})\geq 1-2^{-n}\mbox{ for all $i\geq i_{n}.$}\]
Let
\[A_{n,i}=\bigcap_{g,h\in F_{n}}\{\sigma:u_{d_{i}}(\{j:\sigma(g)\sigma(h)(j)=\sigma(gh)(j)\})\geq 1-2^{-n}\}\]
\[B_{n,i}=\bigcap_{g,h\in F_{n},g\ne h} \{\sigma:u_{d_{i}}(\{j:\sigma(g)(j)\ne \sigma(h)(j)\})\geq 1-2^{-n}\},\]
\[C_{n,i}=A_{n,i}\cap B_{n,i}.\]
Then for all $i\geq i_{n}$
\[\frac{1}{\kappa_{i}(C_{n,i})}\int_{C_{n,i}}S_{\varepsilon}(\Map(\rho,F,\delta,\sigma),\rho_{2})\,d\kappa_{i}(\sigma)\leq \frac{1}{1-2^{1-n}}\exp(d_{i}h_{\kappa}(\rho,F,\delta,\varepsilon)+d_{i}2^{-n}).\]
So we can find a $\sigma_{n}\in C_{n,i_{n}}$ so that
\[\frac{1}{d_{i_{n}}}\log S_{\varepsilon}(\Map(\rho,F,\delta,\sigma_{n}),\rho_{2})\leq \frac{1}{d_{i_{n}}}\log\left(\frac{1}{1-2^{1-n}}\right)+h_{\kappa}(\rho,F,\delta,\varepsilon)+2^{-n}.\]
Then $\Sigma=(\sigma_{n})_{n=1}^{\infty}$ is a sofic approximation with
\[h_{\Sigma}(\rho,F,\delta,\varepsilon)\leq h_{\kappa}(\rho,F,\delta,\varespilon).\]

\end{proof}

	The above statement may seem obvious, and even though the proof is simple, let us point out that it is not clear how to prove the analogous statement for supremums. Indeed, suppose we try to repeat the proof and find an increasing sequence of integers $i_{n}$ so that
\[\kappa_{i_{n}}(C_{n,i_{n}})\geq 1-2^{1-n}\]
and
\[\frac{1}{d_{i_{n}}}\log \int S_{\varepsilon}(\Map(\rho,F,\delta,\sigma),\rho_{2})\,d\kappa_{i}(\sigma)\geq h_{\kappa}(\rho,F,\delta,\varepsilon)-2^{-n}.\]
Then the best we can conclude is that
\[\int_{C_{n,i_{n}}} S_{\varepsilon}(\Map(\rho,F,\delta,\sigma),\rho_{2})\,d\kappa_{i}(\sigma)\geq \exp(d_{i_{n}}[_{\kappa}(\rho,F,\delta,\varepsilon)-2^{-n}])-2^{1-n}S_{\varepsilon}(X,\rho)^{d_{i_{n}}}.\]
And so there is some $\sigma_{n}\in C_{n,i_{n}}$ with
\[S_{\varepsilon}(\Map(\rho,F,\delta,\sigma_{n}),\rho_{2})\geq \exp(d_{i_{n}}[h_{\kappa}(\rho,F,\delta,\varepsilon)-2^{-n}])-2^{1-n}S_{\varepsilon}(X,\rho)^{d_{i_{n}}}.\]
However, this estimate is not good enough to show that
\[h_{\kappa}(X,\Gamma)\leq \sup_{\Sigma}h_{\Sigma}(X,\Gamma),\]
in fact
\[\exp(d_{i_{n}}[h_{\kappa}(\rho,F,\delta,\varepsilon)-2^{-n}])-2^{1-n}S_{\varepsilon}(X,\rho)^{d_{i_{n}}}\]
could  be negative.

	The difficulty here is that we need to control $S_{\varepsilon}(\Map(\rho,F,\delta,\varepsilon),\rho_{2})$ on the set where our random sofic approximations fails to be free. We shall prove that this can be done for measure-theoretic entropy if the action is essentially free. Recall that  an action $\Gamma\actson (X,\mu)$ is \emph{essentially free} if
\[\mu(\{x\in X:gx=x\})=0\]
for all $g\in \Gamma\setminus\{e\}.$ The idea is basically clear: if a sequence of almost multiplicative maps fails to be asymptotically free, then one cannot use them to model an essentially free action.

\begin{lemma} Let $\Gamma$ be a countable discrete group, and $X$ a compact metrizable space with $\Gamma\actson X$ by homeomorphisms. Let $\rho$ be a compatible metric on $X.$  Let $\mu$ be a $\Gamma$-invariant Borel probability measure on $X.$ Suppose that the action $\Gamma\actson (X,\mu)$ is essentially free. Then, for all $E\subseteq \Gamma\setminus\{e\}$ finite, $\eta>0,$ there is an $L\subseteq C(X)$ finite, $\delta>0,$ $F\subseteq \Gamma$ finite, so that if $\sigma\colon \Gamma\to S_{d}$ is a function, and $\Map(\rho,F,\delta,\sigma)\ne \varnothing,$ then for all $g\in E,$
\[|\{j:\sigma(g)(j)= j\}|\leq d\eta.\]
\end{lemma}

\begin{proof} It suffices to assume $E=\{g\}$ for some $g\in \Gamma\setminus\{e\}.$ Then,
\[\{x\in X:gx=x\}=\bigcap_{\delta>0}\{x\in X:\rho(gx,x)<\delta\}.\]
Hence, we may find a $\delta'>0$ so that
\[\mu(\{x\in X:\rho(gx,x)<\delta'\})<\eta.\]
Let
\[V=\{x\in X:\rho(gx,x)<\delta'\}\]
\[K=\left\{x\in X:\rho(gx,x)\leq \frac{\delta'}{2}\right\}.\]
Choose $f\in C(X)$ so that
\[\chi_{K}\leq f\leq \chi_{V}.\]
If $\phi\in \Map(\rho,\{g\},\delta,\{f\},\sigma),$ then
\[\frac{1}{d}\left|\{j:\rho(g\phi(j),\phi(\sigma(g)(j)))\geq \sqrt{\delta}\}\right|\leq \sqrt{\delta}.\]
Further,
\begin{align*}
\frac{1}{d}\left|\{j:\rho(g\phi(j),\phi(j))\leq \frac{\delta'}{2}\}\right|&=\phi_{*}(u_{d})(K)\\
&\leq \delta+\int f\,d\mu\\
&\leq \delta+\mu(V)\\
&\leq \delta+\eta.
\end{align*}
Thus,
\begin{align*}
\frac{1}{d}|\{j:\sigma(g)(j)=j\}|&\leq \delta+\eta+\sqrt{\delta}\\
&+\frac{1}{d}\left|\left\{j:\sigma(g)(j)=j,\rho(g\phi(j),\phi(\sigma(g)(j)))<\sqrt{\delta},\rho(g\phi(j),\phi(j))>\frac{\delta'}{2}\right\}\right|.
\end{align*}
Choose $\delta$  so that $\delta<\eta,\sqrt{\delta}<\eta$ and $\sqrt{\delta}<\frac{\delta'}{2}.$ We then see that
\[\frac{1}{d}|\{j:\sigma(g)(j)=j\}|\leq 3\eta,\]
since $\eta$ is arbitrary this proves the Lemma.

\end{proof}

\begin{proposition}\label{P:randomupperbound} Let $\Gamma$ be a countable discrete group with random sofic approximation $\kappa$, and $(X,\mu)$ a standard probability space. Suppose that $\Gamma\actson (X,\mu)$ is an essentially free measure-preserving action. Then,
\[h_{\kappa,\mu}(X,\Gamma)\leq \sup_{\Sigma}h_{\Sigma,\mu}(X,\Gamma)\]
where the supremum is over all sofic approximations of $\Gamma.$
\end{proposition}

\begin{proof}
	Let $\kappa=(\kappa_{i})$ with $\kappa_{i}\in \Prob(S_{d_{i}}^{\Gamma}).$ It is well-known that there is a compact metrizable space $Y,$ an action $\Gamma\actson Y$ by homeomorphisms, and a Borel probability measure $\nu$ on $Y$ so that $\Gamma\actson (Y,\nu)\cong \Gamma\actson (X,\mu)$ (e.g. take the spectrum of a weak$^{*}$-dense unital separable $C^{*}$-subalgebra of $L^{\infty}(X,\mu)$ for $Y$).  Thus we may assume $X$ is a compact metrizable space, and $\Gamma\actson X$ by homeomorphisms. Let $\rho$ be a compatible metric on $X.$ We may assume that
\[h_{\kappa,\mu}(X,\Gamma)>-\infty.\]
As
\[\sup_{\Sigma}h_{\Sigma,\mu}(X,\Gamma)=\sup_{\varepsilon>0}\sup_{\Sigma}h_{\Sigma,\mu}(\rho,\varespilon)\]
it is enough to show that for any $\varepsilon>0,$
\[\sup_{\Sigma}h_{\Sigma,\mu}(\rho,\varespilon)\geq h_{\kappa,\mu}(\rho,\varepsilon).\]

	Let $F_{n}'$ be an increasing sequence of finite subset of $\Gamma,$ so that
\[\Gamma=\bigcup_{n=1}^{\infty}F_{n}',\]
let $\delta_{n}'$ be a decreasing sequence of positive real numbers converging to zero, and let $L_{n}'$ be an increasing sequence of finite subsets of $C(X)$ so that
\[C(X)=\overline{\bigcup_{n=1}^{\infty}L_{n}}.\]
Applying the preceding Lemma, we may assume that there are an increasing sequence of finite subsets $F_{n},L_{n}$ of $\Gamma,C(X)$ with $F_{n}'\subseteq F_{n},L_{n}'\subseteq L_{n}$ and $0<\delta<\delta'$ so that if  $\sigma\colon \Gamma\to S_{d}$ is any function, and
\[ \Map(\rho,F_{n},\delta_{n},L_{n},\sigma)\ne \varnothing\]
then
\[u_{d}\left(\bigcap_{g,h\in F_{n}',g\ne h}\{j:\sigma(g)(j)\ne \sigma(h)(j)\}\right)\geq 1-2^{-n}.\]
Choose an increasing sequence of integers $i_{n}$ so that
\[\frac{1}{d_{i_{n}}}\log \int S_{\varespilon}(\Map(\rho,F_{n},L_{n},\delta_{n},\sigma),\rho_{2})\,d\kappa_{i_{n}}(\sigma)\geq h_{\kappa}(\rho,F_{n},L_{n},\delta_{n},\varespilon)-2^{-n},\]
\[\kappa_{i_{n}}(\{\sigma:u_{d_{i_{n}}}(\{j:\sigma(g)\sigma(h)(j)=\sigma(gh)(j)\})\geq 1-2^{-n}\mbox{ for all $g,h\in F_{n}$}\})=1.\]

	Since
\[h_{\kappa}(\rho,F_{n},L_{n},\delta_{n},\varespilon)\geq h_{\kappa}(\rho,\varepsilon)>-\infty,\]
we  can find a $\sigma_{n}\in S_{d_{i_{n}}}^{\Gamma}$ so that for all $g,h\in F_{n}$
\[u_{d_{i_{n}}}(\{j:\sigma_{n}(g)\sigma_{n}(h)(j)=\sigma_{n}(gh)(j)\})\geq 1-2^{-n},\]
and
\[S_{\varespilon}(\Map(\rho,F_{n},L_{n},\delta_{n},\sigma_{n}),\rho_{2})\geq \max(1,\exp(d_{i_{n}}[h_{\kappa}(\rho,F_{n},L_{n},\delta_{n},\varepsilon)-2^{-n}])).\]
So by the preceding Lemma,
\[u_{d_{i_{n}}}\left(\bigcap_{g,h\in F_{n}',g\ne h}\{j:\sigma_{n}(g)(j)\ne \sigma_{n}(h)(j)\}\right)\geq 1-2^{-n}.\]
If $m\geq n,$ then
\[S_{\varespilon}(\Map(\rho,F_{m},L_{m},\delta_{m},\sigma_{m}),\rho_{2})\leq S_{\varespilon}(\Map(\rho,F_{n},L_{n},\delta_{n},\sigma_{m}),\rho_{2})\]
so if we set $\Sigma=(\sigma_{n})_{n=1}^{\infty},$ then $\Sigma$ is a sofic approximation with
\[h_{\Sigma,\mu}(\rho,F_{n},L_{n},\delta_{n},\varespilon)\geq \max(0,h_{\kappa,\mu}(\rho,\varespilon)).\]
Letting $n\to \infty$ proves that
\[h_{\Sigma,\mu}(\rho,\varepsilon)\geq h_{\kappa,\mu}(\rho,\varepsilon).\]

\end{proof}

	Using results of Brandon Seward, we can apply this to $f$-invariant entropy of algebraic actions.
\begin{proposition} Let $X$ be a compact metrizable group, and $\FF_{r}\actson X$ by automorphisms, with $r>1.$ If the action is not essentially free (with respect to the Haar measure), then
\[h_{U,m_{X}}(X,\FF_{r})=(1-r)\log|X|,\]
if $X$ is finite, and
\[h_{U,m_{X}}(X,\FF_{r})=-\infty\]
if $X$ is infinite.
\end{proposition}

\begin{proof}
	
	The finite case follows from Theorem \ref{T:finvarianttheorem} and the definition of the $f$-invariant, so we assume $X$ is infinite. For $g\in \FF_{r},$ we let $\Fix_{g}(X)$ be the set of elements fixed by $g.$ Choose $g\in \FF_{r}\setminus\{e\}$ so that $\Fix_{g}(X)$ has positive measure. We first show that the entropy of $\ip{g}\actson (X,m_{X})$ is zero.

	For this, note that $\Fix_{g}(X)$ is a closed subgroup of $X,$ and by assumption it has positive measure. Since $X$ is compact, this forces
\[[X:\Fix_{g}(X)]<\infty.\]
Set
\[Y=\bigcap_{x\in X}x\Fix_{g}(X)x^{-1},\]
then $Y$ is a closed normal subgroup of $X,$ and it is a standard fact that
\[[X:Y]<\infty.\]
Since $g$ acts by automorphisms,  we know that $Y$ is $\ip{g}$-invariant. We have the following exact sequence of compact groups with $\ip{g}$-actions
\[\begin{CD}
1 @>>> Y @>>> X @>>> X/Y @>>> 1.
\end{CD}\]
Since $\ip{g}$ is cyclic we may apply  Yuzvinski\v\i's addition   formula (see \cite{Yuz}) to see that
\[h_{m_{X}}(X,\ip{g})=h_{m_{Y}}(Y,\ip{g})+h_{m_{X/Y}}(X/Y,\ip{g})=h_{m_{Y}}(Y,\ip{g}),\]
since $X/Y$ is finite, and $\ip{g}\cong \ZZ.$ Since $\ip{g}\actson Y$ trivially, we know that
\[h_{m_{Y}}(Y,\ip{g})=0,\]
so
\[h_{m_{X}}(X,\ip{g})=0.\]

	We now prove the proposition. Let $\alpha$ be a finite partition of $X,$ then
\[h_{U,\mu}(\alpha)\leq h_{U,\mu}(\alpha;\alpha).\]
Let $S$ be the $\sigma$-algebra of measurable sets generated by $\{h\alpha:h\in \FF_{r}\},$ let $(Z,\zeta)$ be the factor of $(X,m_{X})$ corresponding to $S.$ By Theorem \ref{T:finvarianttheorem}, we have
\[h_{U,\mu}(\alpha)\leq h_{U,\mu}(\alpha;\alpha)=h_{U,\zeta}(Z,\FF_{r})=f_{\zeta}(Z,\FF_{r}).\]
Since entropy for actions of amenable groups decreases under factor maps, we know that
\[h_{\zeta}(Z,\ip{g})=0.\]
Thus by \cite{SewardFree} Theorem 1.6, we know that
\[f_{\zeta}(Z,\FF_{r})=-\infty.\]
Thus,
\[h_{U,\mu}(\alpha)=-\infty,\]
as $\alpha$ is arbitrary we know that
\[h_{U,\mu}(X,\FF_{r})=-\infty.\]

\end{proof}

\begin{cor} Let $r\in\NN$ with $r>1.$ Let $h\in M_{m,n}(\ZZ(\FF_{r}))$ and suppose $\lambda(h)$ is injective. Let $U=(u_{\Hom(\FF_{r},S_{n})}),$ then
\[h_{U,m_{X_{h}}}(X_{h},\FF_{r})\leq \log \Det^{+}_{L(\FF_{r})}(h)\]
with equality if $m=n.$ In particular, if the action $\FF_{r}\actson (X_{h},m_{X_{h}})$ has a generator with finite Shannon entropy, then
\[f_{m_{X_{h}}}(X_{h},\FF_{r})\leq \log \Det^{+}_{L(\FF_{r})}(h)\]
with equality if $m=n.$
\end{cor}

\begin{proof}
	The ``in particular" part is a consequence of Proposition \ref{P:finvarianttheorem}. The lower bound when $m=n,$ follows Proposition \ref{P:randomlowerbound}, and the proof of Theorem \ref{T:measureentropy}. The upper bound is a consequence of the preceding proposition and Proposition \ref{P:randomupperbound}.

\end{proof}

\begin{cor} Let $r\in\NN$ with $r>1.$ Let $h\in M_{m,n}(\ZZ(\FF_{r}))$ and suppose that $\lambda(f)$ has dense image but is not injective (so necessarily $m<n.$) Let $U=(u_{\Hom(\FF_{r},S_{n})}),$ then
\[h_{U,m_{X_{h}}}(X_{h},\FF_{r})=\infty.\]
\end{cor}

\begin{proof} Automatic from Proposition \ref{P:randomlowerbound} and the proof of Theorem \ref{T:infinitemeasureentropyexample}.

\end{proof}
	Using results of Brandon Seward, we can extend these results to \emph{virtually} free groups. In \cite{SewardSubgroup}, it is proved that if $\Lambda\subseteq \FF_{r}$ has finite index, and $(X,\mu)$ is a standard probability space with $\FF_{r}\actson (X,\mu)$ by measure-preserving transformations, then if the action has a generator with finite Shannon entropy,
\begin{equation}\label{E:virtuallyfree}
f_{\mu}(X,\Lambda)=[\FF_{r}\colon \Lambda]f_{\mu}(X,\FF_{r}).
\end{equation}
Thus, if $\Gamma$ is a virtually free group, and $\Gamma\actson (X,\mu)$ is a standard probability measure-preserving action with a finite Shannon entropy generator, we can define
\[f_{\mu}(X,\Gamma)=\frac{1}{[\Gamma\colon \Lambda]}f_{\mu}(X,\Lambda)\]
where $\Lambda\subseteq \Gamma$ is any finite index free group. By $(\ref{E:virtuallyfree}),$ this does not depend on the choice of finite index subgroup.

\begin{cor}Let $\Gamma$ be a virtually free group, and $h\in M_{m,n}(\ZZ(\Gamma))$ be injective as a left multiplication operator on $\ell^{2}(\Gamma)^{\oplus n}.$ Suppose $\Gamma\actson (X_{h},m_{X_{h}})$ has  a generator with finite Shannon entropy, then
\[f_{m_{X_{h}}}(X_{h},m_{X_{h}})\leq \log \Det^{+}_{L(\Gamma)}(h)\]
with equality if $m=n.$
\end{cor}

\begin{proof}

	Choose a finite index free subgroup $\Lambda$ of $\Gamma.$ Using a system of coset representatives,
\[\ZZ(\Gamma)^{\oplus n}\cong \ZZ(\Lambda)^{\oplus [\Gamma\colon \Lambda]n},\]
\[\ZZ(\Gamma)^{\oplus m}\cong \ZZ(\Lambda)^{\oplus [\Gamma\colon \Lambda]m},\]
as $\ZZ(\Lambda)$-modules. The map
\[\ZZ(\Gamma)^{\oplus m}\to \ZZ(\Gamma)^{\oplus n}\]
given by right multiplication by $h$ is $\ZZ(\Lambda)$-modular, hence under these isomorphisms corresponds to an element $\widetilde{h}\in M_{[\Gamma\colon \Lambda]m,[\Gamma\colon \Lambda]n}(\ZZ(\Lambda)).$ As operators on $\ell^{2},$ we can obtain $\widetilde{h}$ by regarding $\ell^{2}(\Gamma)^{\oplus n},\ell^{2}(\Gamma)^{\oplus m}$ as isomorphic to $\ell^{2}(\Lambda)^{\oplus [\Gamma\colon \Lambda]n},\ell^{2}(\Lambda)^{\oplus [\Gamma\colon \Lambda]n}$ as representations of $\Lambda$ (using the same system of coset representatives as before). Using the well-known formula (see \cite{Luck} Theorem 1.12 (6)).
\[\dim_{L(\Gamma)}(\mathcal{H})=\frac{1}{[\Gamma\colon \Lambda]}\dim_{L(\Lambda)}(\mathcal{H}),\]
for $\Gamma$-invariant subspaces of $\ell^{2}(\NN,\ell^{2}(\Gamma))$ we have
\[\frac{1}{[\Gamma\colon \Lambda]}\mu_{|\widetidle{h}|}=\mu_{|h|}.\]
	By the preceding Corollary,
\[f_{m_{X_{h}}}(X_{h},\Gamma)=\frac{1}{[\Gamma\colon \Lambda]}f_{m_{X_{\widetilde{h}}}}(X_{\widetilde{h}},\Lambda)\leq \frac{1}{[\Gamma\colon \Lambda]}\log \Det^{+}_{L(\Lambda)}(\widetilde{h})=\log \Det^{+}_{L(\Gamma)}(h),\]
with equality if $m=n.$

\end{proof}

\subsection{Applications to Metric Mean Dimension and Entropy}

	The first application is a complete classification of when a finitely presented algebraic action has finite topological entropy. We use $\mdim_{\Sigma,M},\mdim_{\Sigma}$ for metric mean dimension and mean dimension respectively (see \cite{Li} for the definition). For a $\ZZ(\Gamma)$-module $A,$ we use $\vr(A)$ for the von Neumann rank of $A,$ (see \cite{LiLiang} for the definition).

\begin{theorem}\label{T:whenisitinfinite}  Let $\Gamma$ be a countable discrete sofic group with sofic approximation $\Sigma=(\sigma_{i}\colon \Gamma\to S_{d_{i}}).$ Suppose that $A$ is a finitely-presented $\ZZ(\Gamma)$-module.
\begin{enumerate}[(i)]
\item
The following are equivalent:

(1) $h_{\Sigma}(\widehat{A},\Gamma)<\infty,$

(2) $\mdim_{\Sigma,M}(\widehat{A},\Gamma)=0,$

(3) $\vr(A)=0.$

\item Suppose that $\Gamma$ is residually finite. Let $\Gamma_{i}\triangleleft \Gamma$ be a decreasing sequence with $[\Gamma\colon \Gamma_{i}]<\infty$ for all $i\in\NN$ and
\[\bigcap_{i=1}^{\infty}\Gamma_{i}=\{e\}.\]
Let $\Sigma=(\sigma_{i}\colon\Gamma\to \Sym(\Gamma/\Gamma_{i}))_{i=1}^{\infty}$ be defined by
\[\sigma_{i}(g)(x\Gamma_{i})=gx\Gamma_{i}.\]
Then $(1)-(3)$ are equivalent to

(4) $\mdim_{\Sigma}(\widehat{A},\Gamma)=0.$

\end{enumerate}

\end{theorem}

\begin{proof}

(i): It is easy to see that $(1)$ implies $(2).$ The equivalence of $(3)$ and $(2)$ is the content of Theorem 5.1 in \cite{Me4}.

	Suppose that $\vr(A)=0.$ We may assume $A=\ZZ(\Gamma)^{\oplus n}/r(f)(\ZZ(\Gamma)^{\oplus m})$ with $f\in M_{m,n}(\ZZ(\Gamma)).$ By Lemma 5.4 in \cite{LiLiang} $\vr(A)=\dim_{L(\Gamma)}(\ker \lambda(f)),$
 so our hypothesis implies $0=\dim_{L(\Gamma)}(\ker \lambda(f)).$
Thus $\lambda(f)$ is injective, so by Theorem \ref{T:upperbound}
\[h_{\Sigma}(\widehat{A},\Gamma)\leq \log\Det^{+}_{L(\Gamma)}(f)\leq \log\|\widehat{f}\|_{1}<\infty.\]
So we have shown that $(3)$ implies $(1).$

(ii): By Theorem 6.2 in \cite{Me4}, we know that
\[\vr(A)=\mdim_{\Sigma}(\widehat{A},\Gamma)\]
and so (4) is equivalent to $(3).$
\end{proof}

It is easy to prove Theorem 1.1 (i) from the above.

\begin{cor}\label{C:whenisitinfinite}  Let $\Gamma$ be a countable discrete sofic group with sofic approximation $\Sigma=(\sigma_{i}\colon \Gamma\to S_{d_{i}})$ and fix $f\in M_{m,n}(\ZZ(\Gamma)).$ Then $h_{\Sigma}(X_{f},\Gamma)<\infty$ if and only if $\lambda(f)$ is injective.

\end{cor}

\begin{proof} Set
\[A=\ZZ(\Gamma)^{\oplus n}/r(f)(\ZZ(\Gamma)^{\oplus m})\]
(so $X_{f}=\widehat{A}$). Then arguing as in the preceding theorem we see that $\vr(A)=0$ if and only if $\lambda(f)$ is injective. Hence it follows by the preceding theorem that $h_{\Sigma}(X_{f},\Gamma)<\infty$ if and only if $\lambda(f)$ is injective.

\end{proof}

 For the next application, we need to single out a nice class of groups.
\begin{definition}\emph{Let $\Gamma$ be a countable discrete group.  Let $\mathcal{O}(\Gamma)$ denote the abelian subgroup of $\QQ$ generated by $|\Lambda|^{-1}$ for all finite subgroups $\Lambda\subseteq \Gamma.$ We say that $\Gamma$ has the} Strong Atiyah property \emph{if $\dim_{L(\Gamma)}(\ker(\lambda(f)))\in \mathcal{O}(\Gamma)$ for all $f\in M_{m,n}(\ZZ(\Gamma)).$}\end{definition}

For example, let $\mathcal{C}$ be the smallest class of groups containing all free groups, and closed under direct unions and extensions  with elementary amenable quotients, then by Theorem 10.19 in \cite{Luck} we know that $\Gamma$ has the strong Atiyah property.

\begin{theorem} Let $\Gamma$ be a countable discrete sofic group with the Strong Atiyah property, and such that
\[\sup\{|\Lambda|:\Lambda\subseteq \Gamma\mbox{ is finite }\}<\infty.\]
Let $\Sigma$ be a  sofic approximation of $\Gamma.$ Let $A$ be a finitely generated $\ZZ(\Gamma)$-module. Then $\mdim_{\Sigma,M}(\widehat{A},\Gamma)=0$ if and only if $h_{\Sigma}(\widehat{A},\Gamma)<\infty$.\end{theorem}
\begin{proof} Our proof is essentially the same as the proof of Corollary 9.5 in \cite{LiLiang}. It is straightforward to see that
\[h_{\Sigma}(\widehat{A},\Gamma)<\infty\]
implies that
\[\mdim_{\Sigma,M}(\widehat{A},\Gamma)=0.\]

	Conversely, suppose that
\[\mdim_{\Sigma,M}(\widehat{A},\Gamma)=0.\]
Without loss of generality $A=\ZZ(\Gamma)^{\oplus n}/B,$ for some $\ZZ(\Gamma)$-submodule $B$ of $\ZZ(\Gamma)^{\oplus n}.$  Write
\[B=\bigcup_{m=1}^{\infty}B_{m}\]
where $B_{m}$ are finitely generated. By the proof of Lemma 2.3 in \cite{Me4}, we know that
\[0=\vr(A)=\lim_{m\to \infty}\vr(\ZZ(\Gamma)^{\oplus n}/B_{m}).\]
Our assumptions imply that $\mathcal{O}(\Gamma)$ is discrete and thus for all large $m$
\[\vr(\ZZ(\Gamma)^{\oplus n}/B_{m})=0.\]
The preceding theorem and Theorem 5.1 in \cite{Me4} then imply that
\[h_{\Sigma}((\ZZ(\Gamma)^{\oplus n}/B_{m})^{\widehat{}},\Gamma)<\infty.\]
As
\[\widehat{A}\subseteq (\ZZ(\Gamma)^{\oplus n}/B_{m})^{\widehat{}},\]
we find that
\[h_{\Sigma}(\widehat{A},\Gamma)\leq h_{\Sigma}((\ZZ(\Gamma)^{\oplus n}/B_{m})^{\widehat{}},\Gamma)<\infty.\]

\end{proof}

\subsection{Failure of Yuzvinski\v\i\ Addition Formula and Relation to Torsion}\label{S:Yuz}

	In this section, we study when the Yuzvinski\v\i\ addition formula fails for a non-amenable sofic group $\Gamma.$ As we shall see, the Yuzvinski\v\i\ addition formula will automatically fail when $\Gamma$ contains a nonabelian free group. We also show that it fails if the $L^{2}$-torsion of $\Gamma$ is defined and nonzero. Unfortunately, we cannot come up with an example of a sofic group $\Gamma$ not containing a free subgroup, and with nonzero $L^{2}$-torsion, so at this stage this is just another point of view on the Yuzvinski\v\i\ addition formula.

\begin{definition}\emph{Let $\Gamma$ be a countable discrete sofic group with sofic approximation $\Sigma.$ Let $\mathcal{C}$ be a class of $\ZZ(\Gamma)$-modules. We say that $(\Gamma,\Sigma)$ fails Yuzvinski\v\i's addition formula for the class $\mathcal{C}$ if there is an exact sequence of $\ZZ(\Gamma)$ modules}
\[\begin{CD}
0 @>>> A @>>> B @>>> C @>>> 0
\end{CD}\]
\emph{ with $A,B,C\in \mathcal{C}$ such that}
\[h_{\Sigma}(\widehat{B},\Gamma)\ne h_{\Sigma}(\widehat{A},\Gamma)+h_{\Sigma}(\widehat{C},\Gamma).\]
\end{definition}
Let us first note that failure of the Yuzvinski\v\i\ addition formula is closed under supergroups. For this, we need the notion of coinduction.

\begin{definition}\emph{ Let $\Lambda\subseteq \Gamma$ be countable discrete groups. Let $X$ be a compact metrizable space and $\Lambda\actson X$ by homeomorphisms. Let}
\[Y=\{f\colon \Gamma\to X:f(g\lambda)=\lambda^{-1}f(g)\mbox{\emph{ for all $g\in \Gamma,\lambda\in \Lambda$}}\},\]
\emph{and give $Y$ the product topology. Then $Y$ is a compact metrizable space and  $\Gamma\actson Y$ by homeomorphisms as follows:}
\[(gf)(x)=f(g^{-1}x).\]
\emph{The action $\Gamma\actson Y$ is called the} coinduced action of \emph{$\Lambda\actson X$}.\end{definition}
We use $\Gamma/\Lambda$ for the set of left cosets of $\Lambda$ in $\Gamma.$ Suppose we choose a section $s\colon\Gamma/\Lambda\to \Gamma,$
i.e.
$s(a)\Lambda=a.$
And consider the corresponding cocycle $c\colon\Gamma\times\Gamma/\Lambda\to \Lambda$
given by
\[gs(a)=s(ga)c(g,a).\]
Then it is not hard to show that the coinduced action is isomorphic to the action $\Gamma\actson X^{\Gamma/\Lambda}$ given by
\[(gx)(a)=c(g^{-1},a)^{-1}x(g^{-1}a).\]

\begin{proposition}\label{P:coinducent}\emph{Let $\Lambda\subseteq \Gamma$ be countable discrete sofic groups, and let $\Sigma$ be a sofic approximation of $\Gamma.$ Let $X$ be a compact metrizable space and $\Lambda\actson X$ by homeomorphisms. Let $\Gamma\actson Y$ be the coinduced action. Then}
\[h_{\Sigma}(Y,\Gamma)=h_{\Sigma\big|_{\Lambda}}(X,\Lambda).\]
\end{proposition}

\begin{proof} Let $\Sigma=(\sigma_{i}\colon \Gamma\to S_{d_{i}}).$ Let $\rho$ be a compatible metric on $X.$ Define a dynamically generating pseudometric on $Y$ by
$\widetilde{\rho}(\alpha,\beta)=\rho(\alpha(e),\beta(e)).$
Given $F\subseteq \Gamma$ finite, $\delta>0,$ and $\phi\in \Map(\widetidle{\rho},F,\delta,\sigma_{i}),$ let
$\alpha_{\phi}\colon \{1,\cdots,d_{i}\}\to X$
be given by
\[\alpha_{\phi}(j)=\phi(j)(e).\]
Then $\alpha_{\phi}\in \Map(\rho,F\cap \Lambda,\delta,\sigma_{i})$ and
$\rho(\alpha_{\phi},\alpha_{\psi})=\widetidle{\rho}(\phi,\psi)$
for $\phi,\psi\in \Map(\widetidle{\rho},F,\delta,\sigma_{i}).$ Thus
\[h_{\Sigma}(Y,\Gamma)\leq h_{\Sigma\big|_{\Lambda}}(X,\Lambda).\]

	For the reverse inequality, let $s\colon \Gamma/\Lambda\to \Gamma$ be a section of the quotient map, i.e.
$s(c)\Lambda=c,$
additionally we assume that $s(\Lambda)=e.$ Let $c\colon \Gamma\times \Gamma\to \Lambda$ be the induced cocycle given by
\[gs(h\Lambda)=s(gh\Lambda)c(g,h).\]
As explained before the proposition, we may regard $Y$ as $X^{\Gamma/\Lambda}$ with the action of $\Gamma$ given by
\[(gx)(a)=c(g^{-1},a)^{-1}x(g^{-1}a).\]
We use the dynamically generating pseudometric $\widetilde{\rho}$ on $X^{\Gamma/\Lambda}$ given by
\[\widetilde{\rho}(x,y)=\rho(x(\Lambda),y(\Lambda)).\]

	Let $F\subseteq\Gamma$ be finite, $\delta>0,$ and let $F'\subseteq \Lambda$ finite, $ \delta'>0$ to be determined. Given $\alpha\in \Map(\rho,F',\delta',\sigma_{i}),$ let $\phi_{\alpha}\colon \{1,\cdots,d_{i}\}\to Y$ be given by
\[\phi_{\alpha}(j)(a)=\alpha(\sigma_{i}(s(a)^{-1})(j)).\]
Set $F'=\{c(g,e):g\in F\cup F^{-1}\}\cup \{e\}\cup \{c(g,e)^{-1}:g\in F\cup F^{-1}\}.$ We  claim that if $\delta'>0$ is sufficiently small, the for all sufficiently large $i$ and all $\alpha\in \Map(\rho,F',\delta',\sigma_{i}),$ we have $\phi_{\alpha}\in \Map(\widetilde{\rho},F,\delta,\sigma_{i}).$ Note that
\[(g\phi_{\alpha})(j)(\Lambda)=c(g^{-1},\Lambda)^{-1}\alpha(\sigma_{i}(s(g^{-1}\Lambda)^{-1})(j)),\]
\[\phi_{\alpha}(\sigma_{i}(g)(j))(\Lambda)=\alpha(\sigma_{i}(g)(j)).\]
So
\begin{equation}\label{E:coinductionequiv1}
\widetilde{\rho}_{2}(g\phi_{\alpha},\phi_{\alpha}\circ\sigma_{i}(c(g^{-1},\Lambda)^{-1})\sigma_{i}(s(g^{-1}\Lambda)^{-1}))<\delta',
\end{equation}
and by soficity,
\begin{equation}\label{E:coinductionequiv2}
\sup_{\alpha\in X^{d_{i}}}\widetilde{\rho}_{2}(\phi_{\alpha}\circ\sigma_{i}(c(g^{-1},\Lambda)^{-1})\sigma_{i}(s(g^{-1}\Lambda)^{-1}),\phi_{\alpha}\circ\sigma_{i}(c(g^{-1}\Lambda)^{-1}s(g^{-1},\Lambda)^{-1}))\to_{i\to\infty}0.
\end{equation}
As $s(\Lambda)=e,$ we have
$g^{-1}=s(g^{-1}\Lambda)c(g^{-1},\Lambda)$
so
\begin{equation}\label{E:coinductionequiv3}
g=c(g^{-1},\Lambda)^{-1}s(g^{-1}\Lambda)^{-1}.
\end{equation}
Now choose $\delta'$ to be any number less than $\delta.$ Then equations $(\ref{E:coinductionequiv1}),(\ref{E:coinductionequiv2}),(\ref{E:coinductionequiv3})$ show that for all large $i,$ and for any $\alpha\in\Map(\rho,F',\delta',\sigma_{i})$ we have that $\phi_{\alpha}\in\Map(\rho,F,\delta,\sigma_{i}).$ As
$\rho_{2}(\alpha,\beta)=\widetilde{\rho}_{2}(\phi_{\alpha},\phi_{\beta}),$
we see that
\[h_{\Sigma\big|_{\Lambda}}(X,\Lambda)\leq h_{\Sigma}(Y,\Gamma).\]

\end{proof}

We can show that the coinduction of an algebraic action is an algebraic action.

\begin{proposition}\label{P:algcoind} Let $\Lambda\subseteq \Gamma$ be countable discrete groups. Let $A$ be a $\ZZ(\Lambda)$-module, and let
\[B=\ZZ(\Gamma)\otimes_{\ZZ(\Lambda)}A.\]
	Then $\Gamma\actson \widehat{B}$ is the coinduced action of $\Lambda\actson \widehat{A}.$

\end{proposition}

\begin{proof} Let $Y=\{f\colon \Gamma\to \widehat{A}:f(g\lambda)=\lambda^{-1}f(g),g\in \Gamma,\lambda\in \Lambda\}$ with $\Gamma$ action as in the definition of coinduction. We have a map
$\Phi\colon \widehat{B}\to Y$
given by
$\Phi(\chi)(g)(a)=\chi(g\otimes a).$
Further if $f\in Y,$ we may define
$B_{f}\colon \Gamma\times A\to \TT$
by
$B_{f}(g,a)=f(g)(a).$
Then
$B_{f}(g\lambda,a)=B_{f}(g,\lambda a)$
and so there is a $\Psi(f)\in \widehat{B},$ so that $\Psi(f)(g\otimes a)=B_{f}(g,a).$ It is easy to check that $\Phi,\Psi$ are continuous $\Gamma$-equivariant, and inverse to each other.

\end{proof}

\begin{cor}\label{C:supergroup} Let $\Gamma$ be a countable discrete sofic group with sofic approximation $\Sigma.$ Suppose that $\Lambda\subseteq \Gamma,$ and that $(\Lambda,\Sigma\big|_{\Lambda})$ fail Yuzvinski\v\i's addition formula for a class of $\ZZ(\Lambda)$-modules $\mathcal{C}.$ Let $\mathcal{C}'$ be the class of all $\ZZ(\Gamma)$-modules of the form $\ZZ(\Gamma)\otimes_{\ZZ(\Lambda)}A$ for $A\in \mathcal{C}.$ Then $(\Gamma,\Sigma)$ fails Yuzvinski\v\i's addition formula for the class $\mathcal{C}'.$ In particular, if $\Gamma$ contains a nonabelian free subgroup, then $(\Gamma,\Sigma)$ fails Yuzvinski\v\i's addition formula for the class of finitely presented modules.
\end{cor}

\begin{proof}
	The first half is a combination of Propositions \ref{P:coinducent} and \ref{P:algcoind}, along with the fact (left as an exercise) that $\ZZ(\Gamma)\otimes_{\ZZ(\Lambda)}?$ preserves exact sequences. For the second half, we use that Ornstein-Weiss at the end of \cite{OrnWeiss} found a continuous surjective homomorphism
\[(\ZZ/2\ZZ)^{\FF_{2}}\to ((\ZZ/2\ZZ)^{\oplus 2})^{\FF_{2}}.\]
And this gives a counterexample to the Yuzvinski\v\i\ addition formula for $\FF_{2}.$

\end{proof}

	Lastly, we relate the existence of a Yuzvinski\v\i\ addition formula to the $L^{2}$-torsion of $\ZZ(\Gamma)$-modules $A.$

\begin{definition}\emph{Let $\Gamma$ be a countable discrete sofic group, and $A$ a $\ZZ(\Gamma)$-module. A} partial resolution of $A$ \emph{is an exact sequence of $\ZZ(\Gamma)$-modules of the form}
\[\begin{CD}
\ZZ(\Gamma)^{\oplus n_{k}} @>>> \ZZ(\Gamma)^{\oplus n_{k-1}} @>>> \cdots \ZZ(\Gamma)^{\oplus n_{0}} @>>> A @>>> 0.
\end{CD}\]
\emph{Suppose  the map $\ZZ(\Gamma)^{\oplus n_{j}}\to \ZZ(\Gamma)^{\oplus n_{j-1}}$ is given by $r(f_{j})$ for some $f_{j}\in M_{n_{j},n_{j-1}}(\ZZ(\Gamma)).$ If  we then have $\overline{\im(\lambda(f_{j-1}))}=\ker(\lambda(f_{j})),$ we call this resolution a} $\ell^{2}$-partial resolution of $A.$ \emph{ We call $k$ the} length of the resolution. \emph{We write $\mathcal{C}_{*}$ for the complex}
\[\begin{CD}
\ZZ(\Gamma)^{\oplus n_{k}} @>>> \ZZ(\Gamma)^{\oplus n_{k-1}} @>>> \cdots \ZZ(\Gamma)^{\oplus n_{0}} @>>> A @>>> 0,
\end{CD}\]
\emph{and sometimes say $\mathcal{C_{*}}\to A$ is a $\ell^{2}$-partial resolution. If $\mathcal{C^{*}}\to A$ is a $\ell^{2}$-partial resolution, we define the} $L^{2}$-torsion \emph{of $\mathcal{C}_{*}$ by}
\[\rho^{(2)}(\mathcal{C}_{*},\Gamma)=\sum_{j=1}^{k}(-1)^{j+1}\log \Det^{+}_{L(\Gamma)}(f_{j}).\]
\emph{We say that $A$ is} $\ell^{2}-FL$ \emph{if there is an exact sequence of the form}
\[\begin{CD}
0 @>>>\ZZ(\Gamma)^{\oplus n_{k}} @>>> \ZZ(\Gamma)^{\oplus n_{k-1}} @>>> \cdots \ZZ(\Gamma)^{\oplus n_{0}} @>>> A @>>> 0,
\end{CD}\]
\emph{where if $f_{j}$ is as before, then $\overline{\im(\lambda(f_{j-1}))}=\ker(\lambda(f_{j})).$ If we again write $\mathcal{C}_{*}$ for this complex, then the $L^{2}$-torsion of $\mathcal{C}_{*}$ will be called the $L^{2}$-torsion of $A.$ We say that the} $L^{2}$-torsion of $\Gamma$ is defined, \emph{ if the trivial $\ZZ(\Gamma)$-module $\ZZ$ is of type $\ell^{2}$-FL, and in this case define the} $L^{2}$-torsion of $\Gamma$ \emph{by $\rho^{(2)}(\Gamma)=\rho^{(2)}(\ZZ,\Gamma).$ }
	
\end{definition}

	One remark about this definition. By Corollary \ref{C:integrability}, we know that if $\Gamma$ is sofic then $\Det^{+}_{L(\Gamma)}(f)\geq 1$ for all $f\in M_{m,n}(\ZZ(\Gamma)).$ By the results in Chapter 3 of \cite{Luck} if $A$ is $\ell^{2}$-FL, then  the $L^{2}$-torsion of $A$ does not depend on the choice of $\ell^{2}$-free resolution. We would also like to point out that, at least to our knowledge, sofic groups are the largest class of groups for which it is known that $\Det^{+}_{L(\Gamma)}(f)\geq 1$ for all $f\in M_{m,n}(\ZZ(\Gamma)).$ That is there is no group $\Gamma$ which is not known to be sofic, yet which is known to have $\Det^{+}_{L(\Gamma)}(f)\geq 1$ for all $f\in M_{m,n}(\ZZ(\Gamma)).$ We have the following result. It is mostly a nice remark, as most examples that can be obtained this way can also be obtained from other methods. However, it may eventually be of interest.

\begin{proposition}\label{P:YAF} Let $\Gamma$ be a countable discrete sofic group with sofic approximation $\Sigma.$ Suppose that $\Gamma$ satisfies Yuzvinski\v\i 's addition formula for finitely presented algebraic actions. That is, for every exact sequence
\[\begin{CD}
0 @>>> A @>>> B @>>> C @>>> 0
\end{CD}\]
of finitely presented $\ZZ(\Gamma)$-modules, we have
\[h_{\Sigma}(\widehat{B},\Gamma)= h_{\Sigma}(\widehat{A},\Gamma)+h_{\Sigma}(\widehat{C},\Gamma).\]
Let $A$ be a finitely presented $\ZZ(\Gamma)$-module with $\vr(A)=0.$  Let $\mathcal{C}_{*}\to A$ be a $\ell^{2}$-partial resolution of $A$. Then
\[h_{\Sigma}(\widehat{A},\Gamma)\leq \rho^{(2)}(\mathcal{C}_{*}),\]
if $\mathcal{C}_{*}$ has odd length and
\[h_{\Sigma}(\widehat{A},\Gamma)\geq \rho^{(2)}(\mathcal{C}_{*})\]
if $\mathcal{C}_{*}$ has even length.
Further if $A$ is $\ell^{2}$-FL, then
\[h_{\Sigma}(\widehat{A},\Gamma)=\rho^{(2)}(A,\Gamma).\]

\end{proposition}

\begin{proof} Once we know Theorems \ref{T:upperbound} and \ref{T:squarematrix}, the proof is the same as the proof of Theorem 1.1 in \cite{LiThom}.

\end{proof}

	This theorem may be useful to disprove the Yuzvinski\v\i\ addition formula for certain groups. For example, we may be able to directly compute both sides and show they are not equal (this is what we will do to get our concrete counterexamples in this work). Additionally it might be possible to  find an $\ell^{2}$-FL $\ZZ(\Gamma)$-module $A$ for which $\rho^{(2)}(A,\Gamma)<0,$ and so cannot be the entropy of some action. For example, we can use \cite{Luck} Theorem 3.152 and the preceding discussion to give another proof of the fact that fundamental groups of certain hyperbolic manifolds must fail the Yuzvinski\v\i\ addition formula.
 Additionally, one may be able to compute the $L^{2}$-torsion of some partial resolution $\mathcal{C}_{*}\to A,$ if for example $\mathcal{C}_{*}$ has odd length and has negative torsion, then we get a contradiction to the above inequality. We now show that if the $L^{2}$-torsion of $\Gamma$ is defined and nonzero, then $\Gamma$ fails Yuzvinski\v\i 's addition formula. For this, we need to know that a trivial action of a sofic group has zero entropy (see \cite{KerrLi2} Corollary 8.5, or \cite{Me6} Corollary 7.11).

\begin{cor}\label{C:YuzFail} Let $\Gamma$ be a countable discrete sofic group with sofic approximation $\Sigma$. Suppose that $\Gamma\supseteq \Lambda,$ where the $L^{2}$-torsion of $\Lambda$ is defined and nonzero. Then Yuzvinski\v\i's addition formula fails for $(\Gamma,\Sigma)$ and the class of finitely presented $\ZZ(\Gamma)$-modules.
\end{cor}

\begin{proof}  This is automatic from Proposition \ref{P:YAF}, Corollary 8.4 in \cite{KerrLi2} and Corollary \ref{C:supergroup}.

\end{proof}

Let us close by proving Proposition \ref{P:torsionintro} from the introduction.

\begin{proposition}\label{P:L2torsioncounterexample} Let $\Gamma$ be a cocompact lattice in $SO(n,1)$ with  $n$ odd. Consider the trivial $\ZZ(\Gamma)$-module $\ZZ.$ Then:

(i): $\Gamma$ is a sofic group,

(ii): the $L^{2}$-torsion of $\ZZ$ as a $\ZZ(\Gamma)$-module is defined,

(iii): for every sofic approximation $\Sigma$ of $\Gamma,$ one has
\[h_{\Sigma}(\TT,\Gamma)\ne \rho^{(2)}(\ZZ,\Gamma),\]
\[h_{\Sigma,m_{\TT}}(\TT,\Gamma)\ne \rho^{(2)}(\ZZ,\Gamma).\]

(iv): If in addition $n$ is congruent to $1$ modulo $4,$ then for any random sofic approximation $\kappa$ of $\Gamma$ we have
\[h_{\kappa}(\TT,\Gamma)\ne \rho^{(2)}(\ZZ,\Gamma),\]
\[h_{\kappa,m_{\TT}}(\TT,\Gamma)\ne \rho^{(2)}(\ZZ,\Gamma).\]
\end{proposition}

\begin{proof} Part $(i)$ is a consequence of the fact that $\Gamma$ is a linear group. Part $(ii)$ follows from \cite{Luck} Theorem 3.152 which shows that in fact $\rho^{(2)}(\ZZ,\Gamma)\ne 0.$ To prove $(iii)$ fix a sofic approximation $\Sigma$ of $\Gamma.$ First note that
\[h_{\Sigma}(\TT,\Gamma)=0\]
when $\Gamma\actson \TT$ trivially.  By the variational principle, we see that
\[h_{\Sigma,m_{\TT}}(\TT,\Gamma)\in\{-\infty,0\}.\]
Since $L^{2}$-torsion, when it is defined, is never $-\infty$ and $\rho^{(2)}(\ZZ,\Gamma)\ne 0,$ we have proved $(iii).$ The proof of $(iv)$ is similar except in this case we note that
\begin{itemize}
\item $\rho^{(2)}(A,\Gamma)>0,$ (again this follows from \cite{Luck} Theorem 3.152),\\
\item $h_{\kappa}(\TT,\Gamma)=0$ (again because $0$ is a fixed point),\\
\item $h_{\kappa,m_{\TT}}(\TT,\Gamma)\leq 0$ (again by the variational principle).
\end{itemize}
This completes the proof.

\end{proof}

\section{Remaining Questions and Conjectures}

Related to our relation between the failure of a Yuzvinski\v\i\ addition formula and possible values of $L^{2}$-torsion we ask the following.

\begin{question} \emph{Does there exists a non-amenable sofic group $\Gamma$ not containing any free subgroups, so that the $L^{2}$-torsion of $\Gamma$ is defined and $\rho^{(2)}(\Gamma)\ne 0?$}\end{question}

Given Corollary \ref{C:supergroup} we conjecture the following.

\begin{conjecture} Let $\Gamma$ be a countable discrete non-amenable sofic group with sofic approximation $\Sigma.$ Then there is an exact sequence
\[\begin{CD}
0 @>>> A @>>> B @>>> C @>>> 0,
\end{CD}\]
of countable $\ZZ(\Gamma)$-modules so that
\[h_{\Sigma}(\widehat{B},\Gamma)\ne h_{\Sigma}(\widehat{A},\Gamma)+h_{\Sigma}(\widehat{C},\Gamma).\]
\end{conjecture}
This conjecture is also an analogue of what is already conjectured for metric mean dimension. For example, given our previous work in \cite{Me4} a failure of additivity of metric mean dimension for finitely generated algebraic actions reduces to the case of sofic groups  Conjecture 6.48 in \cite{Luck} (see Section 7 of \cite{Me4}).

We conjecture the opposite however for Lewis Bowen's $f$-invariant entropy.

\begin{conjecture}\label{C:additivefinvariant} Let $r\in \NN,$ and
\[\begin{CD}
0 @>>> A @>>> B @>>> C @>>> 0
\end{CD}\]
be an exact sequence of $\ZZ(\FF_{r})$-modules. Suppose that the actions of $\FF_{r}$ on $(\widehat{A},m_{\widehat{A}}),$ $(\widehat{B},m_{\widehat{B}})$,$(\widehat{C},m_{\widehat{C}})$ all have finite generating partitions. Then
\[f_{m_{\widehat{B}}}(\widehat{B},\FF_{r})=f_{m_{\widehat{A}}}(\widehat{A},\FF_{r})+f_{m_{\widehat{C}}}(\widehat{C},\FF_{r}).\]
\end{conjecture}
We actually conjecture that this should hold with the $f$-invariant replaced more generally by the random sofic entropy with respect to the random sofic approximation given by $u_{\Hom(\FF_{r},S_{n})}.$ Conjecture \ref{C:additivefinvariant}  has been proven by Bowen-Gutman in \cite{BowenGun} (see Theorem 2.2)  when each of the actions $\Gamma\actson (\widehat{B},m_{\widehat{B}}),$ $\actson (\widehat{A},m_{\widehat{A}})$ $\Gamma\actson (\widehat{C},m_{\widehat{C}})$ has a finite generating partition and one of the following two conditions hold:
\begin{itemize}
\item either $A$ is finitely generated as an abelian group, or
\item there exists a finite, abelian group $D$ so that $C$ embeds into $D(G)$ as a $\ZZ(G)$-module.
\end{itemize}

	Although we have already settled finiteness of topological entropy, we propose the following question for measure-theoretic entropy.

\begin{question} Let $\Gamma$ be a countable discrete sofic group with sofic approximation $\Sigma.$ Let $f\in M_{m,n}(\ZZ(\Gamma))$ and suppose that $f$ is not injective as a left multiplication operator on $\ell^{2}(\Gamma)^{\oplus n}.$ Is it true that
\[h_{\Sigma,m_{X_{f}}}(X_{f},\Gamma)=\infty?\]
\end{question}

	The following is an interesting question.
\begin{question} Let $\Gamma$ be a countable discrete sofic group. For what $\ZZ(\Gamma)$ modules $A$ and sofic approximations $\Sigma$ do we have
\[h_{\Sigma}(\widehat{A},\Gamma)=h_{\Sigma,m_{\widehat{A}}}(\widehat{A},\Gamma)?\]
\end{question}

	We caution the reader that we cannot allow $(A,\Sigma)$ to be arbitrary. For example, if $\Gamma$ is a nonabelian free group, then for every transitive action of $\Gamma$ on a finite set $F$ with at least two elements there is a sofic approximation $\Sigma$ of $\Gamma$ so that $h_{\Sigma,u_{F}}(F,\Gamma)=-\infty.$ See e.g. the proof  Lemma 3.2 in  \cite{SewardGGD}.
%
%

\end{document}